\documentclass[11pt]{amsart}

\usepackage[algoruled,vlined,norelsize]{algorithm2e}

\usepackage[T1]{fontenc}
\usepackage[utf8]{inputenc}
\usepackage[english]{babel}
\usepackage{latexsym,amsfonts,amsmath,amssymb,amsthm,amscd,stmaryrd,float, fullpage, hyperref}

\usepackage{graphicx,tikz}
\usetikzlibrary{patterns}

\makeatletter
\renewcommand\section{\@startsection
	{section}% nom du titre
{1}% niveau de titre
{0pt}% indentation
{-3.5ex plus -1ex minus -.2ex}%espace vertical avant 
{2.3ex plus.2ex}% espace vertical après 
{\centering\normalfont\Large\scshape}}

\renewcommand\subsection{\@startsection
	{subsection}% nom du titre
{2}% niveau de titre
{0pt}% indentation
{-3ex plus -1ex minus -.2ex}%espace vertical avant 
{1ex plus.2ex}% espace vertical après 
{\normalfont\large\bfseries}}

\renewcommand\subsubsection{\@startsection
	{subsubsection}% nom du titre
{3}% niveau de titre
{0pt}% indentation
{-1.5ex plus -1ex minus -.2ex}%espace vertical avant 
{0.8ex plus .2ex}% espace vertical après 
{\normalfont\bfseries}}

\renewcommand\paragraph{\@startsection
	{paragraph}% nom du titre
{4}% niveau de titre
{0em}% indentation
{-1.2ex plus -0.4ex minus -.2ex}%espace vertical avant 
{0ex}% espace vertical après 
{\normalfont\bfseries}}

\makeatother

\newcommand\A{\mathcal A}
\newcommand\B{\mathcal B}
\newcommand\R{\mathbb R}
\newcommand\N{\mathbb N}
\newcommand\Z{\mathbb Z}
\newcommand\Q{\mathbb Q}
\newcommand\U{\mathbb U}
\newcommand\s\sigma
\newcommand\az{\A^\Z}
\newcommand\bz{\B^\Z}
\newcommand\dc{d_C}
\newcommand\per[1]{\mathstrut^\infty#1^\infty}

\newcommand\tx{\text}

\newcommand{\Vb}{\mathbf{V}}

\newcommand\dm{d_\mathcal M}

\newcommand\Ba{\mathfrak{B}}
\newcommand\Merg{\mathcal{M}_{\s-\textrm{erg}}}
\newcommand\Mpsimix{\mathcal{M}_{\s-\textrm{mix}}}

\newcommand\Mcomp{\mathcal{M}_{\s}^{\textrm{comp}}}
\newcommand\Mscomp{\mathcal{M}_{\s}^{\textrm{$\Delta_2$-comp}}}
\newcommand\V{\mathcal{V}}
\newcommand\K{\mathcal{K}}
\newcommand\Ms{\mathcal{M}_{\s}}
\newcommand\Ball{\mathbf{B}}
\newcommand\supp{\textrm{supp}}
\newcommand\F{F_\ast}
\newcommand\Mergfull{\Merg^\textrm{full}}
\newcommand\Mmixfull{\Mpsimix^\textrm{full}}

\newcommand{\Kset}{\mathfrak{K}}

\newcommand\meas[1]{\widehat{\delta_{#1}}}

\theoremstyle{definition}
\newtheorem{definition}{Definition}
\newtheorem*{remark}{Remark}
\newtheorem{example}{Example}

\theoremstyle{plain}

\newtheorem{theorem}{Theorem}
\newtheorem{proposition}{Proposition}
\newtheorem{fact}{Fact}
\newtheorem{open}{Open question}
\newtheorem{corollary}{Corollary}

\newcounter{claimcount}[theorem]
\newcommand{\THMfont}[1]{{\sl #1}}
\newcommand{\Claim}[1]{\refstepcounter{claimcount} \vspace{0.4em}\noindent {\sc Claim \theclaimcount: \ }\THMfont{ #1}}
\newcommand{\bclaimprf}[1][Proof.~]{\begin{list}{}{\setlength{\leftmargin}{0.5em}
\setlength{\rightmargin}{2em}\setlength{\listparindent}{1em}}\item
{\em #1}}

\newcommand{\eclaimprf}{ \hfill $\Diamond$~{\scriptsize {\tt Claim~\theclaimcount}}\end{list}} % %

\usepackage{xcolor}
\usepackage{framed}

\newcommand{\wall}{\hspace{-0.1cm}
\vbox to 9pt{\hbox {
\begin{tikzpicture}[scale=0.38]
\draw [fill=black!20!white] (0,0) rectangle (1,1);
\draw (0.5,0.5) node {{\footnotesize W}};
\end{tikzpicture}
}}\hspace{-0.1cm}}
\newcommand{\start}{\hspace{-0.1cm}
\vbox to 9pt{\hbox{
\begin{tikzpicture}[scale=0.38]
\draw (0,0) rectangle (1,1);
\draw (0.5,0.5) node {{\footnotesize I}};
\end{tikzpicture}
}}
\hspace{-0.1cm}}

\newcommand{\fusion}{\hspace{-0.1cm}
\vbox to 9pt{\hbox{
\begin{tikzpicture}[scale=0.38]
\draw (0,0) rectangle (1,1);
\draw (0.5,0.5) node {{\footnotesize M}};
\end{tikzpicture}
}}
\hspace{-0.1cm}}
\newcommand{\copie}{\hspace{-0.1cm}
\vbox to 9pt{\hbox{
\begin{tikzpicture}[scale=0.38]
\draw (0,0) rectangle (1,1);
\draw (0.5,0.5) node {{\footnotesize C}};
\end{tikzpicture}
}}
\hspace{-0.1cm}}
\newcommand{\fire}{\vbox to 9pt{\hbox{
\begin{tikzpicture}[scale=0.4]
\draw [dashed, fill=gray!20] (0,0) rectangle (1,1);
\draw (0.5,0.5) node{$F$};
\end{tikzpicture}
}}}
\newcommand{\lfire}{\vbox to 9pt{\hbox{
\begin{tikzpicture}[scale=0.23]
\draw [dashed, fill=gray!20] (0,0) rectangle (1,1);
\draw (0.5,0.5) node{{\tiny $F$}};
\end{tikzpicture}
}}}
\newcommand{\order}{\vbox to 9pt{\hbox{
\begin{tikzpicture}[scale=0.4]
\draw [dashed, fill=black!50] (0,0) rectangle (1,1);
\draw (0.5,0.5) node{$O$};
\end{tikzpicture}
}}}

\newcommand{\white}{\vbox to 9pt{\hbox{
\begin{tikzpicture}[scale=0.4]
\draw [dashed,] (0,0) rectangle (1,1);
\draw (0.5,0.5) node{$ $};
\end{tikzpicture}
}}}
\newcommand{\lwhite}{\vbox to 9pt{\hbox{
\begin{tikzpicture}[scale=0.23]
\draw [dashed,] (0,0) rectangle (1,1);
\draw (0.5,0.5) node{$ $};
\end{tikzpicture}
}}}

\newcommand{\block}{\vbox to 9pt{\hbox{
\begin{tikzpicture}[scale=0.4]
\draw [dashed, fill=black!50] (0,0) rectangle (1,1);
\draw (0.5,0.5) node{$ $};
\end{tikzpicture}
}}}

\renewcommand\alph[1]{\A_\texttt{#1}}

\newcommand\define\textbf

\newcommand{\card}{\textrm{Card}}
\newcommand{\freq}{\textrm{Freq}}

\title[Characterisation of sets of limit measures of a CA]{Characterisation of sets of limit measures of a cellular
automaton iterated on a random configuration}

\author{Benjamin Hellouin de Menibus}
\address{Aix Marseille Université, CNRS, Centrale Marseille, I2M UMR 7373, 13453, Marseille, France}
\email{benjamin.hellouin@gmail.com}
\urladdr{http://mat-unab.cl/$\sim$hellouin/}

\author{Mathieu Sablik}
\address{Aix Marseille Université, CNRS, Centrale Marseille, I2M UMR 7373, 13453, Marseille, France}
\email{mathieu.sablik@univ-amu.fr}
\urladdr{http://www.i2m.univ-amu.fr/$\sim$sablik/}
\subjclass[2010]{37B10,37B15,03D10,03D80}
\keywords{Symbolic Dynamics, Cellular automata, SRB measures,Turing machines}

\date{}
\begin{document}
\begin{abstract}The asymptotic behaviour of a cellular automaton iterated on a random configuration is well described by
its limit probability measure(s). In this paper, we characterise measures and sets of measures that can be reached as
limit points after iterating a cellular automaton on a simple initial measure. 
In addition to classical topological constraints, we exhibit necessary computational obstructions. With an additional
hypothesis of connectivity, we show these computability conditions are sufficient by constructing a cellular automaton
realising these sets, using auxiliary states in order to perform computations. Adapting this construction, we obtain a
similar characterisation for the Ces\`aro mean convergence, a Rice theorem on the sets of limit points, and we are able
to perform computation on the set of measures, i.e. the cellular automaton converges towards a set of limit points that
depends on the initial measure. Last, under non-surjective hypotheses, it is possible to remove auxiliary states from
the construction.
\end{abstract}

\maketitle

\section*{Introduction}
A cellular automaton is a complex system defined by a local rule which acts synchronously and uniformly on the
configuration space $\az$, where $\A$ is a finite alphabet. These simple models have a wide variety of different
dynamical behaviours. We are interested in the typical asymptotic behaviour starting from a random configuration, as
this is usually done in simulations; different approaches stemmed from such observations. It is well-described by taking
the iterated image of the initial measure under the action of the cellular automaton, and considering the limit points
of this sequence in the weak$^\ast$ topology. 

It is natural to ask which sets of measures can be obtained as limit points in this way. Obviously, any measure can be
reached by iterating the identity on itself. Therefore, a more interesting approach is to start from some simple measure
such as the uniform Bernoulli measure. In some sense, this is similar to SRB measures which are “physically” relevant
invariant measures obtained when starting from the Lebesgue measure in continuous dynamical systems~\cite{Young-2002}.

Formally speaking, given a simple initial measure $\mu$, we want to characterise all reachable $\V(F,\mu)$, the sets
of accumulation points of the sequence $(\F^t\mu)_{t\in\N}$ of the images of $\mu$ under the iterated action of $F$,
and $\V'(F,\mu)$, the sets of accumulation points of $\left(\frac{1}{t+1}\sum_{i=0}^t\F^i\mu\right)_{t\in\N}$, the
Ces\`aro mean of the previous sequence, for all possible cellular automata $F$. 

Previous works focused on the $\mu$-limit set, which corresponds to the union of the support of the limit
measures~\cite{Kurka-Maass-2000,Kurka-2005}. 
Very complex $\mu$-limit sets can be constructed~\cite{Boyer-Poupet-Theyssier-2006,Boyer-Delacourt-Sablik-2010}, and our
construction is partly inspired from these works.\bigskip

Describing limit measures has been done for only few concrete nontrivial examples. There are essentially two types of
convergence quite well understood:
\begin{itemize}
\item convergence towards a simple measure: for example, the cyclic cellular automaton on three states introduced
in~\cite{Fisch-1990}, starting from a Bernoulli measure, converges towards a linear combination of Dirac measures
supported by uniform configurations~\cite{Hellouin-Sablik-2011}; 
\item randomisation phenomenon for linear cellular automata: the Ces\`aro mean sequence of the iteration of a linear
cellular automaton on a initial measure converges to the uniform Bernoulli measure as soon as the initial measure is in
a large class which contains Markov measures~\cite{Lind84, Ferrari-Maass-Martinez-Ney-2000, Maass-Martinez-1998,
Pivato-Yassawi-2002}.
\end{itemize}

For any cellular automaton, starting from a Bernoulli measure or a Markov measure, we obtain after a finite number of
steps a hidden Markov chain which is well understood~\cite{Boyle-Petersen-2011}. If we consider a computable initial
measure $\mu$ (which means that there is an algorithm that approximates at a known rate the probability that a word
$u\in\A^*$ appears), then it is easy to see that $\F^t\mu$ is also computable. For example, a Bernoulli or Markov
measure is computable iff its parameters are computable real numbers.

The limit measure(s) are not necessarily computable since the speed of convergence is not known. Nevertheless, we show
in Section~\ref{section:ComputableObstruction} that there exists necessary computational obstructions. The main problem
is to prove the reciprocal, in other words: given a set of measures satisfying the computational obstructions, construct
a cellular automaton which, starting on any simple initial measure, reaches exactly this set asymptotically.
Similar computational obstructions appear when characterising possible topological dynamics properties of subshifts of
finite type or cellular automata: possible entropies~\cite{Hochman-Meyerovitch-2010}, possible growth-type
invariants~\cite{Meyerovitch-2011}, possible sub-actions~\cite{Hochman-2009,Aubrun-Sablik-2010}... 
However, the construction is quite different here since starting from a random configuration requires to self-organise 
the space, in the same spirit as the probabilistic cellular automaton of~\cite{Gacs-2001} which corrects the random
perturbations. 

In Section~\ref{section:Construction}, we construct a cellular automaton $F$ such that, starting from any shift-mixing
probability measure $\mu$ with full support, the limit points of the sequence of measures $(\F^t\mu)_{t\in\N}$ are
described as the accumulation points of a computable polygonal path of measures supported by periodic orbits. 
First of all the cellular automaton divides the initial configuration in segments and formats each segment using a
method similar to the one developed in~\cite{Delacourt-Poupet-Sablik-Theyssier-2011}. Computation takes place in a
negligible part of each segment and the result is copied periodically on the rest of the segment. In order to have an
arbitrarily large area of computation, segments are merged progressively in a controlled manner. The difficulty of the
construction is to synchronise all the operations to ensure the convergence.

In the Section~\ref{section:NoAux}, we modify this construction so that we do not need auxiliary states, i.e., the
cellular automaton only uses the same alphabet as the limit measure(s). This is only possible,
however, at the price of some additional hypotheses on the support of the measures. 

In Section~\ref{section:RelatedProblem} we use these constructions to answer various questions, along with some open
questions. The results are, for a fixed measure $\mu$ in a large class of measures:
\begin{itemize}
 \item characterisation of shift-invariant measures $\nu$ such that there exists a cellular automaton $F$ which
verifies $\F^t\mu\underset{t\to\infty}{\longrightarrow}\nu$ (Corollary~\ref{cor:ConvUniqueMeasure});
 \item characterisation of connected subsets of shift-invariant measures $\K$ such that there exists a cellular
automaton $F$ which verifies $\V(F,\mu)=\K$ (Corollary~\ref{cor:ConvSetMeasure});
 \item characterisation of subsets of shift-invariant measures $\K'$ such that there exists a cellular automaton $F$
which verifies $\V'(F,\mu)=\K'$ (Corollary~\ref{cor:ConvCesàroMeasure});
 \item Rice theorem for shift-invariant measures and connected subsets of shift-invariant measures reached by a cellular
automaton (Corollary~\ref{cor:Rice}).
\item counterparts of all these results without auxiliary states, using the modified construction
(Corollaries \ref{cor:ConvSetMeasureNoAux}, \ref{cor:ConvCesaroMeasureNoAux} and \ref{cor:RiceNoAux}).
\end{itemize}

In Section~\ref{section:Oracle}, we consider the case where the set of limit points depends on the initial measure.
Computational constraints appear to describe functions $\mu\longmapsto\V(F,\mu)$ that can be realised in this way.
Indeed, it is possible to ``transfer'' the computational complexity of the initial measure (using it as an oracle) to
the set of limit points. Modifying the construction of Section~\ref{section:Construction}, we manage to build 
a set of limit points depending on the density of a special state; however, we do not obtain a complete
characterisation.\bigskip

\section{Definitions}

\subsection{Configuration space and cellular automata} 

Let $\A$ be a finite alphabet. Consider $\az$ the \define{space of configurations} which are $\Z$-indexed sequences in
$\A$. If $\A$ is
endowed with the discrete topology, $\az$ is compact, perfect and totally disconnected in the product topology.
Moreover one can define a metric on $\az$ compatible with this topology:
\[\forall x,y\in\az,\quad \dc(x,y)=2^{-\min\{|i|\ :\ x_i\ne y_i, \ i\in\Z \}}.\]

Let $\U\subset\Z$. For $x\in\az$, denote $x_{\U}\in\A^{\U}$ the restriction of $x$ to $\U$. Given a pattern
$w\in\A^{\U}$, one defines the cylinder $[w]_{\U}=\{x\in\az : x_{\U}=w\}$. Denote $\A^{\ast} = \bigcup_n\A^n$ the set of
all \define{finite words} $w=w_0\dots w_{n-1}$; $|w|=n$ is the \define{length} of $w$. Also denote
$[w]_i=[w]_{[i,i+|w|-1]}$ and $[w]=[w]_0=[w]_{[0,|w|-1]}$, and $\A^{\leq k} = \bigcup_{n\leq k}\A^n$. For any
$u\in\A^\ast$ such that $|u|\leq |w|$, define the \define{frequency} of $u$ in $w$ as
$\freq(u,w)=\frac{1}{|w|-|u|+1}\card\left(\left\{i\in[0,|w|-|u|]:w_{[i,i+|u|]}=u\right\}\right)$.

The \define{shift} map $\s:\az\to\az$ is defined by $\s(x)_i=x_{i+1}$ for $x=(x_m)_{m\in\Z}\in\az$ and $i\in\Z$ and a \define{subshift} is a closed $\s$-invariant subset of $\az$. For $w\in\A^{\ast}$, $\per{w}$ is the \define{$\s$-periodic word} defined by
$\per{w}_{[0,|w|-1]}=w$ and $\s^{i+|w|}(\per{w})=\s^i(\per{w})$ for all $i\in\Z$.

A \define{cellular automaton} (CA) is a pair $(\az,F)$ where $F:\az\to\az$ is  a continuous function that commutes with
the shift ($\s\circ F=F\circ\s$). By the Curtis–Hedlund–Lyndon theorem \cite{Hedlund}, it is equivalent to a function
defined by $F(x)_i=\overline{F}((x_{i+u})_{u\in\U_F})$ for all $x\in\az$ and $i\in\Z$, where
$\U_F\subset\Z$ is a finite set named \define{neighbourhood} and $\overline{F}:\A^{\U_F}\rightarrow\A$ is a \define{local
rule}.

\subsection{\texorpdfstring{Sets of measures on $\az$}{Sets of measures on AZ}}

\subsubsection{Dynamical properties}\label{section:dynamical}
 Let $\Ba$ be the Borel sigma-algebra of $\az$. Denote by $\mathcal{M}(\az)$ the set of probability measures on $\az$
defined on the sigma-algebra $\Ba$. Let $\Ms(\az)$ be the \define{$\s$-invariant probability measures} on $\az$, that is
to say the measures $\mu$ such that $\mu(\s^{-1}(B))=\mu(B)$ for all $B\in\Ba$. Cylinders corresponding to finite words
form a base of the topology, so a measure $\mu\in\Ms(\az)$ is entirely characterised by $\{\mu([u])\ :\ u\in\A^\ast\}$.

Usually $\Ms(\az)$ is endowed with the \define{weak$^\ast$ topology}: a sequence $(\mu_n)_{n\in\N}$ in $\Ms(\az)$
converges to $\mu\in\Ms(\az)$ if and only if, for all finite subsets $\U\subset\Z$ and for all patterns $u\in\A^{\U}$,
one has $\lim_{n\to\infty}\mu_n([u]_{\U})=\mu([u]_{\U})$. In the weak$^{\ast}$ topology, the
set $\Ms(\az)$ is compact and metrisable. A metric is defined by
\[\dm(\mu,\nu)=\sum_{n\in\N}\frac{1}{2^n}\max_{u\in\A^n}|\mu([u])-\nu([u])|.\]

Define the \define{ball} centered on $\mu\in\Ms(\az)$ of radius $\varepsilon >0$ as
\[\Ball(\mu,\varepsilon)=\left\{\nu\in\Ms(\az) : \dm(\mu,\nu)\leq\varepsilon \right\},\]
and this definition is extended to balls around a set of measures. A measure $\mu\in\Ms(\az)$ is \define{$\s$-ergodic} if for every $\s$-invariant $B\in\Ba$ ($\s^{-1}(B)=B$
$\mu$-almost everywhere), one has $\mu(B)=0$ or $1$. The set of $\s$-ergodic probability measures is denoted by
$\Merg(\az)$. 

For $\U\subset\Z$ not necessarily finite, denote by $\Ba_{\U}$ the $\s$-algebra generated by the set
$\{[u]_{\U}:u\in\A^{\U'}, \U'\underset{\textrm{finite}}{\subset}\U\}$.
Define the \define{weak mixing coefficients} of a measure $\mu\in\Ms(\az)$ as
\[\psi_{\mu}(n)=\sup\left\{\left|\frac{\mu(A\cap B)}{\mu(A)\mu(B)}-1\right| : A\in \Ba_{]-\infty,0]},\,
B\in\Ba_{[n,\infty[},\, \mu(A)>0,\, \mu(B)>0\right\}.\]

A measure $\mu\in\Ms(\az)$ is $\mathbb{\psi}$\define{-mixing} if $\psi_{\mu}(n)\underset{n\to\infty}{\longrightarrow}0$.
Denote $\Mpsimix(\az)$ the set of $\psi$-mixing measures. In particular $\Mpsimix(\az)\subset\Merg(\az)$.

For a measure $\mu\in\Ms(\az)$, define the \define{support} of $\mu$ $\supp(\mu)$ as the set of configurations of
$\az$ such that any open neighbourhood of these points have positive measure. $\mu$ has \define{full support} if
$\supp(\mu) = \az$, which is equivalent to $\mu([u])>0$ for all $u\in\A^{\ast}$. Denote $\Mergfull(\az)$ the set of
ergodic measures with full support, and $\Mmixfull(\az)$ the set of $\psi$-mixing measures with full support.

\subsubsection{Classical examples} 

The \define{Dirac measure} supported by $x\in\az$ is defined as $\delta_x(A)= \mathbf 1_{x\in A}$. Generally $\delta_x$
is not $\s$-invariant. However, for any $\s$-periodic configuration $^{\infty}w^{\infty}$, it is possible to define
the \define{$\s$-invariant measure supported by $^{\infty}w^{\infty}$} by taking the mean of the Dirac measures
on the orbit under $\sigma$:
\[\meas{w}=\frac1{|w|}\sum_{i\in[0,|w|-1]}\delta_{\s^i(\mathstrut^\infty w^\infty)}.\]

The set of measures $\left\{\meas{w}:w\in\A^{\ast}\right\}$ is dense in $\Ms(\az)$~\cite{Peteren-1983-livre}.

Given $(p_a)_{a\in\A}$ a family of elements of $[0,1]$ such that $\sum_{a\in\A}p_a=1$. The \define{Bernoulli} measure $\lambda_{(p_a)}$ associated to $(p_a)_{a\in\A}$ is defined by \[\lambda_{(p_a)}([u])=p_{u_1}p_{u_2}\dots p_{u_n}\textrm{ for all } u=u_1u_2\dots u_n\in\A^n.\]

\subsection{Action of a cellular automaton on $\Ms(\az)$ and limit points} 

\subsubsection{Definition of $\V(F,\mu)$ and $\V’(F,\mu)$}

Let $(\az,F)$ be a cellular automaton and $\mu\in\Ms(\az)$. Define the \define{image measure} $\F\mu$ by
$\F\mu(A)=\mu(F^{-1}(A))$ for all $A\in\Ba$.
Since $F$ is $\s$-invariant, that is to say $F\circ\s=\s\circ F$, one deduces that $\F(\Ms(\az)) \subset \Ms(\az)$ and
$\F(\Merg(\az))\subset\Merg(\az)$. This
defines a continuous application $\F:\Ms(\az)\to\Ms(\az)$.  

We consider in particular $\F^t\mu$ the iterated image of $\mu$ by $\F$, and its \define{Cesàro mean} at
time $t\in\N$ defined by $\varphi_t^F(\mu)=\frac{1}{t+1}\sum_{i=0}^t\F^i\mu\in\Ms(\az)$.

For a measure $\mu\in\Ms(\az)$, we are interested in the asymptotic behaviour of the sequences $(\F^t\mu)_{t\in\N}$ and
$(\varphi_t^F\mu)_{t\in\N}$. Define
\begin{itemize} 
\item $\V(F,\mu)$, the \define{$\mu$-limit measures set}, as the set of \textbf{limit points} of the sequence $(F^t\mu)_{t\in\N}$;
\item $\V'(F,\mu)$, the \define{Cesàro mean $\mu$-limit measures set}, as the set of limit points of the sequence
$(\varphi_t^F\mu)_{t\in\N}$. 
\end{itemize}
Since $\Ms(\az)$ is compact, $\V(F,\mu)$ and $\V'(F,\mu)$ are nonempty. When
$\V(F,\mu)$ is a singleton $\{\nu\}$, $\F^t\mu([u])\underset{t\to\infty}{\longrightarrow}\nu([u])$. \bigskip

\subsubsection{Topological obstructions}

Our main purpose is to characterise which sets of measures can be realised in this way.
There are topological obstructions for these sets: $\V(F,\mu)$ and $\V'(F,\mu)$ are closed and thus compact since $\Ms(\az)$ is it.

Moreover $\V'(F,\mu)$ is connected. Indeed, let $\mathcal K$ be a connected component of $\V'(F,\mu)$ and assume that $\V'(F,\mu)$ is not reduced to $\mathcal K$. Since $\mathcal K$ and $\V'(F,\mu)\setminus \mathcal K$ are closed in a compact set, there exists a minimum distance $d>0$ between them. Since $\dm(\varphi_t^F(\mu),\varphi_{t+1}^F(\mu))\underset{t\to\infty}{\longrightarrow}0$, the set $\Ball(\K, \frac {2d}3) \backslash \Ball(\K, \frac d3)$ contains arbitrarily many points in the sequence and thus contains a limit point, this is a contradiction.

%%%%

\section{Computability obstructions}\label{section:ComputableObstruction}

In this section, we explore the computability obstructions of $\V(F,\mu)$ and $\V'(F,\mu)$ when the initial measure $\mu$ is computable.

\subsection{Notions of computability}\label{section:TuringMachine}

\begin{definition}A \define{Turing machine} $\mathcal{TM}=(Q,\Gamma,\#,q_0,\delta,Q_F)$ is defined by:
\begin{itemize}
\item $\Gamma$ a finite alphabet and a blank symbol $\#\notin\Gamma$;
\item $Q$ the finite set of states of the head; $q_0\in Q$ is the initial state;
\item $\delta : Q\times(\Gamma\cup\{\#\})\to Q\times(\Gamma\cup\{\#\})\times\{\leftarrow,\rightarrow\}$ the transition function. 
\item $Q_F\subset Q$ the set of final states. 
\end{itemize}
Initially, a one-sided infinite memory tape is filled with $\#$, except for a finite prefix (the input), and a computing head in state $q_0$ is located on the first letter of the tape. At each time step, given the state of the head and the letter it reads on the tape --- depending on its position --- the head changes state, replaces the letter and moves by one cell at most, according to the transition function. When a final state is
reached, the computation stops and the output is the value currently written on the tape.

A function $f : \A^\ast\to \A^\ast$ is \define{computable} if there exists a Turing machine working on an alphabet
$\Gamma\supset\A$ that, on any input $w\in\A^\ast$, eventually stops and outputs $f(w)$.
\end{definition}

To generalise this definition to functions mapping arbitrary countable sets $X \to Y$, we introduce the notion of
\define{encoding.}

An encoding for a countable set $X$ is the
choice of a finite alphabet $\Gamma_X$, a subset $V_X\subset \Gamma_X^\ast$ of \define{valid encodings} and a surjection $e_X :
V_X\twoheadrightarrow X$. Intuitively, a word in $\Gamma_X^\ast$ represents an
element of $X$, but an element can have several valid encodings and not all words of $\Gamma^\ast$ are a valid encoding
of an element of $X$. Strictly speaking, the following definitions depend on the chosen encoding, but empirically all
reasonable choices lead to the same notion of computability. Therefore we fix canonical
encodings that will be valid throughout this article for $\N$, $\Q$ and their products.
\begin{itemize}
 \item $\N$ or $\Z$: $\Gamma = \{0,1\}$ and to each binary sequence we assign the
corresponding integer, with the first bit encoding the sign for $\Z$;
 \item $X\times Y$: If $\Gamma_X, e_X$ and $\Gamma_Y, e_Y$ are the encodings
fixed for $X$ and $Y$, respectively, we put $\Gamma = \Gamma_X\cup\Gamma_Y
\sqcup \{|\}$ (disjoint union, i.e. assume $|$ is a fresh symbol),
 and to $x|y$ we assign $(e_X(x),e_Y(y))$;
 \item $\Q$: Take the encoding for $\Z\times\N^{\ast}$ and compose the encoding function by $(p,q)\longmapsto \frac pq$
(surjection).
\end{itemize}

\begin{definition}
A function $f : X\to Y$ between two countable sets is \define{computable}  if there exists a Turing machine
working on an alphabet containing $\Gamma_X \cup \Gamma_Y$ that, on any input
$x\in V_X$, stops and outputs $y\in V_Y$ such that $f(e_X(x)) = e_Y(y)$. 

A set $K\subset X$ is \define{computable} if $\mathbf{1}_K$ is computable.
\end{definition}

\subsection{Measures and computability}\label{section:ComputableMeasure}

\subsubsection{Definitions and some examples}

Since a probability measure $\mu\in\Ms(\az)$ is characterised by the value of $\mu([u])$ for all words $u\in\A^\ast$, it
can be seen as a particular type of function $\A^\ast\to\R$. Here $\R$ is not countable but is a metric space with a
countable dense set. Therefore it is natural to define a computable function $\A^\ast\to\R$ as a function that can be
approximated by functions $\A^\ast\to\Q$ in a computable manner.

\begin{definition}[Computability of probability measures]
A measure $\mu\in\Ms(\az)$ is \define{computable} (or \define{$\Delta_1$-computable}) iff there exists $f:\A^{\ast}\times\N\to\Q$ computable such that
\[|\mu([u])-f(u,n)|<2^{-n} \qquad \textrm{for
all }u\in\A^{\ast} \textrm{ and }n\in\N.\]  

A sequence of measures $(\mu_i)_{i\in\N}$ is \define{a uniformly computable sequence of computable measures} iff there exists $f:\A^{\ast}\times\N\times \N\to\Q$ computable such that
$|\mu_i([u])-f(u,n,i)|<2^{-n}.$ This is a stronger statement than saying that all $\mu_i$ are computable. 

A measure $\mu\in\Ms(\az)$ is \define{limit-computable} (or \define{$\Delta_2$-computable}) iff there exists a uniformly computable sequence of computable measures
$(\mu_i)_{i\in\N}$ such that
$\lim_{i\to\infty}\mu_i=\mu$.
 Equivalently there exists $f:\A^{\ast}\times\N\to\Q$ computable such that
\[|\mu([u])-f(u,n)|\underset{n\to\infty}{\longrightarrow}0\qquad \textrm{for all
}u\in\A^{\ast}.\]
\end{definition}

Denote $\Mcomp(\az)$ the set of computable measures and $\Mscomp(\az)$ the set of limit-computable measures. Of course
$\Mcomp(\az)\subset\Mscomp(\az)$. 
\begin{example}
Let us provide some examples of such measures:
\begin{itemize}
\item any measure supported by a periodic orbit is computable;
\item any Bernoulli measure or Markov measure with computable\footnote{Here the computability of a real is defined as the computability of the function that maps $n$ to the
$n^{\textrm{th}}$ digit of the real.} (resp. limit-computable) parameters is computable (resp. limit-computable);
\item if an effective subshift (i.e. such the set of forbidden patterns can be enumerated by a Turing machine) has a unique $\s$-ergodic measure $\mu$, then $\mu$ is computable. 

To obtain an approximation of $\mu([u])$ for a pattern $u$, we construct an algorithm as follows. At step $n$, the algorithm enumerates the first $n$ forbidden patterns and produces all words of size $n$ which does not contain any of these forbidden patterns. If the frequencies of $u$ in all these words are sufficiently close (less than the requested precision), the average frequency is an approximation of $\mu([u])$. If not, we continue to the step $n+1$, until an approximation is found. This process must stop since, for all elements of an unique ergodic subshift, the frequency of a pattern converges towards the measure of this pattern (we would otherwise obtain another $\s$-invariant measure). We deduce that $\mu$ is a computable measure. This proof can be found in~\cite{Galatolo-Hoyrup-Rojas-2011} with a more abstract point of view.

This class of measures is very large. For example, this is the case for any subshift obtained by a primitive substitution or the orbit of a Sturmian word defined by some computable slope. For more details about measures of primitive substitutions and Sturmian words, see for example~\cite{CANT}.
\end{itemize}
\end{example}

\subsubsection{Approximation by measures supported by periodic orbits}

It is well known that the set of measures supported by periodic orbits is dense in $\Ms(\az)$ for the weak topology. The notions of computability for probability measures can be defined in an equivalent way using this countable dense set and the distance $\dm$ which measures the approximation.

\begin{proposition}[Approximation by measures supported by periodic orbits]\label{prop:ApproximComputableMeasure}
It is possible to define computable and limit-computable measure as approximation of measures supported by periodic orbits.

\begin{itemize}
\item[(i)] A measure $\mu\in\Ms(\az)$ is computable if and only if there exists a computable function $f:\N\to\A^{\ast}$  such that
$\dm\left(\mu,\meas{f(n)}\right)\leq 2^{-n}$
for all $n\in\N$.
\item[(ii)] A measure $\mu\in\Ms(\az)$ is limit-computable if and only if there exists a computable function $f:\N\to\A^{\ast}$ such that
$\underset{n\to\infty}{\lim}\meas{f(n)}=\mu$.
\end{itemize}
\end{proposition}
\begin{proof}(i)
Let $\mu\in\Mcomp(\az)$. Given some $n\in\N$, we can enumerate words in $\A^\ast$ in a computable manner until we find a
word $f(n)$ such that $|\mu([u])-\meas{f(n)}([u])|<2^{-n-2}$ for all $u\in\A^{\leq n+1}$. Such a word exists since the
set $\left\{\meas{w}:w\in\A^{\ast}\right\}$ is dense in $\Ms(\az)$, and it is eventually found since $\mu$ and
$\meas{w}$ are computable. One
has \[\dm(\mu,\meas{f(n)})=\sum_{i\in\N}\frac{1}{2^i}\max_{u\in\A^i}|\mu([u])-\meas{f(n)}([u])|\leq
\frac{1}{2^{n+1}}+\sum_{i\geq n+2}\frac{1}{2^i}\leq\frac{1}{2^n}.\]
The converse is obvious since $(n,u) \longmapsto \meas{f(n)}([u])$ is computable and we have \[\forall
u\in\A^\ast, \left|\mu([u]) - \meas{f(n)}([u])\right|\leq 2^{|u|}\dm\left(\mu,\meas{f(n)}\right).\]

(ii) Let $\mu\in\Mscomp(\az)$. There exists a uniformly computable sequence of computable measures $(\mu_i)_{i\in\N}$ such that
$\lim_{i\to\infty}\mu_i=\mu$. For each $n\in\N$,
we enumerate words until we find $f(n)\in\A^{\ast}$ such that $\dm(\mu_n,\meas{f(n)})\leq 2^{-n}$. Clearly
$f:\N\to\A^{\ast}$ is computable and we have $\dm(\mu,\meas{f(n)})\leq
\dm(\mu,\mu_n)+\dm(\mu_n,\meas{f(n)})\underset{n\to\infty}{\longrightarrow}0$.
The converse is similar to the previous case.
\end{proof}

\subsection{Action of a cellular automaton on computable measures}

When a computable measure is iterated by a cellular automaton, the resulting sequence of measures is uniformly computable. This is formalised in the following property.

\begin{proposition}[Uniform computability]\label{prop:UniformComp}
Let $F:\az\to\az$ be a cellular automaton. If $\mu\in\Mcomp(\az)$, then $(\F^t\mu)_{t\in\N}$ is an
uniformly computable sequence of computable measures. \end{proposition}
\begin{proof}
By definition, there is a computable function $f:\A^{\ast}\times\N\to\Q$ such that $|\mu([u])-f(u,n)|\leq2^{-n}$ for all
$u\in\A^\ast$. Because $F$ is defined locally, $F^t(x)_{[0,k]}$ depends
only on $x_{[lt,rt+k]}$ where $l = \min \U_F$ and $r = \max \U_F$. In other words, for all $u\in\A^k$, there is a set $\mathbf{Pred}_t(u)\subset\A^{[lt,rt+k]}$
such that $F^{-t}([u])=\cup_{v\in\mathbf{Pred}_t(u)}[v]$. Now consider the function \[f':(u,n,t)\longmapsto
\sum_{v\in\mathbf{Pred}_t(u)}f(v,|u|+n+(r-l)t).\] It is computable by enumerating elements of $\A^{k+(r-l)t}$ and
checking if $F^t([v]_{-lt})\subset [u]$ by iterating the local rule on $v$. Finally:
\begin{align*}
|\F^t\mu([u])-f'(u,n,t)|&=\left|\mu\left(\bigcup_{v\in\mathbf{Pred}_t(u)}[v] \right) -
\sum_{v\in\mathbf{Pred}_t(u)}f(v,|u|+n+(r-l)t)\right| \\
&\leq \sum_{v\in\mathbf{Pred}_t(u)} \left|\mu([v]) - f(v,|u|+n+(r-l)t)\right|\\
&\leq 2^{|u|+(r-l)t}\cdot 2^{-|u|-n-(r-l)t} = 2^{-n}
\end{align*}
which means that $(\F^t\mu)_{t\in\N}$ is a uniformly computable sequence of computable measures.
\end{proof}

From this proposition, we deduce the first computability obstruction when the sequence $(\F^t\mu)_{t\in\N}$ converges toward a single limit starting from a computable measure $\mu$.

\begin{proposition}[First computability obstruction]\label{prop:OneStep}
Let $F:\az\to\az$ be a cellular automaton. If $\mu\in\Mcomp(\az)$ and $\F^t\mu\underset{t\to\infty}{\longrightarrow}\nu$ then $\nu\in\Mscomp(\az)$.
\end{proposition}

We have obtained a computability obstruction on single limit measures. In the following subsections, we extend this obstruction to sets of limit points.

\subsection{Closed sets in computable analysis}\label{section:ComputableAnalysis}

\subsubsection{Definitions and some examples}

We introduce computability notions on compact subsets of metric spaces which cannot be defined using the characteristic 
function since the compact set is not necessary countable. A standard reference book of the theory of computable analysis 
on metric spaces is~\cite{Weihrauch-2000}, but this theory is widely applied in the context of invariant measures 
(see for example \cite{Galatolo-Hoyrup-Rojas-2011}). Computability in a general metric space is defined according 
to a countable dense subset and a metric, $(\meas{w})_{w\in\A^\ast}$ and $\dm$ in the case of $\Ms(\az)$.

\begin{definition}
 A closed set $\K\subset \Ms(\az)$ is \define{computable} if the set
 \[\left\{(w,r)\in\A^\ast\times \Q\ :\ \overline{\Ball(\meas{w}, r)} \cap \K \neq \emptyset \right\}\]
is computable, that is, if its characteristic function is.
\end{definition}

However, a set of limit points of a sequence $(\F^t\mu)_{t\in\N}$, for $\mu\in\Mcomp(\az)$, is not necessarily
computable. We need to extend these definitions further in order to obtain an arithmetic
hierarchy.

\begin{definition}[$\Sigma_2$ and $\Pi_2$-computable function]
Let $X,Y$ be two countable sets, with $Y$ being ordered. A sequence of functions $(f_i:X\to Y)_{i\in\N}$ (resp. $(f_{i,j}:X\to Y)_{i,j\in\N^2}$) is 
\define{uniformly computable} if $(i,x)\longmapsto f_i(x)$ (resp. $(i,j,x)\longmapsto f_{i,j}(x)$) is computable.

A function $f:X\to Y$ is \textbf{$\Pi_1$-computable} (resp. \textbf{$\Sigma_1$-computable}) if $f =
\inf_{i\in\N}f_{i}$ (resp. $f = \sup_{i\in\N}f_{i}$), where $(f_{i})_{i\in\N}$ is
an uniformly computable sequence of functions.

A function $f:X\to Y$ is \textbf{$\Pi_2$-computable} (resp. \textbf{$\Sigma_2$-computable}) if $f =
\inf_{i\in\N}\sup_{j\in\N}f_{i,j}$ (resp. $f = \sup_{i\in\N}\inf_{j\in\N}f_{i,j}$), where $(f_{i,j})_{(i,j)\in\N^2}$ is
an uniformly computable sequence of functions.
\end{definition}

\begin{definition}[$\Sigma_2$ and $\Pi_2$-computable closed set]
 A closed set $\K\subset \Ms(\az)$ is \define{$\Pi_2$-computable} (resp. \define{$\Sigma_2$-computable}) if the set
 \[\left\{(w,r)\in\A^\ast\times \Q\ :\ \overline{\Ball(\meas{w}, r)} \cap \K \neq \emptyset \right\}\]
is $\Pi_2$-computable (resp. $\Sigma_2$), that is, if its characteristic function is.
\end{definition}

\begin{remark}
The symmetric notions of $\Pi_2$- and $\Sigma_2$-computability come from an analogy with the real arithmetic
hierarchy~\cite{Zheng-Weihrauch-2001,Ziegler-2005}.
These definitions extend naturally to $\Pi_n$- and $\Sigma_n$-computability.
\end{remark}

\begin{example}~

\begin{itemize}
\item $\Ms(\az)$ is a computable set.
\item the set of shift-invariant measures supported by any effective subshift is a $\Pi_1$-computable
compact set (that is to say the set $\left\{(w,r)\in\A^\ast\times \Q\ :\ \overline{\Ball(\meas{w}, r)} \cap \K \neq \emptyset \right\}$ is  $\Pi_1$-computable);
\item let $K \subset[0,1]$ be a closed $\Pi_n$-computable set\footnote{The computability of a closed set of
real numbers is defined similarly to the computability of a closed set of probability measures.}
and denote $\lambda_p\in\Ms(\{0,1\}^\Z)$ the Bernoulli measure which charges $0$ with the probability $p$ and $1$ with the probability $1-p$. The set $\K=\{\lambda_p:p\in K\}$ is a $\Pi_n$-computable
compact set of $\Ms(\{0,1\}^\Z)$ and is connected if and only if $K$ is. Furthermore $\{\alpha\lambda_p+(1-\alpha)
\lambda_q:p,q\in K\textrm{ and }\alpha\in[0,1]\}$ is a $\Pi_n$-computable compact connected set of
$\{0,1\}^\Z$. This example extends naturally to larger alphabets and Markov measures;
\item denote $\mu_\alpha\in\Ms(\{0,1\}^\Z)$ the measure supported by the Sturmian subshift of slope $\alpha$. The set
$\K=\{\mu_\alpha:\alpha\in K\}$, where $K$ is a $\Pi_n$-computable closed subset of $[0,1]$, is a
$\Pi_n$-computable compact set of $\Ms(\{0,1\}^\Z)$ and is connected if and only if $K$ is.
\end{itemize}
\end{example}

\begin{proof}[Proof of the second example]
Let $\Sigma\subset\az$ be an effective subshift, which means that it is defined by a set of forbidden patterns $\mathcal F$ and 
there exists a computable function $f : i,u\longmapsto\{0,1\}$ such that 
$u\in \mathcal F\Longleftrightarrow \sup_if(i,u)=1$. Denote $\mathcal F_i = \{u\in\A^\ast\ :\ \sup_{j\leq i} f(j,u)=1\}$.
 
There exists an effective sequence of integers $(\alpha_i)_{i\in\N}$ such that: \[\forall \mu\in\Ms(\az), \exists w\in\A^{\leq \alpha_i},\ \dm(\mu,\meas{w})\leq \frac{1}{i}.\] This is due from the fact that $\Ms(\az)$ is a recursively precompact metric space (see~\cite{Galatolo-Hoyrup-Rojas-2011}).

%The method to exhibit the sequence $(\alpha_i)_{i\in\N}$ is to take $n$ such that $\frac{1}{2^{n-1}}<\frac {1}{2i}$ and consider the smallest $k$ such that for all $w\in\A^{\leq k}$ there , then consider the set of vector $$\left\{\vec{v}\in(\Q\cap[0,1])^{\A^n}: \sum_{u\in\A^n}\vec{v}_u=1 \textrm{ and } \sum_{a\in\A}\vec{v}_{wa}=\sum_{a\in\A}\vec{v}_{aw} \textrm{ for all }w\in\A^{n-1}\right\},$$%A such effective sequence exists (see \cite{Hellouin-2014}) 
 
Now define: \[\mathbf{W}_{i} = \left\{w\in\A^{\leq\alpha_i}\ :\ \sum_{\ell\in\N}\frac 1{2^\ell}\max_{v\in\mathcal F_i\cap \A^\ell}\meas{w}([v])\leq \frac{1}{i}\right\},\]
where the maximum is worth 0 when the set is empty, which means that the sum has a finite number of terms: the maximum length of a word in $\mathcal F_i$.

Let $A$ be the algorithm that, on input $(u,r) \in \A^\ast\times\Q$ and $i\in\N$, 
\begin{enumerate}
 \itemsep0em
 \item computes all elements of $\mathcal F_i$ (evaluating a computable function over a finite set of arguments);
 \item computes all $w\in\mathbf{W}_{i}$ (a finite number of tests, and $u\longmapsto\meas{w}([u])$ is a function $\A^\ast\to\Q$ 
 which can be evaluated exactly);
  \item computes $d_i(\meas{w},\meas{u}) = \sum_{n=0}^i\frac 1{2^n}\max_{v\in\A^n}|\meas{w}([v])-\meas{u}([v])|$ for all $w\in\mathbf{W}_{i}$;
 \item outputs 1 if there exists $w\in\mathbf{W}_{i}$ such that $d_i(\meas{w},\meas{u})\leq r+\frac 1i$ and $0$ otherwise.
\end{enumerate}

We prove the correctness of this algorithm, that is, we show that \[\inf_{i\in\N} A(u,r,i) = 1 
\Longleftrightarrow \overline{\Ball(\meas{u}, r)} \cap \Ms(\Sigma) \neq \emptyset.\]

Notice that for every sequence $(w_i)_{i\in\N}$ satisfying $w_i\in\mathbf{W}_{i}$ for all $i$,
any accumulation point $\mu$ of the sequence $(\meas{w_i})_{i\in\N}$ satisfies $\mu([u])=0$ for all $u\in\mathcal F$, and therefore
$\mu\in\Ms(\Sigma)$.

If $\inf_{i\in\N} A(u,r,i) = 1$ then for all $i\in\N$ there exists $w_i\in\mathbf{W}_{i}$ such that $d_i(\meas{w_i},\meas{u})\leq r+\frac{1}{i}$. Thus one has $\dm(\meas{w_i},\meas{u})\leq r+\frac 1{2^i}+\frac {1}{i}$. We deduce that for any accumulation point $\mu$ of $(\meas{w_i})_{i\in\N}$ one has $\dm(\mu,\meas{u})\leq r$ and therefore $\overline{\Ball(\meas{u}, r)} \cap \Ms(\Sigma) \neq \emptyset$.

Conversely, let $\mu \in \overline{\Ball(\meas{u}, r)} \cap \Ms(\Sigma)$. There exists a sequence $(w_i)_{i\in\N}$ such that $w_i\in\A^{\leq\alpha_i}$ and $\dm(\mu,\meas{w_i})\leq \frac{1}{i}$. Since $\mu\in\Ms(\Sigma)$, one has
\[\frac 1i\geq \dm(\mu,\meas{w_i}) \geq \sum_{\ell\in\N}\max_{u\in\mathcal F\cap \A^\ell}\meas{w_i}([u])
 \geq \sum_{\ell\in\N}\max_{u\in\mathcal F_i\cap \A^\ell}\meas{w_i}([u])\]
which means that $w_i \in \mathbf{W}_{i}$. Furthermore, \[d_i(\meas{u},\meas{w_i})\leq \dm(\meas{u},\meas{w_i})\leq\dm(\meas{u},\mu)+\dm(\mu,\meas{w_i})  \leq r+\frac 1i\]
so $A(u,r,i) = 1$ for all $i\in\N$. 

We conclude that $\Ms(\Sigma)$ is a $\Pi_1$-computable set.
\end{proof}

\subsubsection{Equivalent definitions of $\Pi_2$-computability}

The $\Pi_2$-computability of a closed set can be defined in other equivalent ways, which requires to extend notions
of computability and $\Pi_2$-computability to functions mapping metric spaces with countable dense sets.

\begin{definition}\label{def:sigma2compact}
A sequence of functions $(f_n:\Ms(\az)\longrightarrow\R)_{n\in\N}$ is \define{a uniformly computable sequence of functions} if:
\begin{itemize}
\item there exists  $a:\N\times\N\times\A^{\ast}\longrightarrow\Q$ computable such that
$\left|f_n(\meas{w})-a(n,m,w)\right|\leq\frac 1m$ for every $w\in\A^{\ast}$ and $n,m\in\N$ (sequential
computability);
\item there exists $b:\N\longrightarrow\Q^+$ computable such that $\dm(\mu,\nu)\leq b(m)$ implies
$\left|f_n(\mu)-f_n(\nu)\right|\leq\frac{1}{m}$ for all $n,m\in\N$ (computable uniform equicontinuity).
\end{itemize}
A function $f:\Ms(\az)\longrightarrow\R$ is \define{$\Pi_1$-computable} if there exists
a uniformly computable sequence of functions $(f_n:\Ms(\az)\longrightarrow\R)_{n\in\N}$ such that $f=\inf_n\, f_n$.

A function $f:\Ms(\az)\longrightarrow\R$ is \define{$\Sigma_2$-computable} if there exists a uniformly computable 
sequence of computable functions $(f_{i,j}:\Ms(\az)\longrightarrow\R)_{(i,j)\in\N}$ such that $f=\sup_i\inf_j\, f_{i,j}$. 
\end{definition}

\begin{proposition}\label{prop:EquivComput}
 Let $\K\subset \Ms(\az)$ be a closed set. The following are equivalent:
 \begin{enumerate}
  \item $\K$ is $\Pi_2$-computable;
  \item $d_\K : \mu\longmapsto \min_{\nu\in\K}\dm(\mu,\nu)$ is $\Sigma_2$-computable;
  \item $\K = f^{-1}(\{0\})$ where $f:\Ms(\az)\longrightarrow\R$ is a $\Pi_1$-computable function.
 \end{enumerate} 
\end{proposition}

\begin{proof}~

$\mathbf{1\Longrightarrow2:}$
Assume there is a computable function
$f:\N^2\times\A^{\ast}\times\Q\longrightarrow\R$ such that, for every
$w\in\A^\ast$ and $r\in\Q$,
$\overline{\Ball(\meas{w},r)}\cap \K\ne\emptyset \Longleftrightarrow\inf_i\sup_j f(i,j,w,r)=1$.
Consider the sequence \[\left(d_{i,j,w,r}: \mu\longmapsto (1-f(i,j,w,r))\max\left(0,r-\dm(\meas{w},\mu)\right)\right)_
{(i,j,w,r)\in\N^2\times\A^\ast\times\Q}.\] The function $(i,j,w,r,w')\longmapsto d_{i,j,w,r}(\meas{w'})$ is computable as a
product of computable functions (sequential computability) and every $d_{i,j,w,r}$ is 1-Lipschitz (computable uniform equicontinuity), hence this sequence is a uniformly computable sequence of functions. We now show that $d_\K=\sup_{w,r}\sup_i\inf_jd_{i,j,w,r}$.

For any $(w,r)$ such that $\inf_i\sup_jf(i,j,w,r)=0$, we have $d_\K(\meas{w})>r$, and thus for all $\mu\in\Ms(\az)$ one has:
\[\sup_i\inf_jd_{i,j,w,r}(\mu) = \max\left(0, r-\dm(\meas{w},\mu)\right) \leq \max\left(0,
d_\K(\meas{w})-\dm(\meas{w},\mu)\right) \leq d_\K(\mu).\]

If $\mu\in\K$, we conclude that $\sup_{i,w,r}\inf_jd_{i,j,w,r}(\mu) = 0 = d_\K(\mu)$.

Now let $\mu\notin\K$. For all $\varepsilon>0$, there exists $w$ such that
$\dm(\meas{w},\mu)\leq \varepsilon$. Let $r\in\Q$ be such that $0< d_\K(\meas{w})-r<\varepsilon$, which implies that
$\overline{\Ball(\meas{w},r)}\cap \K=\emptyset$ and so $\inf_i\sup_jf(i,j,w,r) = 0$. Furthermore $d_\K(\mu)\leq
d_\K(\meas{w}) +\dm(\meas{w},\mu)\leq r+2\varepsilon$, we deduce that 
\[\sup_i\inf_jd_{i,j,w,r}(\mu) = r-\dm(\meas{w},\mu)\leq r-2\varepsilon- \dm(\meas{w},\mu)\leq d_\K(\mu)-3\varepsilon.\] 
The latter is true for every $\varepsilon>0$, we deduce that
$\sup_{i,w,r}\inf_jd_{i,j,w,r}(\mu) = d_\K(\mu)$.\bigskip

$\mathbf{2\Longrightarrow3:}$ Let $(d_{i,j}:\Ms(\az)\longrightarrow\R)_{(i,j)\in\N^2}$ be a uniformly computable sequence of
computable functions such that $d_\K=\sup_{i\in\N}\inf_{j\in\N}\, d_{i,j}$. By considering $\sup(d_{i,j},0)$ (which is uniformly computable
since $x\longmapsto \sup(x,0)$ is computable), these
functions are assumed nonnegative w.l.o.g. Denote $g_{i,n}=\inf\{{d_{i,j}\ :\ j\in\{0,\dots,n\}}\}$.
\begin{align*}
d_\K(\mu)=0&  \Longleftrightarrow  \sum_{i\in\N} \frac 1{2^i} \left(\inf_{j\in\N} d_{i,j}(\mu)\right) =0\\
        & \Longleftrightarrow  \sum_{i\in\N} \frac 1{2^i} \left(\inf_{n\in\N} g_{i,n}(\mu)\right) =0\\
        & \Longleftrightarrow \inf_{n\in\N}\sum_{i\in\N} \frac 1{2^i}  g_{i,n}(\mu) = 0,
\end{align*}
where the last equivalence is obtained by the monotone convergence theorem, $g_{i,n}$ being decreasing in $n$. Let $f_n: \mu\longmapsto \sum_{i\in\N} \frac 1{2^i}  g_{i,n}(\mu)$. $(f_n)_{n\in\N}$ is a uniformly computable sequence of functions,
since computing $(n,w’)\longmapsto f_n(\meas{w'})$ up to precision $2^{-r}$ only requires to compute the values of $d_{i,j}(\meas{w'})$ for $i,j\in\{0,\dots,r\}$, and the computable uniform equicontinuity of $(f_n)_{n\in\N}$ is a consequence of the computable uniform equicontinuity of $(d_{i,j})_{(i,j)\in\N^2}$. Thus $\K=f^{-1}(0)$ where $f=\inf_n f_n$.\bigskip

$\mathbf{3\Longrightarrow1:}$ Let $(f_n:\Ms(\az)\to\R)_{n\in\N}$ be a uniformly computable sequence of functions such that
$f=\inf_{n\in\N}\, f_n$. We assume w.l.o.g that the sequence is decreasing. For $i\in\N$, we note: 
\[d_i(\mu, \nu) = \sum_{n=1}^i\frac 1{2^n}\max_{u\in\A^n}|\mu([u])-\nu([u])|,\]
so that $0\leq \dm(\mu,\nu)-d_i(\mu,\nu)\leq \frac 1{2^i}$.
For any $n\in\N, w'\in\A^\ast$ and $i\in\N$,
define 
\[K_{n,w’,i}=\left\{(w,r)\in\A^\ast\times \Q\ :\ d_i(\meas{w}, \meas{w'})\leq r\textrm{ and } |f_i(\meas{w’})|\leq \frac{1}{n}\right\}.\] 

The function $(n,w’,i,w,r) \longmapsto \mathbf 1_{K_{n,w’,i}}(w,r)$ is
computable and thus the characteristic functions $\mathbf 1_{K_{n,w’,i}}$ are uniformly computable. Define:
\[K =\bigcap_{n\in\N}\bigcup_{\substack{w'\in\A^\ast\\
i\in\N}}K_{n,w’,i}\hspace{1cm}\textrm{and thus} \hspace{1cm} \mathbf 1_{K} = \inf_{n\in\N}\sup_{\substack{w'\in\A^\ast\\ i\in\N}}\mathbf 1_{K_{n,w’,i}}.\]

The set $K$ is $\Pi_2$-computable by definition, we are going to prove that \[K = \left\{(w,r)\in \A^\ast\times\Q\ :\ \overline{\Ball(\meas w,r)}\cap \K \ne \emptyset\right\}.\]

Let $(w,r)\in K$. For all $n\in\N$, there exists $w_n\in\A^{\ast}$ and $i_n\in\N$ such that $\dm(\meas{w},\meas{w_n})\leq r$ and $|f_{i_n}|\leq \frac{1}{n}$. By compactness, there exists $\mu\in\Ms(\az)$ such that $\dm(\meas{w},\mu)\leq r$ and $f(\mu)=\lim_{n\to\infty} f_{i_n}(w_{i_n})=0$. Thus $\overline{\Ball(\meas w,r)}\cap \K \ne \emptyset$.

Conversely, consider $\mu\in\overline{\Ball(\meas w,r)}\cap \K$. Since $f(\mu)=0$, for all $n\in\N$ there exists $i_n\in\N$ such that $|f_{i_n}(\mu)|\leq\frac{1}{2n}$. Thus one has $|f_{i_n}(\nu)|\leq\frac{1}{n}$ for all $\nu\in \Ball(\mu,b(2n))$. Let $w_n\in\A^{\ast}$ such that $\meas{w_n}\in\overline{\Ball(\meas w,r)}\cap\Ball(\mu,b(n))$. One has $|f_{i_n}(w_n)|\leq\frac{1}{n}$, that is to say $(w,r)\in K_{n,w_n,i_n}$. Since it is verified for all $n\in\N$, one deduces that $(w,r)\in K$.
\end{proof}

\begin{remark}
There exist other equivalent definitions for $\Pi_2$-computable closed set. See~\cite{Hellouin-2014} for some complements.
\end{remark}

\subsection{Computability obstruction for $\V(F,\mu)$ and $\V'(F,\mu)$}

We now state the computability obstruction for subsets of $\Ms(\az)$ reachable as limit sets
of the sequence $(\F^t\mu)_{t\in\N}$ ($\mu$-limit measures sets).

\begin{proposition}[Second computability obstruction]\label{prop:add=s2ccc}
Let $F : \az\to\az$ be a cellular automaton and $\mu\in\Mcomp(\az)$. Then $\V(F,\mu)$ and $\V'(F,\mu)$ are nonempty
$\Pi_2$-computable compact sets.
\end{proposition}

\begin{remark}
If a $\Pi_2$-computable closed set of measures is reduced to a single measure, then this measure is limit-computable. Thus Proposition~\ref{prop:add=s2ccc} implies Proposition~\ref{prop:OneStep}.
\end{remark}

\begin{proof}
 Let $f_n:\nu\longmapsto\dm(\F^n\mu,\nu)$. Since $\mu$ is computable, $(f_n)_{n\in\N}$ is sequentially computable. Moreover
$|f_n(\nu)-f_n(\nu')|=|\dm(\F^n\mu,\nu)-\dm(\F^n\mu,\nu')|\leq\dm(\nu,\nu')$ so $(f_n)_{n\in\N}$ is computably uniformly
equicontinuous. The result follows from the fact that $d_{\V(F,\mu)}(\nu)=\liminf_{n\to\infty}\dm(\F^n\mu,\nu) =
\sup_m\inf_{n>m}f_n(\nu)$, using Proposition~\ref{prop:EquivComput}.
 
The same reasoning holds for $\V'(F,\mu)$.
\end{proof}

\begin{remark}
When the initial measure is not computable, it can be used as an oracle. These obstructions will be generalised accordingly in Section~\ref{section:Oracle}. 
\end{remark}

\subsection{\texorpdfstring{Technical characterisation of $\Pi_2$-computable compact connected sets}{Technical
characterisation of Pi2-computable compact connected sets}}

$\Pi_2$-computable compact set of measures can be described as the limit points of a sequence $(\meas{w_n})_{n\in\N}$
corresponding to some uniformly computable sequence of words $(w_n)_{n\in\N}$. However, for technical reasons, 
the $\mu$-limit measures set of the construction presented in
Section~\ref{section:Construction} corresponds to the limit set of an infinite polygonal path composed of segments of
the form \[\left[\meas{u},\meas{v}\right]=\left\{t\meas{u}+(1-t)\meas{v}:t\in[0,1]\right\} \subset\Ms(\az)\] where
$u,v\in\A^{\ast}$, and is in particular connected. This is why we describe in the following proposition how compact, $\Pi_2$-computable, 
connected sets can be covered by a polygonal path corresponding to a uniformly computable sequence of words.

\begin{definition}
Let $(w_n)_{n\in\N}$ be a sequence of words of $A ^{\ast}$. Denote $\V((w_n)_{n\in\N})$ the \define{limit points of
the polygonal path} defined by the sequence of measures $(\meas{w_n})_{n\in\N}$:
\[\V((w_n)_{n\in\N})=\bigcap_{N>0}\overline{\bigcup_{n\geq N}\left[\meas{w_n},\meas{w_{n+1}}\right]}.\]
\end{definition}

\begin{proposition}\label{prop:polygonalcover}
 Let $\K\subset\Ms(\az)$ be a non-empty $\Pi_2$-computable, compact, connected set ($\Pi_2$-CCC for short). Then
there exists a uniformly computable sequence of words $(w_n)_{n\in\N}$ such that $\K =\V((w_n)_{n\in\N})$.
\end{proposition}

\begin{proof} By Proposition~\ref{prop:EquivComput} there is a uniformly computable sequence of functions $(f_n)_{n\in\N}$
satisfying $\K = f^{-1}(\{0\})$ where $f=\inf_{n\in\N} f_n$. Let $a:\N\times\N\times\A^{\ast}\to\Q$ and  $b:\N\to\Q^+$
be the computable functions given by Definition~\ref{def:sigma2compact}. Without loss of generality, we can assume that $b$ is a decreasing function and $b(i)\underset{i\to\infty}{\longrightarrow}0$.

For $k\in\N$, define:

\[\alpha_k^t = \min\{\ell\leq t\ :\ \forall u\in \A^{\leq t}, \exists w\in\A^{\leq \ell},\dm(\meas u,\meas w)\leq
b(k)\}\]
\[\alpha_k = \min\left\{\ell\in\N\ :\ \Ms(\az)=\bigcup_{u\in\A^{\leq \ell}}\Ball(\meas{u},b(k))\right\}\]
\[\Vb_k^t = \left\{u\in\A^{\leq \alpha_k^t} :\exists n\leq t \textrm{ such that } a(n,k,u)<\frac{2}{k}\right\}\]
\[\Vb_k = \left\{u\in\A^{\leq \alpha_k} : \exists n\in\N  \textrm{ such that }a(n,k,u)<\frac{2}{k}\right\}\]
Since the periodic measures are dense in $\Ms(\az)$, we have $\alpha_k^t = \alpha_k$ when $t$ is large enough. Furthermore $\alpha_k\underset{k\to+\infty}\longrightarrow +\infty$. 

Moreover, for $u\in\A^{\leq \alpha_k}$, if $f(\meas{u})<\frac{1}{k}$ there exists $n\in\N$ such that $f_n(\meas{u})<\frac{1}{k}$ so $a(n,k,u)<\frac{2}{k}$ which implies that $u\in\Vb_k$. Conversely, if $u\in\Vb_k$ then there exists $n\in\N$ such that $a(n,k,u)<\frac{2}{k}$ so $f(\meas{u})\leq f_n(\meas{u})\leq a(n,k,u)+\frac{1}{k}\leq\frac{3}{k}$.

\Claim{$\Vb_k^t$ is increasing with regards to $t$ and there exists $T_k$ such that $\Vb_k^{T_k}=\Vb_k$. 
Furthermore, the function $(k,t,w) \to 1_{\Vb_k^t}(w)$ is computable.}
\bclaimprf
For all $k$ and $t$, $\Vb_k^t \subset \Vb_k^{t+1}$. Furthermore, if $w\in \Vb_k$, then $w\in\Vb_k^t$ for $t$ large enough. 
Since $\Vb_k$ is finite, there is a $T_k$ such that $\Vb_k = \Vb_k^{T_k}$. 

The conditions for being included in $\Vb_k^t$ can be
checked by computing computable functions over a finite range of values, so $(k,t,w) \longmapsto 1_{\Vb_k^t}(w)$ is
computable.
\eclaimprf

Notice that the $T_k$ are not necessarily computable, which means that even though each $\Vb_k$ is finite, there is not
necessarily a way to know when the enumeration is finished. 

\Claim{\[\K = \bigcap_k\bigcup_{u \in \Vb_{k}}\Ball\left(\meas{u}, b(k)\right).\]}
\bclaimprf
For each element $\mu\in\K$ and $k\in\N$, there is an element $u_k\in\A^{\leq \alpha_k}$ such that
$\dm(\mu,\meas{u_k})\leq b(k)$, and therefore $f(\meas{u_k})\leq\frac{1}{k}$. Thus, there is $m\in\N$ such that $f_{m}(\meas{u_k})<\frac{1}{k}$. One deduces that
$a(m,k,u_k)\leq f_{m}(\meas{u_k})+\frac{1}{k}< \frac{2}{k}$, which means that $u_k \in \Vb_k$.  In other words, 
\[\forall k\in\N, \K \subset \bigcup_{u \in \Vb_{k}}\Ball\left(\meas{u}, b(k)\right).\]

Conversely, let $\mu\in \bigcap_k\bigcup_{u \in \Vb_{k}}\Ball\left(\meas{u}, b(k)\right)$. For all $k\in\N$, there exists $u_k\in \Vb_{k}$ such that $\mu\in\Ball\left(\meas{u_k}, b(k)\right)$. Following the same reasoning as before, there exists $n\in\N$ such that $a(n,k,u_k)\leq\frac{2}{k}$ and so
\[f(\mu)\leq f(\meas{u_k})+\frac{1}{k}\leq f_n(\meas{u_k})+\frac{1}{k}\leq a(n,k,u_k)+\frac{2}{k}\leq\frac{4}{k}.\]
 We conclude that $f(\mu)=0$ so $\mu\in\K$.
\eclaimprf

We introduce Algorithm~\ref{algo.enumSequ} for computing the sequence $(w_n)_{n\in\N}$ which realizes $\K$ as the limit points of the polygonal path defined by $(\meas{w_n})_{n\in\N}$.

\begin{algorithm}[H]\label{algo.enumSequ}
\KwData{An algorithm computing $(k,t,w) \longmapsto 1_{\Vb_k^t}(w)$}
\KwResult{Enumeration of the sequence $(w_n)_{n\in\N}$}

\begin{center}\rule{3cm}{0.7pt}\end{center}

$n \gets 0$\;
1.\ \For{$t\in\N$, by increasing order}{
2.\ \For{$k\leq t$, by increasing order}{
3.\ \For{each element $w \in \Vb_k^t$}{
\eIf{$n= 0$}{
$w_0 \gets w$\;
$n\leftarrow 1$\;
Go to the next element of $\Vb_k^t$\;}
{ \eIf{$t>0$ and $w \in \Vb_k^{t-1}$}{
Go to the next element of $\Vb_k^t$\;}{
4.\ \For{$i\leq k$, by decreasing order}{
Enumerate all finite sequences without repetition $u_1, \dots, u_{l-1} \in \Vb_i^t$\;
\If {a \= path $w_{n-1}= u_0, u_1,\dots, u_l = w$ with $\dm(u_k,u_{k+1})\leq 4b(i)$ is
found}{
$w_n\gets u_1, \dots, w_{n+l-1}\gets u_l$ and $n\gets n+l$\;
Go to the next element of $\Vb_k^t$\;}
\If {no such path was found for any $i$}{
$w_n \gets w$, $n\gets n+1$\; 
Go to the next element of $\Vb_k^t$\;}
}}}}}}

\caption{Enumeration of the sequence $(w_n)_{n\in\N}$.}
\label{algorithm.CheckSeq}
\end{algorithm}

Notice that in the fourth loop, if a path is found, then it corresponds to the largest $i\leq k$ for which such a path
exists. Now we prove the correctness of this algorithm. First notice that all elements of all $\Vb_k$ is enumerated in $(w_n)_{n\in\N}$.

\Claim{
 If $\mu\in\K$, then $\mu\in \V((w_n)_{n\in\N})$.
} 

\bclaimprf 
By Claim~2, there is a sequence of words $(u_k)_{k\in\N}$ such that $u_k\in\Vb_k$ and
$\dm(\meas{u_k},\mu)<b\left(k\right)$ for all $k\in\N$. So $u_k$ appears at some point in the sequence $(w_n)_{n\in\N}$ for every $k\in\N$. We conclude that 
$\lim_{k\to\infty}\meas{u_k}=\mu\in\V((w_n)_{n\in\N})$.
\eclaimprf

\Claim{For every $\varepsilon >0$, there exists a $t_\varepsilon$ such that in the previous algorithm, if $t'\geq t\geq
t_\varepsilon$, $w \in \Vb^{t+1}_k\backslash\Vb^t_k$ and $w’ \in\Vb^{t'+1}_{k'}\backslash\Vb^{t'}_{k'}$, then the path
$w=u_0, \dots ,w’=u_l$ built in the fourth loop satisfies $d_\K(\nu)\leq \varepsilon$ for all $\nu\in\bigcup_{0\leq
i<l}[\meas{u_i},\meas{u_{i+1}}]$.}

\bclaimprf 
Let $\varepsilon>0$, there exists by compacity a $n_\varepsilon\in\N$ such that $f(\mu)\leq \frac {4}{n_\varepsilon}\Longrightarrow d_\K(\mu)\leq \varepsilon$. Iterating the same argument, there exists $k_\varepsilon\geq 3 n_{\varepsilon}$ such that $f(\mu)\leq \frac {4}{k_\varepsilon}\Longrightarrow d_\K(\mu)\leq b(n_{\varepsilon})$. 

Let $t_\varepsilon = \max_{0\leq i\leq k_\varepsilon}(T_i)$ and assume $w \in \Vb^{t+1}_k\backslash\Vb^t_k$ and $w’ \in \Vb^{t'+1}_{k'}\backslash\Vb^{t'}_{k'}$ with $t'\geq t\geq t_\varepsilon$. By definition of the $T_i$, we have $\Vb^{t_\varepsilon}_n = \Vb^t_n=\Vb^{t’}_n$ for all $n\leq k_\varepsilon$. For $w$ and $w’$ to be chosen by the algorithm, we must have $k\geq k_\varepsilon$ and $k'\geq k_\varepsilon$ with $w\in\Vb_k$ and $w’\in\Vb_{k’}$ so $f(\meas{w})<\frac {3}{k_\varepsilon}<\frac{1}{n_\varepsilon}$ and $f(\meas{w’})<\frac {3}{k_\varepsilon}<\frac{1}{n_\varepsilon}$, thus $w,w’\in\Vb_{n_\varepsilon}$. Moreover, by definition of $k_\varepsilon$, $d_\K(\meas{w})\leq b(n_\varepsilon)$ and $d_\K(\meas{w’})\leq b(n_\varepsilon)$.

Therefore $\bigcup_{u \in\Vb_{n_\varepsilon}}\Ball\left(\meas{u}, b(n_\varepsilon)\right)$ contains $\meas{w}$ and $\meas{w’}$ as well as $\K$ in a single connected component, since $\K$ is connected. This means that in the fourth loop of the algorithm, a path can be found for some $i\geq n_\varepsilon$. The path is entirely included in $\bigcup_{u \in \Vb_{i}}\Ball\left(\meas{u},b(i)\right)$. For $\nu\in\bigcup_{u \in \Vb_{i}}\Ball\left(\meas{u},b(i)\right)$, there exists $u \in \Vb_{i}$ such that $f(\nu)\leq f(\meas{u})+\frac{1}{i}\leq\frac{4}{i}\leq\frac {4}{n_\varepsilon}$ so 
 $d_\K(\nu)\leq \varepsilon$ by definition of $n_\varepsilon$. The result follows.
\eclaimprf

\Claim{If $\mu\in \V((w_n)_{n\in\N})$, then $\mu\in\K$.}
\bclaimprf
Take any $\varepsilon>0$, and wait that the first loop reaches the value $t = t_\varepsilon$ where $t_\varepsilon$ is defined in Claim~4. 
At some point, a new element $w_n$ will be found in the third loop and it will be added to the sequence already built 
(with a path of words before it). By construction, $w_n\in\Vb_k^t$ for some $t\geq t_\varepsilon$, and the same is true for 
any element found in the third loop from now on.

By Claim~4, this means that any pair of elements $(w_k, w_{k+1})$ with $k\geq n$ added in the sequence from now on satisfies 
$\forall \nu\in[\meas{w_k},\meas{w_{k+1}}],\ d_\K(\nu)\leq \varepsilon$. This is true for all $\varepsilon>0$, so any accumulation point 
of the polygonal path $\overline{\bigcup_{n\geq N}[\meas{w_n},\meas{w_{n+1}}]}$ is included in $\K$.
\eclaimprf
\end{proof}

\section{Construction of a cellular automaton realising a given set of measures}\label{section:Construction}

In this section, we prove a reciprocal to the computability obstructions of
Proposition~\ref{prop:OneStep} and a partial reciprocal to Proposition~\ref{prop:add=s2ccc} using
Proposition~\ref{prop:polygonalcover}. Given an
uniformly computable sequence of words $(w_n)_{n\in\N}$ in $\B^{\ast}$, we construct a cellular automaton realising
$\V((w_n)_{n\in\N})$ as its $\mu$-limit measures set. We remind that $\V((w_n)_{n\in\N})$ is defined as the set of
limit points of the polygonal path defined by the sequence of measures $(\meas{w_n})_{n\in\N}$:
\[\V((w_n)_{n\in\N})=\bigcap_{N>0}\overline{\bigcup_{n\geq N}\left[\meas{ w_n},\meas{w_{n+1}}\right]}.\]
\begin{theorem}[Realisation of a computable polygonal path of measures]\label{MainTheorem}~\\
Let $(w_n)_{n\in\N}$ be a uniformly computable sequence of words of $\B^{\ast}$, where $\B$ is a finite alphabet. Then there is a
finite alphabet $\A\supset \B$ and a cellular automaton $F : \az\to \az$ such that:
\begin{itemize}
 \item for any measure $\mu \in \Mmixfull(\az)$, $\V(F,\mu) = \V((w_n)_{n\in\N})$.
 \item if $\V((w_n)_{n\in\N})=\{\nu\}$, then for any measure $\mu \in \Mergfull(\az)$,
$\F^t\mu\underset{t\to\infty}\longrightarrow \nu$.
\end{itemize}
\end{theorem}
Furthermore we get an explicit bound for the convergence rate in the first point of the theorem. Assume that $w_n$ is computable in space $O(\sqrt{n})$ 
(by repeating elements of the sequence $(w_n)_{n\in\N}$ if necessary), one has:
\[\dm(\F^t\mu,\V((w_n)_{n\in\N}))=\ O\left(\frac{1}{\log(t)}\right)+\sup\left\{\dm(\nu,\V((w_n)_{n\in\N}))\ :\
\nu\in\bigcup_{n\geq C(\log t)^2}[\meas{w_n},\meas{w_{n+1}}]\right\} \]
for some constant $C>0$. The first term of the upper bound corresponds to the intrinsic limitations of the construction, 
the second term depends on the speed of convergence of the polygonal path defined by $\meas{w_n}$ towards $\V((w_n)_{n\in\N})$, 
which is intuitively the quality of the approximation of $\V(F,\mu)$ by a computable path.

This construction has many applications detailed in Section~\ref{section:RelatedProblem}. We just mention here Corollary~\ref{cor:ConvSetMeasure} which says that every compact, $\Pi_2$-computable and connected subset of $\Ms(\B^\Z)$ can be obtained as the $\mu$-limit measures set of a cellular automaton $F:\az\to\az$ for any $\mu \in \Mmixfull(\az)$.
  
In the rest of this section, we detail the construction of this cellular automaton and prove the correctness of the construction in  Subsection~\ref{section:correction}.

\subsection{Overview of the construction}

This section presents a sketch of the construction. The alphabet $\A$, where is defined the cellular automaton, contains a symbol $\wall$ (for \define{wall}) persisting in time, except under special
circumstances, defining independent areas of computation (\define{segments}). Independently in each segment, three tasks are performed in parallel:
\begin{description}
\itemsep0em
 \item[Formatting] the initial contents of the segment are erased;
 \item[Computation and copy] each word $w_i$ is successively computed and concatenated copies of it are written on the whole segment;
 \item[Merging] the length of the segment is checked at regular intervals, and it merges with the segment to its right if it is too small.
\end{description}

The key task is the second, since the goal of the construction is that $\F^t\mu$ gets close to each measure $\meas{w_i}$ successively.
This requires that the computation is performed synchronously between all segments, so that each segment contains copies of the same $w_i$ 
at the same instant. To do this, we define another symbol $\start$ (\define{init}), which appears
only in the initial configuration, creating a wall and initialising computation and auxiliary processes. 
This process is detailed in Section~\ref{section:bootstrapping}.

Any symbol or process created in this way is referred to as \define{initialised} ; uninitialised processes are those already present in the initial configuration over which we have no control, and that we wish to erase. In particular, \define{uninitialised walls} are not considered as valid segment borders. \bigskip

Apart from $\start$ and $\wall$, the new alphabet $\A$ is divided in different layers: 
the \define{main layer} where the words $w_n$ are output and copied out, and
\define{auxiliary layers} where computation and other processes take place. This allows to perform all tasks in parallel.

\paragraph{Formatting}~Since we have no control over the initial contents of each segment, we first want to \define{format} 
the segment, that is, to erase uninitialised walls and uninitialised contents of the auxiliary layers.

Most processes defined below are designed to self-destruct when they are not initialised. This is detailed as each new process is introduced.
The difficult task is to distinguish uninitialised walls from initialised walls. 

To do that, each initialised wall sends to its right
a signal on a specific layer progressing at speed one (\define{formatting counter} - see
Section~\ref{section:formatting}), that keeps track of its age using a binary counter. 
Meanwhile, each initialised walls also keeps track of its age under the form of a binary counter on
another layer, to its left, incrementing at each step (\define{time counter} - see Section~\ref{section:time}).

Time and formatting counters already present in the initial configuration (uninitialised) have a nonnegative value at
time 0, whereas those created by an $\start$ symbol (initialised) have value 0 at time 1, and they increment at the same
rate. Thus, uninitialised walls have older time counters, and by comparing time counters and formatting counters as they
cross, we can erase older counters and uninitialised walls. Figure~\ref{figure:initialisation} is an overview
of those processes.

\begin{figure}[!ht]
\begin{center}
 \begin{tikzpicture}
 \fill[black!20!white] (-3,0) rectangle (11,6);

 \foreach \a in {0,8}
   {
   \fill[white] (\a,0) -- (\a+4,3) -- (\a+4,6.05) -- (\a,6.05) -- cycle;
   \fill[white] (\a,0) -- (\a-.5,0) arc (270:180:1.5 and 6.05) -- (\a,6.05) --
cycle;
   \draw[black, very thick] (\a,0) -- (\a,6);
   \draw[blue, thick] (\a,0) arc (270:190:1.5 and 7.25);
   \fill[blue, pattern = north east lines] (\a,0) arc (270:190:1.5 and 7.25) --
(\a,6) -- cycle;
   \filldraw[fill=white] (\a-0.17,-0.17) rectangle (\a+0.17,0.17);
   \draw (\a,0) node {\small $I$};
   }
\filldraw[fill=white] (8.3,-0.17) rectangle (8.64,0.17);
\draw (8.47,0) node {\small $I$};
\draw (7.7, -0.3) rectangle (8.77,0.3);
\draw (8.47,-0.5) node {\small{see Section~\ref{section:bootstrapping}}};
\fill[white] (2.55,4.09) -- (5.21, 6.05) -- (2.55,6.05) -- cycle;
\draw[red,thick] (0.2,0.15) -- (4,3) -- (2.55,4.09) -- (5.21, 6);
\draw[red,thick] (8.2,0.15) -- (11,2.25);
\draw[red,thick] (6,0) -- (8,1.5) -- (6.9,2.325);
\fill[blue, pattern = north east lines] (4,0) -- (3.5,0) arc (240:185:2.2 and
7.7) -- (4,5.11) -- cycle;
\draw[blue, thick] (3.5,0) arc (240:185:2.2 and 7.7);
\draw[black, very thick] (4,0) -- (4,5.11);

%Rectangles
\draw (-1.4,-0.5) rectangle (1.4,2.3);
\draw (-1.8,0.9) node {\rotatebox{90}{\small{Figure~\ref{figure:time} and
\ref{figure:formatting}}}};

\draw (2,2.4) rectangle (4.8,5.2);
\draw (1.7,3.8) node {\rotatebox{90}{\small{Figure~\ref{figure:comp1}}}};

\draw (6,0.6) rectangle (8.4,3);
\draw (5.7,1.8) node {\rotatebox{90}{\small{Figure~\ref{figure:comp2}}}};

 \end{tikzpicture}
\end{center}
\caption{Sketch of the bootstrapping and formatting processes. Vertical lines are walls. Dashed parts contain time
counters (section \ref{section:time}) and Turing machines (section \ref{section:computation}).
Slanted lines are formatting counters (section \ref{section:formatting}), white and grey areas are respectively formatted and
non-formatted.}\label{figure:initialisation}
\end{figure}
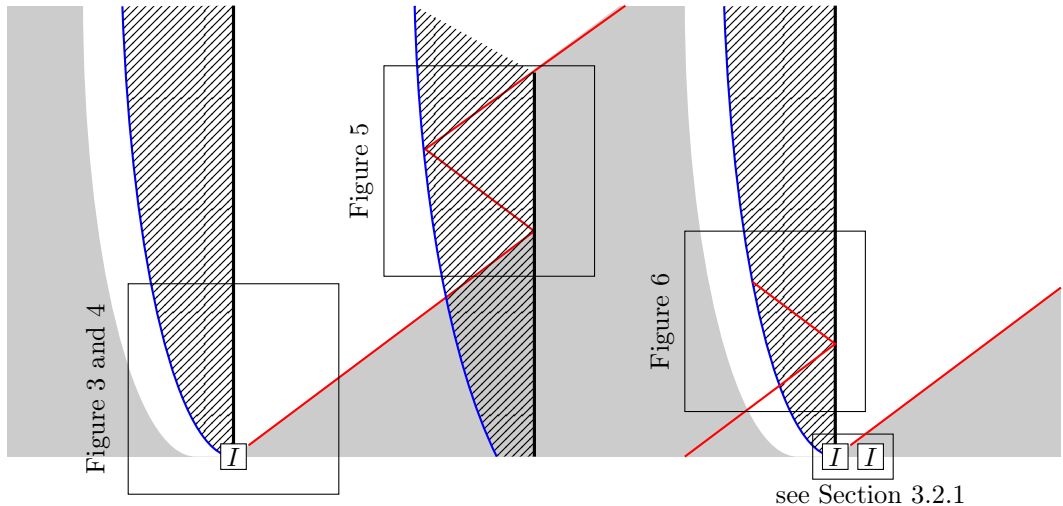

\paragraph{Computation and copy}~Meanwhile, on another layer, a Turing machine is simulated in the space delimited by the time counter.
This machine successively computes each $w_n$ 
(see Section~\ref{section:computation}) and writes concatenated copies on the main layer of the segment to its left (see
Section~\ref{section:copy}). For each $w_n$, this happens synchronously on the whole configuration, so as to approach the measure $\meas{w_n}$. 

\paragraph{Merging}~Synchronously at some time $T_n$, segments of a given length $n$ are merged with their left neighbour. This allows us to enlarge computational space and decrease the density of cells with nonempty auxiliary layers, so that they do not appear in the limit measure (see Section~\ref{section:merge}). To determine the length of its right segment, each wall sends a
signal to the right on a dedicated layer that bounces off the next wall and counts the return time.
Figure~\ref{figure:CopyMerge} is an overview of copy and merging processes.

\begin{figure}[!ht]
\begin{center}
 \begin{tikzpicture}
\fill[black!20!white] (8,0) -- (8,1) -- (5,4) -- (5,1) -- (3.5,2.5) -- (3.5,1)
--
(2,2.5) -- (2,1) -- (0,3) -- (0,1) -- (-1,2) -- (-1,1) -- (-1,0) -- cycle;
\fill[black!20!white] (8,5) -- (6.5,6.5) -- (8,6.5) -- cycle (5,5) -- (3.5,6.5)
-- (5,6.5) -- cycle
(0,5) -- (-1,6) -- (-1,6.5) -- (0,6.5) -- cycle;
\foreach \a/\b in {0/6.5,2/5,3.5/5,5/6.5,8/6.5}
   {
   \draw[black, very thick] (\a,0) -- (\a,\b);
   }
\draw[red,very thick] (2,1) -- (0,3) (3.5,1) -- (2,2.5) (5,1) -- (3.5,2.5) (8,1)
-- (5,4)
(0,1) -- (-1,2) (0,5) -- (-1,6) (5,5) -- (3.5,6.5) (8,5) -- (6.5,6.5);
\draw[dotted,very thick] (0,1) -- (2,3) -- (0,5) (2,1) -- (3.5,2.5) -- (2,4)
(3.5,1) -- (5,2.5) -- (3.5,4) (5,1) -- (8,4) -- (6,6) (5,5) -- (6.5,6.5);

\draw (1.75,4) rectangle (2,4.25);
\draw (1.875,4.1) node {\tiny M};
\draw[->, dotted] (1.875,4.25) -- (1.875,4.7);
\draw (1.75,4.75) rectangle (2,5);
\draw (1.875,4.85) node {\tiny M};

\draw (3.25,4) rectangle (3.5,4.25);
\draw (3.375,4.1) node {\tiny M};
\draw[->, dotted] (3.375,4.25) -- (3.375,4.7);
\draw (3.25,4.75) rectangle (3.5,5);
\draw (3.375,4.85) node {\tiny M};

\draw[<->] (3.5,0.2) -- (5,0.2);
\draw (4.25,-0.2) node {$n$};
\draw[->] (-1,-0.5) -- (-1,7);
\draw (-1,7.5) node {time};
\draw (-1.5,1) node {$T_{n-1}$};
\draw (-1.4,5) node {$T_n$};
\draw (-1.1,1) -- (-0.9,1) (-1.1,5) -- (-0.9,5);
\draw[dotted] (-1,1) -- (9,1) (-1,5) -- (9,5);

%Rectangles
\draw (0.9,0.9) rectangle (4,4);
\draw (0.7,2.5) node {\rotatebox{90}{\small{Figure~\ref{figure:merge}}}};

\draw (6.4,0.6) rectangle (8.4,2.6);
\draw (6.2,1.6) node {\rotatebox{90}{\small{Figure~\ref{figure:copy}}}};

\draw (5.4,5.2) rectangle (6.8,6.6);
\draw (6.1,7.2) node {\footnotesize{Newer merging signals}};
\draw (6.1,6.85) node {\footnotesize{erase older signals}};
 \end{tikzpicture}
\end{center}
\caption{Sketch of the copying and merging processes. Here all walls are
initialised. Slanted thick lines are copy processes (see Section
\ref{section:copy}), slanted dotted lines are merging signals (see Section
\ref{section:merge}).}\label{figure:CopyMerge}
\end{figure}
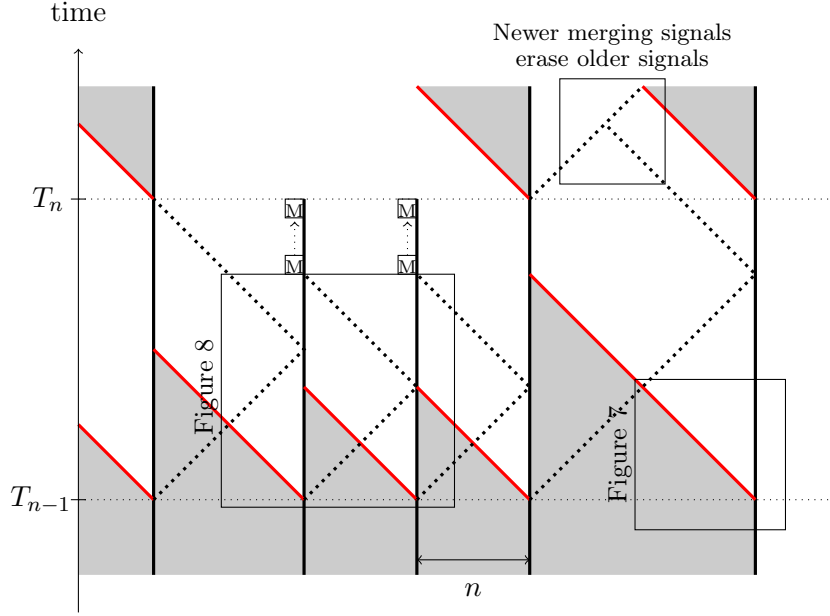

\paragraph{Alphabet}~We obtain an enlarged alphabet $\A = \left\{\start, \wall\right\}  \cup \alph{main}  \times \alph{comp} 
\times \alph{time} \times \alph{format} \times \alph{copy} \times \alph{merge}$. All those alphabets contain a
symbol $\#$ (blank) representing the absence of information.
\begin{itemize}
 \item $\start$ and $\wall$ are the two above-mentioned symbols;
 \item $\alph{main} = \B\cup \{\#\}$ is the layer on which $w_n$ is output and then copied out;
 \item $\alph{comp}$ is the layer where Turing machines are simulated to compute $w_n$ and other processes;
 \item $\alph{time}$ is the layer on which time counters are incremented; 
 \item $\alph{format}$ is the layer on which formatting counters move and are incremented, and where comparisons are
done;
 \item $\alph{copy}$ is a layer used in the process of writing copies of the output on the main layer;
 \item $\alph{merge}$ is a layer used in the process of merging two segments.
\end{itemize}

We have $\B \subset \A$ up to the identification $b \Longleftrightarrow (b, \#, \#, \#, \#, \#)$.
If $u \in \A$, denote $\texttt{main}(u)$, resp. $\texttt{comp}(u),\ \texttt{time}(u)$\dots the projections on each layer (the result being $\#$
on $\start$ and $\wall$). \bigskip

We detail the different alphabets in the following sections. As we will see, our construction needs interactions
at a distance at most three, so we can take $\U_F = \{-3,\dots,3\}$ as the neighbourhood of the local rule of $F$.

\subsection{Formatting the segments}

\subsubsection{Bootstrapping}\label{section:bootstrapping}
If two symbols $\start$ are separated by two cells or less, the rightmost one is destroyed. Otherwise, every $\start$ symbol
turns into a $\wall$, erases the contents of three cells to their right and left (including walls),
and initialises on its left a computation process and a time counter, and on its right a formatting counter. No more $\start$ or $\wall$ symbols can be created. 

\begin{definition}
 Let $x\in\az$. The set of positions $[i,j]$ is a \define{segment at time 0} if $x_i$ and $x_j$ are the symbol $\start$ and this symbol does not appear for intermediate coordonate in $]i,j[$. It is a 
\define{segment at time $t$} if $F^t(x)_i$ and $F^t(x)_j$ are initialized walls $\wall$ (that is to say $x_i=x_j=\start$) and there is not  initialized walls  between them at time $t$. Define the
\define{length} of this segment as $j-i-1$.
\end{definition}

 Walls persist over time and are only destroyed under three circumstances:
\begin{itemize}\itemsep0em
 \item if the time layer of the computation layer of the cell to its left is empty (so the wall must be uninitialised);
 \item by a growing time counter (see Section~\ref{section:time} and Facts~\ref{fact:detached}, \ref{fact:attached} and \ref{fact:erasingtime});
 \item by the merging process detailed in Section~\ref{section:merge}.
\end{itemize}

As the only exception, if a segment is of length three at time 0, then the leftmost $\start$ prevents the creation of a
time counter for the rightmost wall at time 1 and the wall itself is destroyed at time 2. Thus
segments have minimum length four from time 2 onwards.

\subsubsection{Counters}
All counters are binary in a redundant basis, so that they can be incremented by one at each step (keeping track of
current time) in a local manner. Notice that in the following two definitions, the indexing of the letters is inverted.

\begin{definition}[Redundant binary basis]
 Let $u = u_0\dots u_{n-1} \in \{0,1,2\}^*$. The \define{value} of $u$ is \[val(u) = \sum_{i=0}^{n-1}
u_i2^i.\]Since the basis is redundant, different words in $\{0,1,2\}^*$ can have the same value.
\end{definition}
\begin{definition}[Incrementation]\label{definition:incrementation}
 The incrementation operation $inc : \{0,1,2\}^*\longmapsto \{0,1,2\}^*$ is defined in the following way. If $u_{|u|-1} =
2$, then $|inc(u)| = |u|+1$, $|u|$ otherwise, and:
\[inc(u)_i = \left\{\begin{array}{ll}
1&\tx{if }i=|u|\tx{ and } u_{|u|-1} = 2;\\
u_i \mod 2 +1 & \tx{if }i=0\tx{ or } u_{i-1} = 2;\\
u_i \mod 2 & \tx{otherwise.}\end{array}\right.\]
\end{definition}

Intuitively,  the counter is increased by one at the rightmost bit and $2$ behaves as a carry propagating along the
counter. If the most significant bit was a carry, the length of the counter is increased by one. Thus:
\begin{fact}
 $val(inc(u)) = val(u)+1$.
\end{fact}

Taking a symbol as spark where the counter is incremented, in our case the symbol $\wall$, and another one to precise the end of the word, in our case $\#$, this operation can be defined locally and can be seen as the local rule of a cellular automaton.

\subsubsection{Time}\label{section:time}
We use the alphabet $\alph{time} = \{0,1,2,\#\}$. In a configuration, a time counter is a word of maximal length
containing no $\#$ in the time layer. A time counter is \define{attached} if it is bounded on its right by a wall
$\wall$, \define{detached} otherwise. \bigskip

\begin{figure}[!ht]
\begin{center}
 \begin{tikzpicture}

\foreach \x\y\a in
{0/1/?,0/2/?,0/3/?,0/4/\#,0/9/\#,0/10/?,1/2/\#,1/3/\#,1/4/?,1/5/\#,1/9/\#,
1/10/\#,2/2/\#,2/3/\#,2/4/\#,2/5/\#,2/6/\#,2/10/\#,3/2/\#,3/3/\#,3/4/\#,3/5/\#,
3/6/\#,3/7/\#,3/10/\#,4/3/\#,4/4/\#,4/5/\#,4/6/\#,4/7/\#,4/8/\#,4/10/\#,5/3/\#,
5/4/\#,5/5/\#,5/6/\#,5/7/\#,5/8/\#,5/9/\#,5/10/\#,6/3/\#,6/4/\#,6/5/\#,6/6/\#,
6/7/\#,6/8/\#,6/9/\#,6/10/\#,7/4/\#,7/5/\#,7/6/\#,7/7/\#,7/8/\#,7/9/\#,7/10/\#}
  {
  \draw[dotted](-0.8*\y,0.8*\x) rectangle (-0.8*\y +0.8,0.8*\x +0.8);
  \draw(-0.8*\y+0.4,0.8*\x+0.4) node {\a};
  }

  \draw (0.4,0.4) node {I};
  \foreach \x in {1,...,7}
  {
  \filldraw[fill = black!50!white] (0,0.8*\x) rectangle (0.8,0.8*\x+0.8);
  \draw (0.4,0.8*\x+0.4) node {W};
  }
  \foreach \x\y\a in
{1/1/0,7/2/0,2/1/1,4/1/1,6/1/1,4/2/1,5/2/1,7/3/1,3/1/2,5/1/2,7/1/2,6/2/2,0/5/1,
0/6/0,0/7/2,0/8/1,1/6/0,1/7/0,1/8/2,2/7/0 ,2/8/0,2/9/1,3/8/0,3/9/1,4/9/1}
  {
  \draw[very thick] (-0.8*\y,0.8*\x) rectangle (-0.8*\y +0.8,0.8*\x +0.8);
  \draw (-0.8*\y+0.4,0.8*\x+0.4) node {$\a$};
  }
  \draw [very thick] (0,0) rectangle (0.8,0.8);
 \end{tikzpicture}
  \caption{A detached time counter, and a time counter attached to an initialised wall. Only the time layer is
represented. ? cells have arbitrary values.}
\label{figure:time}
 \end{center}
\end{figure}

At each step, attached counters are incremented by one while detached counters have their rightmost bit deleted (see
Figure~\ref{figure:time}). Indeed, detached counters are uninitialised and can be safely deleted. Formally,
\begin{itemize}
 \item if $u_1 = \wall$, then $\texttt{time}(F(u)_0)= \texttt{time}(u_0)\mod 2 +1$;
 \item if $\texttt{time}(u_1) = \#$, then $\texttt{time}(F(u)_0)=\#$;
 \item otherwise, follow the incrementation definition
(Definition~\ref{definition:incrementation}).
\end{itemize}

When a counter increases in length, it may erase a wall by overwriting it. However, this is not a problem, as we shall
see in Facts~\ref{fact:detached} and \ref{fact:attached}.

\begin{fact}\label{fact:detached}
 An initialised wall cannot be erased by a detached time counter.
\end{fact}

\begin{proof} A detached time counter is not incremented and can extend by one cell at most because of the carries
initially present in the word. But $\start$ symbols erase two cells to their right at initialisation.
\end{proof}

\begin{fact}\label{Time}
 Let $x\in\az$ be the initial configuration. Each attached time counter $u$ in $F^t(x)$ satisfies $val(u)\geq t-1$, the
equality being attained if this counter is attached to an initialised wall.
\end{fact}

\begin{proof} No time counter is created except at $t=1$ (by $\start$). Therefore such a counter was present either
in the initial configuration (with a nonnegative value), or was created at $t=1$ by a $\start$ symbol. It is incremented
by one at each step in both cases.
\end{proof}

Thus we can use time counters to tell apart initialised walls from non-initialised walls, which is the object of
the next section.

\subsubsection{Formatting and comparisons}\label{section:formatting}

We want to implement a counter in a new layer which is compared to the time counter when they are in interaction. The formatting layer $\alph{format}$, contains the symbol $\#$ as all layers.  A \define{formatting counter} is a word of maximal length where the value of the cell in $\alph{format}$ is different than $\#$.  Formatting counters are defined and incremented at each step in a similar way as time counters, but they have a range of different behaviours. Thus the other elements of $\alph{format}$ are decomposed into two layers $\alph{value}=\{0,1,2,\#\}$ and $\alph{state}$ where the possible value are:
\begin{description}
 \item[``Go'' state] The counter progresses at speed one to the right.
 \item[``Stop'' state] Once a wall is encountered, the counter progressively (right to left) stops.
 \item[Comparison states] Once the whole counter has stopped, we locally compare the formatting counter and the time
counter, left to right, with a method we describe later which use the symbol $\{-,=_-,=,=_+,+\}$.
\end{description}

The wall is destroyed if the formatting counter is strictly younger, and the formatting counter is destroyed otherwise (see Figures~\ref{figure:comp1} and \ref{figure:comp2}). In the former case, the counter progressively returns to the ``Go''
state. \bigskip

\begin{figure}[!ht]
\begin{center}
 \begin{tikzpicture}
  \draw (0.4,0.4) node {I};

  \foreach \x in {1,...,7}
  {
  \filldraw[fill = black!50!white] (0,0.8*\x) rectangle (0.8,0.8*\x+0.8);
  \draw (0.4,0.8*\x+0.4) node {W};
  }
  \foreach \x/\y/\a in
{1/1/0,7/5/0,2/2/1,4/2/1,6/6/1,4/4/1,5/3/1,7/3/1,3/3/2,7/7/2,6/4/2,5/5/2,4/3/\#,
5/4/\#,6/5/\#,7/4/\#,7/6/\#,0/3/1,0/4/\#,0/5/1,0/6/0,
  1/4/1,1/5/\#,1/6/2,2/5/2,2/6/\#,2/7/1,3/6/0,3/7/\#,3/8/2,4/7/1,4/8/\#,5/8/1}
  {
  \draw[very thick] (0.8*\y,0.8*\x) rectangle (0.8*\y +0.8,0.8*\x +0.8);
  \draw (0.8*\y+0.4,0.8*\x+0.6) node {$Go$};
  \draw[dotted] (0.8*\y,0.8*\x+0.4) -- (0.8*\y+0.8,0.8*\x+0.4);
  \draw (0.8*\y+0.4,0.8*\x+0.2) node {$\a$};
  }
  \draw [very thick] (0,0) rectangle (0.8,0.8);

\foreach \x/\y in {2/1,3/2,4/1,5/2,6/3,7/2}
  {
 \foreach \t in {1,...,\y}
  {
  \draw[dotted](0.8*\t,0.8*\x) rectangle (0.8*\t +0.8,0.8*\x +0.8);
  \draw(0.8*\t+0.4,0.8*\x+0.4) node {\#};
  }
  }

\foreach \x\y in {0/7,1/7,2/8,7/8}
\foreach \t in {\y,...,8}
{\draw[dotted](0.8*\t,0.8*\x) rectangle (0.8*\t +0.8,0.8*\x +0.8);
  \draw(0.8*\t+0.4,0.8*\x+0.4) node {\#};
}

\foreach \x in {1,...,6}
{
 \draw[dotted](0.8*\x+1.6,0.8*\x) rectangle (0.8*\x +2.4,0.8*\x +0.8);
 \draw(0.8*\x+2,0.8*\x+0.4) node {\#};
 \draw[dotted](0.8*\x+0.8,0.8*\x) rectangle (0.8*\x +1.6,0.8*\x +0.8);
 \draw(0.8*\x+1.2,0.8*\x+0.4) node {\#};
}

 \draw[dotted](1.6,0) rectangle (2.4,0.8);
 \draw(2,0.4) node {?};
 \draw[dotted](0.8,0) rectangle (1.6,0.8);
 \draw(1.2,0.4) node {?};
 \draw (8,3) rectangle (8.8,3.4) rectangle (8,3.8);
\draw[->,dotted] (10,3.2) node {\small{\texttt{value}}} (9.5,3.2) -- (9,3.2);
\draw[->,dotted] (10,3.6) node {\small{\texttt{state}}} (9.5,3.6) -- (9,3.6);
 \draw(4.65,3.05) node {{\footnotesize X}};
 \draw(3.85,3.05) node {{\footnotesize X}};

\end{tikzpicture}
\caption{One initialised and one uninitialised formatting counter. X symbols mark the cells where values are prevented to
appear to avoid merging: the right counter is dominated. Only the formatting layer is represented.}

\label{figure:formatting}
 \end{center}
\end{figure}

Changing state takes some time to propagate the information along the counter. Therefore, counters passing from a ``Go''
state to a ``Stop'' state are temporarily in a situation where the left part of the counter progresses whereas
the right part has not. To avoid erasing information, counters in a ``Go'' state have \define{buffers}, i.e. the value
of the counter is only written on half the cells, the other half containing $(Go,\#)$
(see Figure~\ref{figure:formatting}).\bigskip

When its length increase, a counter never merges with another counter, erasing bits from the right-hand
counter instead in order to avoid merging: we say the right-hand counter is \define{dominated}. Notice that it is
impossible for a counter located to the right of another counter to be initialised, and so it is safe to erase bits of
it.

\begin{fact}
 Let $x\in\az$ be the initial configuration. Any non-dominated formatting counter $u$ of $F^t(x)$ satisfies $val(u)\geq
t-1$, the equality being attained if the counter is initialised.
\end{fact}
\begin{proof} Similar to Fact~\ref{Time}.
\end{proof}

Thus we guarantee that an initialised (hence non-dominated) formatting counter is strictly younger than any
uninitialised wall, and symmetrically. Uninitialised formatting counters can only progress to the right 
to be destroyed by the nearest initialised wall. We will see that dominated counters, whose value is arbitrary, are not a problem since
they are erased before any comparison takes place.

\begin{definition}[Comparison method]
 Let $u=u_0u_1\dots $ and $v = v_0v_1\dots$ be two counters in redundant binary basis (adding zeroes so that $|u| =
|v|$). Let us note $sign(u-v)$ the result of the comparison between $u$ and $v$, that is, $+, 0$ or $-$.
\begin{description}
 \item[Case 1] if $|u| = |v| = 1$, $sign(u-v) = sign(u_0-v_0)$;
 \item[Case 2] if $u_0+\lfloor u_1/2\rfloor > v_0+\lfloor v_1/2\rfloor+1$, then $sign(u-v) = +$,\\ and symmetrically;
 \item[Case 3] if $u_0+\lfloor u_1/2\rfloor = v_0+\lfloor v_1/2\rfloor+\varepsilon$ (for
some $\varepsilon\in\{-1,0,1\}$), then $sign(u-v) = sign((u_1'+2\varepsilon) u_2\dots - v_1'v_2\dots)$,\\
where $u_1' = u_1 \mod 2$ and $v_1'=v_1 \mod 2$.
\end{description}
\end{definition}
In other words, we do a bit-by-bit comparison starting from the most significant bit, considering that $\#$ is equal to $0$, and taking into
account the carry propagation ``in advance'', so that the incrementation and carry propagation can continue during
the comparison. When the ``local difference'' $\varepsilon$ is too small, the result cannot be
determined locally and a remainder is carried (consider a comparison between $120\cdots0$ and $11\cdots12$).

Formally, for each pair of bits $(u_n,v_n)$, we add 1 to each bit if the following bit of the
corresponding counter is $2$, and depending on the value of $u_n-v_n+2\varepsilon$:
\begin{center}
\begin{tabular}{|c|c|c|c|c|c|}
\hline result&$<-1$\quad&\quad$-1$\quad\quad&\quad$0$\quad\quad&\quad$+1$\quad\quad&$>+1$\quad\\
\hline \quad new state \quad&$-$&$=_-$&$=$&$=_+$&$+$\\\hline
\end{tabular}
\end{center}

If the result can be determined locally (cases 1 and 2), the state is changed to $+$ or $-$, and the
result propagates to the right without further comparisons. Otherwise (case 3), the state changes to $=$,
which means future bit comparisons will decide the result in the same way. If there is a remainder $\varepsilon$, it is
remembered for the next comparison by having three states $=_-, =_+, =$. See Figure \ref{figure:comp2} for an
example.\bigskip

\begin{figure}[!ht]
\begin{center}
 \begin{tikzpicture}[scale=1.1]
  \foreach \x/\y/\z in {0/2/6,1/1/5,2/1/4,3/1/3,4/1/3,5/1/3,6/1/3,7/1/3,8/0/3}
  \foreach \t in {\y,...,\z}
  {
  \draw[very thick] (-0.8*\t,0.8*\x) rectangle (-0.8*\t +0.8,0.8*\x +0.8);
  }

\foreach \x/\y in {0/7,1/6,2/5,3/4,4/4,5/4,6/4,7/4,8/4}
  \foreach \t in {\y,...,7}
  {
  \draw[dotted] (-0.8*\t,0.8*\x) rectangle (-0.8*\t +0.8,0.8*\x +0.8);
  \draw (-0.8*\t+0.4,0.8*\x+0.4) node {\#};
  }
\draw[dotted] (-0.8,0) rectangle (0,0.8);
\draw (-0.4,0.4) node {\#};
\foreach \x in {0,...,7}
{
  \filldraw[fill = black!50!white] (0,0.8*\x) rectangle (0.8,0.8*\x+0.8);
  \draw (0.4,0.8*\x+0.4) node {W};
}
\foreach \x\y in {0/2,0/4,0/6,1/3,1/5,2/4,7/1,8/0,8/2}
  {
  \draw (-0.8*\y+0.4,0.8*\x+0.4) node {Go};
  }
  \foreach \x\y in {1/1,2/1,2/2,3/1,3/2,3/3,4/1,4/2,5/1}
  {
  \draw (-0.8*\y+0.4,0.8*\x+0.4) node {Stop};
  }
  \foreach \x\y in {0/3,0/5,1/2,1/4,2/3,8/1}
  {
  \draw (-0.8*\y+0.4,0.8*\x+0.4) node {\textcolor{gray}{Go}};
  }
  \foreach \x\y in {4/3,5/3,6/3,7/3,8/3}
  {
  \draw (-0.8*\y+0.4,0.8*\x+0.4) node {$=$};
  }
  \foreach \x\y in {5/2,6/2,6/1,7/2}
  {
  \draw (-0.8*\y+0.4,0.8*\x+0.4) node {$-$};
  }
  \draw[red,very thick] (-0.8,0.8) -- (-0.8,1.6) -- (-1.6,1.6) -- (-1.6,2.4) --
(-2.4,2.4)
		    (-2.4,3.2) -- (-1.6,3.2) -- (-1.6,4) -- (-0.8,4) --
(-0.8,4.8) -- (0,4.8)
                   (0,5.6) -- (-0.8,5.6) -- (-0.8,6.4) -- (-1.6,6.4) --
(-1.6,7.2);

 \end{tikzpicture}
  \caption{A younger formatting counter encountering an older wall, which is destroyed. Only the state layer of
$\alph{format}$ is represented, with greyed words for buffers.}
\label{figure:comp1}
 \end{center}
\end{figure}

\begin{figure}[!ht]
\begin{center}
 \begin{tikzpicture}
\foreach \x in {1,...,7}
{
  \filldraw[fill = black!50!white] (-1,\x-1) rectangle (0,\x);
  \draw (-0.5,\x-0.5) node {W};
}
\foreach \y/\x/\a/\b/\c in {
			    0/1/Stop/2/1, 0/2/Stop/2/1, 0/3/Stop/1/1,0/4/Stop/1/1,
			    1/1/Stop/1/2, 1/2/Stop/1/1, 1/3/Stop/2/1,1/4/=/1/1,
			    2/1/Stop/2/1, 2/2/Stop/1/2, 2/3/=_+/0/1, 2/4/=/2/1,
			    3/1/Stop/1/2, 3/2/+/2/0, 3/3/=_+/0/2, 3/4/=/0/1,3/5/=/1/\#,
			    4/1/+/2/1, 4/2/+/0/1, 4/3/=_+/1/0, 4/4/=/0/2,4/5/=/1/\#,
			    5/2/+/1/1, 5/3/=_+/1/0, 5/4/=/0/0, 5/5/=/1/1,
			    6/3/=_+/1/0, 6/4/=/0/0, 6/5/=/1/1
			    }
{
\draw[very thick] (-\x,\y) rectangle (-\x-1,\y+1);
\draw[dotted] (-\x,\y+0.5) -- (-\x-1,\y+0.5) (-\x-0.5,\y)--(-\x-0.5,\y+0.5);
\draw (-\x-0.5,\y+0.75) node {$\a$};
\draw (-\x-0.25,\y+0.25) node {$\c$};
\draw (-\x-0.75,\y+0.25) node {$\b$};
 }
\foreach \x/\y in {0/5,1/5,2/5,3/6,4/6,5/6,6/6}
\foreach \t in {6,...,\y}
{
\draw[dotted] (-\t,\x) rectangle (-\t-1,\x+1);
\draw (-\t-0.5,\x+0.5) node {\#};
}

\draw[dotted] (-2,5) rectangle (-1,6);
\draw (-1.5,5.5) node {\#};
\draw[dotted] (-3,6) rectangle (-1,7);
\draw (-1.5,6.5) node {\#};
\draw[dotted] (-3,6) rectangle (-2,7);
\draw (-2.5,6.5) node {\#};
\draw[red,very thick] (-5,1) -- (-4,1) -- (-4,2) -- (-3,2) -- (-3,3) -- (-2,3)
-- (-2,4) -- (-1,4);

% Partie droite
\draw[very thick] (2,3) rectangle (3,4);
\draw[dotted] (2,3.5) -- (3,3.5) (2.5,3.5)--(2.5,3);
\draw (2.5,4.3) node {\texttt{state}};
\draw (1.9,2.7) node {\texttt{value}};
\draw (3.1,2.7) node {\texttt{time}};
 \end{tikzpicture}
\end{center}
\caption{The comparison process in detail. Here the formatting counter is older than the time counter and is destroyed.
Only the
layers $\alph{time}$ and $\alph{format}$ are represented.}
\label{figure:comp2}
\end{figure}

After the comparison, two cases are possible:
\begin{itemize}
 \item if the state of the rightmost bit is $-$ or $=_-$, the wall is strictly older than the counter. The wall is
destroyed and
the state of the rightmost bit becomes ``Go''. The counter then progressively returns to the ``Go'' state.
 \item if the state of the rightmost bit is $+$, $=_+$ or $=$, the wall is younger than the counter. The rightmost bit
is erased, and the rest of the counter is progressively erased in a similar way as a detached time counter.
\end{itemize}
The second case covers the case where both the counter and the wall are initialised (result $=$), which means that the
formatting counter has finished formatting its segment and may be erased. Also, if the counter is dominated,
then its leftmost bit is erased at each step, preventing the comparison to start, until the counter is
entirely erased.\bigskip

To sum up, $\alph{format}=\{\#\} \cup \left(\{Go\}\times \{0,1,2,\#\}\right) \cup \left(\{Stop, +, -, =, =_+, =_-\} \times
\{0,1,2\}\right)$.\bigskip

When a formatting counter reaches the right wall of the segment, the segment is said to be $\define{formatted}$. This implies
that the segment contains no more uninitialised walls.

\begin{fact}\label{fact:formattime}
 At time $k(1+\lceil\log k\rceil)$, all segments of length $k$ (for $k>3$) are formatted.
\end{fact}

\begin{proof} As long as $t\leq k(1+\lceil\log k\rceil)$, any initialised formatting counter has length
$\lceil\log t\rceil\leq 2\lceil\log k\rceil$ (excluding the buffers) since it is in base 2. The
counter
progresses at speed one except when it meets another wall. Each comparison takes a time equal to twice the current
length of the counter (again excluding the buffers). Furthermore, two consecutive walls are separated by three cells at
least (cf. Section~\ref{section:bootstrapping}). Thus, the segment is formatted in less than $k
+ \frac k4\cdot 2\cdot 2\lceil\log k\rceil$ steps, which is coherent with our first assumption.
\end{proof}

\begin{fact}\label{fact:attached}
 An initialised wall cannot be erased by a time counter attached to a uninitialised wall.
\end{fact}

\begin{proof} Consider two walls, the left being initialised and the right
uninitialised. As explained in Section~\ref{section:bootstrapping}, we can assume they are separated by $k>3$
cells. The value of the time counter attached to the right wall cannot exceed $2^{k-3}$ at time $1$ (since every symbol
$\start$ erases three cells to its right at time $1$), so it takes more than $2^k-2^{k-3}$ steps before the left wall is erased.
According to Fact~\ref{fact:formattime}, the right wall is destroyed in less than $k(1+\lceil\log k\rceil)$ steps,
and from then its time counter takes at most $k$ more steps to be erased.

For $k\geq 5$, $k(1+\log k)+k\leq 2^k - 2^{k-3}$, so the counter is erased before it reaches the left wall. For $k=4$,
any wall between them is destroyed at time 1, so the destruction time is actually less than $k + 2\log k + k \leq
2^k-2^{k-3}$. 
\end{proof}

\subsection{Computation and copy}
\subsubsection{Simulating a Turing machine in a cellular automaton}
Let $\mathcal{TM}=(Q,\Gamma,\#,q_{0},\delta,Q_{F})$ be a Turing machine. We simulate this machine in a
cellular automaton $F$ on the alphabet $(\Gamma\cup\#)\times(Q\cup\#)$. The left part contains the content of the tape; the
right part contains the state of the machine for the cell where the head is located, and $\#$ everywhere else.

The local rule of $F$ is governed by the rules of the machine. That is, for all $u\in((\Gamma\cup\#)\times(Q\cup\#))^\Z$,
and writing $\_$ to denote an arbitrary value:
\begin{itemize}\itemsep0em
 \item if the head is on $u_0$ and $\delta(u_0) = (q,\gamma,\_)$, then $F(u)_0 = (\gamma,\#)$;
 \item if the head is on $u_1$, $\delta(u_1) = (q,\_,\leftarrow)$ and $u_0 = (\gamma',\#)$, then $F(u)_0 = (\gamma',q)$;
 \item similarly if the head is on $u_{-1}$ and $\delta(u_{-1}) = (q,\_,\rightarrow)$;
 \item otherwise, $F(u)_0 = u_0$.
\end{itemize}
When starting from a configuration filled with $(\#,\#)$ everywhere except for a finite window with only one head, the time
evolution of the cellular automaton matches the time evolution the Turing machine. The Turing machine considered in the proof does not stop, but by consistency we can assume that when the machine has stopped (the state being in $Q_F$), the local rule is the identity function.

\subsubsection{Computation}\label{section:computation}
Computation takes place to the left of each initialised wall. $\alph{comp}$ is divided
into three layers, on which three Turing machines are simulated, using the alphabet \[\alph{comp} =\bigotimes_{i=1}^3
(\Gamma_i\cup\#)\times(Q_i\cup\#).\] We adapt the simulation so that these Turing machines can read input from or write 
output to another layer (when indicated).

We now describe the operations to be performed symchronously between times $T_{n-1}$ and $T_n$ that we will fix later.
Assume that, at time $T_{n-1}$, $n$ is already written on the layer 1 and $T_{n-1}$ on layer 3. The machines:
\begin{enumerate}\itemsep0em
 \item replace $n$ by $n+1$ on layer 1 and stops;
 \item compute $w_n$ on layer 2, outputting it on the main layer, and stops;
 \item compute $T_n$ on layer 3, and stops;
\end{enumerate}
When $t = T_n$ ($t$ being read from the time layer), the copying process triggers and the next computation starts, 
except when merging occurs; see next subsections.

All these operations must be performed in less than $T_n - T_{n-1}$ steps. We now fix the value of $T_n$ so that it is indeed possible.\bigskip

A Turing machine with tape alphabet $\Gamma$ and set of states $Q$ and using only a computational space $S$ stops in time 
$S\cdot|\Gamma|^{S}\cdot |Q|$ which is the number of possible configurations. Otherwise, the same configuration would be reached twice, entering a loop.

Therefore there exists a constant $q>0$ large enough that the operations on layers 1 and 2 can be performed in space $\lfloor\sqrt n\rfloor\log_2q$
and time $O(q^{\lfloor\sqrt n\rfloor})$. Furthermore, the function $(r,n) \longmapsto r^{\lfloor\sqrt n\rfloor}$ is computable in space $\lfloor\sqrt n\rfloor\log_2r$ 
(length of the output) and time $O(n^{3/2}(\log r)^2)$ (compute $\lfloor\sqrt n\rfloor$ in time $O(n)$, then perform $\lfloor\sqrt n\rfloor$ multiplications
between numbers of length $\lfloor\sqrt n\rfloor\log_2 r$ at most in time $O((\sqrt n\log_2r)^2)$). \bigskip

In other words, if we fix \[T_n-T_{n-1} = q^{\lfloor\sqrt n\rfloor},\]
then the operation on layer 3 can be performed in space $\lfloor\sqrt n\rfloor\log_2q$ and time $O(q^{\lfloor\sqrt n\rfloor})$.
However, we need an upper bound on the time at each step and not only an asymptotic bound. This is solved by 
the linear speedup theorem for Turing machines: we can divide the computational time by any fixed constant $C$ by replacing each machine $M_i$
by a new machine $M'_i$, such that $M'_i$ performs $C$ computational steps of $M_i$ at each step, increasing the radius as necessary.

\begin{remark}
We fixed $T_n$ so that the computation space is of size $\sqrt{n}$ at time $T_n$
and constitutes an asymptotically negligible fraction of its segment. We could choose instead of $\sqrt n$ any other function in $o(n)$ which is time constructible.
\end{remark}

Similarly to time counters, whenever they find an empty computational layer to their right (instead of a wall or another computation state), computation states was replaced by the symbol $(\#,\#)$. Thus uninitialised computation states self-destruct progressively. This requires that the Turing machines are adapted so that they never write $(\#,\#)$ in a cell in the middle of a computation.
\subsubsection{Copying}\label{section:copy}

On the layer $\alph{copy}$, the cellular automaton copies periodically the words produced by the Turing machine in view to make samplings of the limit measures, we just put $\alph{copy} = \B\cup\{\#\}$.

At time $T_n\ (n\geq 0)$, $w_n$ has been output on the main layer, followed by a symbol $\#$. 
If the segment is not in the process of merging, repeated copies of $w_n$ have to be written over the main layer. 
The Turing machine triggers the copying process by copying the rightmost letter of $w_n$ from the
main layer to the copy layer.

\begin{description}
\itemsep0em
 \item[First phase] As long as it has not met a symbol $\#$, the word on the copy layer progresses at speed -2 (that is to say if it is in the position $[i,j]$ it moves to the position $[i-2,j-2]$) and at each step a
letter is copied from the main layer to the tail of the word;
 \item[Second phase] The word keeps progressing at speed -2 but the head loses
one letter at each step and copies it on the main layer. The tail keeps copying letters from the main layer.
\end{description}
 
Intuitively, the cellular automaton performs a caterpillar-like movement between the copy and main layers (see
Figure~\ref{figure:copy} for an example). The process ends when it meets a wall. 

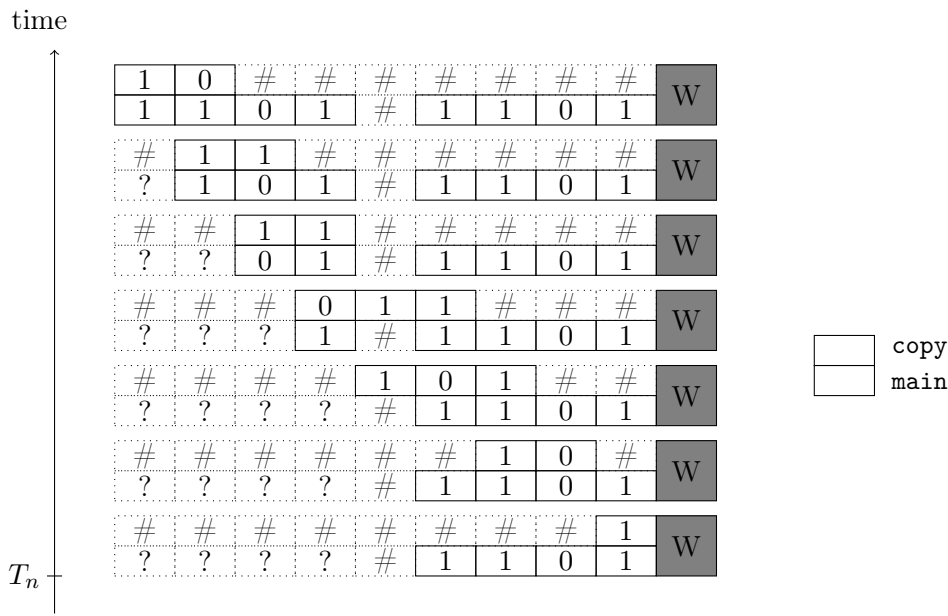
\begin{figure}[!ht]
\begin{center}
 \begin{tikzpicture}
\foreach \x in {1,...,7}
{
  \filldraw[fill = black!50!white] (0,\x) rectangle (0.8,\x+0.8);
  \draw (0.4,\x+0.4) node {W};
}
\foreach \x in {1,...,7}
{
 \foreach \y in {0,...,3}
  {
   \draw (-0.8*\y,\x) rectangle (-0.8*\y-0.8,\x+0.4);
  }

    \draw (-0.4,\x+0.2) node {$1$};
    \draw (-1.2,\x+0.2) node {$0$};
    \draw (-2,\x+0.2) node {$1$};
    \draw (-2.8,\x+0.2) node {$1$};
}
\foreach \x/\y in
{1/1,2/3,3/3,3/5,4/5,4/4,3.6/6,4.6/6,5/6,5/7,6/7,6/8,5.6/6,5.6/8,6.6/6,6.6/8,
6.6/9,7/9}
{
   \draw (-0.8*\y,\x+0.4) rectangle (-0.8*\y+0.8,\x+0.8);
   \draw (-0.8*\y+0.4,\x+0.6) node {$1$};
}
\foreach \x/\y in {2/2,3/4,4/6,4.6/7,5.6/7,6.6/7,7/8}
{
   \draw (-0.8*\y,\x+0.4) rectangle (-0.8*\y+0.8,\x+0.8);
   \draw (-0.8*\y+0.4,\x+0.6) node {$0$};
}
%\draw[red,very thick] (-4,0.8) -- (-4,8);
\foreach \x/\y in
{1/4,1.4/1,1.4/2,1.4/3,1.4/4,1.4/5,1.4/6,1.4/7,1.4/8,2/4,2.4/0,2.4/3,2.4/4,2.4/5
,2.4/6,2.4/7,2.4/8,3/4,3.4/0,3.4/1,3.4/5,3.4/6,3.4/7,3.4/8,  
4/4,4.4/0,4.4/1,4.4/2,4.4/6,4.4/7,4.4/8,5/4,5.4/0,5.4/1,5.4/2,5.4/3,5.4/4,5.4/7,
5.4/8,6/4,6.4/0,6.4/1,6.4/2,6.4/3,6.4/4,6.4/5,6.4/8,7/4,7.4/0,7.4/1,7.4/2,
   7.4/3,7.4/4,7.4/5,7.4/6}
{
   \draw[dotted] (-0.8*\y,\x) rectangle (-0.8*\y-0.8,\x+0.4);
   \draw (-0.8*\y-0.4,\x+0.2) node {\#};
}
\foreach \x/\y in
{1/5,1/6,1/7,1/8,2/5,2/6,2/7,2/8,3/5,3/6,3/7,3/8,4/6,4/7,4/8,5/7,5/8,6/8}
{
   \draw[dotted] (-0.8*\y,\x) rectangle (-0.8*\y-0.8,\x+0.4);
   \draw (-0.8*\y-0.4,\x+0.2) node {?};
}
\draw (2.1,3.4) rectangle (2.9,3.8) rectangle (2.1,4.2);
\draw (3.5,3.6) node {\small{\texttt{main}}} ;
\draw (3.5,4) node {\small{\texttt{copy}}} ;

\draw[->] (-8,0.5) -- (-8,8);
\draw (-8.2,8.4) node {time};
\draw (-8.1,1) -- (-7.9,1);
\draw (-8.4,1) node {$T_n$};
\end{tikzpicture}
\end{center}
\caption{Beginning of the copying process, with $w_n = 1101$. Only the layers
$\alph{copy}$ and $\alph{main}$ are represented.}
%The thick line is the leftmost limit of the time counter.}
\label{figure:copy}
\end{figure}

Uninitialised copying processes may write arbitrary words on the main layer, but they progress to the left at speed one
and are destroyed by the nearest wall in this direction.

\subsection{Merging}\label{section:merge}

At time $T_n$, all segments of length $n$ are forced to merge with their left neighbour, so that the density of walls tends to 0.
This means that merging is performed at time $T_n$ between a segment larger than $n$ to the left, and any number
of consecutive segments of length $n$ to the right.
To determine the length of each segment, a signal is sent to the right and bounces off the right wall, and its return
time is measured.\bigskip

To do so, a \define{merging counter} of value $2n$ is initialised at time $T_{n-1}$ on the merge layer. The value of $n$
is copied from the first computing layer to the merge layer (with an additional 0 at the end), using an auxiliary state
$\copie$ (\define{copy}). This counter decrements at each step in a similar way as incrementing counters, except it uses
-1 as a negative carry. See Figure~\ref{figure:merge} for an example of this process.\bigskip

If the signal returns at or before the end of the decrementation, a symbol $\fusion$ (\define{merge}) is created on the merge
layer to indicate that the wall is to be destroyed at the next $T_n$; this is the only case where the copying process described above does not trigger. 
To sum up, \[\alph{merge} =\{-1,0,1,\fusion,\copie\}\times\{\rightarrow,\leftarrow\}\cup\{\#\}.\]

 \begin{figure}[!ht]
\begin{center}
\begin{tikzpicture}
\draw (-4.4,-0.4) -- (-4.4,5.6);
\draw[dotted] (-4.4,5.6) -- (-4.4,6.4);
\draw[->] (-4.4,6.4) -- (-4.4,8);
\draw (-4.7,8.2) node {time};
\draw (-4.5,0) -- (-4.3,0);
\draw (-4.75,0) node {$T_2$};
\draw (-4.5,7.2) -- (-4.3,7.2);
\draw (-4.75,7.2) node {$T_3$};
\foreach \x in {0,...,6,8}
{
  \filldraw[fill = black!50!white] (0,0.8*\x) rectangle (0.8,0.8*\x+0.8)
(3.2,0.8*\x) rectangle (4,0.8*\x+0.8);
  \draw (0.4,0.8*\x+0.4) node {W};
  \draw (3.6,0.8*\x+0.4) node {W};
}
  \filldraw[fill = black!50!white] (3.2,7.2) rectangle (4,8);
  \draw (3.6,7.6) node {W};

\foreach \x\y in {0/1,1/2,2/3}
 {
   \draw[very thick] (\y*0.8,\x*0.8) rectangle (\y*0.8+0.8,\x*0.8+0.8);
   \draw (\y*0.8+0.4,\x*0.8+0.4) node {$\rightarrow$};
   \draw[very thick] (\y*0.8,4-\x*0.8) rectangle (\y*0.8+0.8,4.8-\x*0.8);
   \draw (\y*0.8+0.4,4-\x*0.8+0.4) node {$\leftarrow$};
 }

\foreach \x\y in
{0/2,0/3,0/-3,0/-4,0/-5,1/1,1/3,1/-4,1/-5,2/1,2/2,2/-4,2/-5,3/1,3/2,3/-4,3/-5,
4/1,4/3,4/-4,4/-5,5/2,5/3,5/-3,5/-4,5/-5,
6/1,6/2,6/3,6/-2,6/-3,6/-4,6/-5,8/-2,8/-3,8/-4,8/-5,8/1,8/2,8/3,9/0,9/-1,9/-2,
9/-3,9/-4,9/-5,9/1,9/2,9/3}
{
  \draw[dotted] (\y*0.8,\x*0.8) rectangle (\y*0.8+0.8,\x*0.8+0.8);
  \draw (\y*0.8+0.4,\x*0.8+0.4) node {\#};
}

\foreach \x\y in {0/2,1/3}
{
   \draw[very thick] (-\y*0.8,\x*0.8) rectangle (-\y*0.8+0.8,\x*0.8+0.8);
   \draw (-\y*0.8+0.4,\x*0.8+0.4) node {C};
}
\foreach \x\y in {1/2,2/3,3/3,4/3,5/2}
{
   \draw[very thick] (-\y*0.8,\x*0.8) rectangle (-\y*0.8+0.8,\x*0.8+0.8);
   \draw (-\y*0.8+0.4,\x*0.8+0.4) node {$1$};
}
\foreach \x\y in {0/1,2/1,2/2,3/2,4/1}
{
   \draw[very thick] (-\y*0.8,\x*0.8) rectangle (-\y*0.8+0.8,\x*0.8+0.8);
   \draw (-\y*0.8+0.4,\x*0.8+0.4) node {$0$};
}
\foreach \x\y in {1/1,3/1,4/2,5/1}
{
   \draw[very thick] (-\y*0.8,\x*0.8) rectangle (-\y*0.8+0.8,\x*0.8+0.8);
   \draw (-\y*0.8+0.4,\x*0.8+0.4) node {$-1$};
}
\draw[very thick] (-0.8,4.8) rectangle (0,5.6);
\draw (-0.4,5.2) node {$M$};
\draw[very thick] (-0.8,6.4) rectangle (0,7.2);
\draw (-0.4,6.8) node {$M$};
\draw[dotted] (-0.4,5.8) -- (-0.4,6.2) (0.4,5.8) -- (0.4,6.2) (3.6,5.8) --
(3.6,6.2);
\end{tikzpicture}
\end{center}
\caption{Determination of the length of a segment. Here the right segment is of length $3$ and the two segments merge
at time $T_3$. Only the merging layer is represented, with the counter of the right segment omitted for clarity.}
\label{figure:merge}
\end{figure}
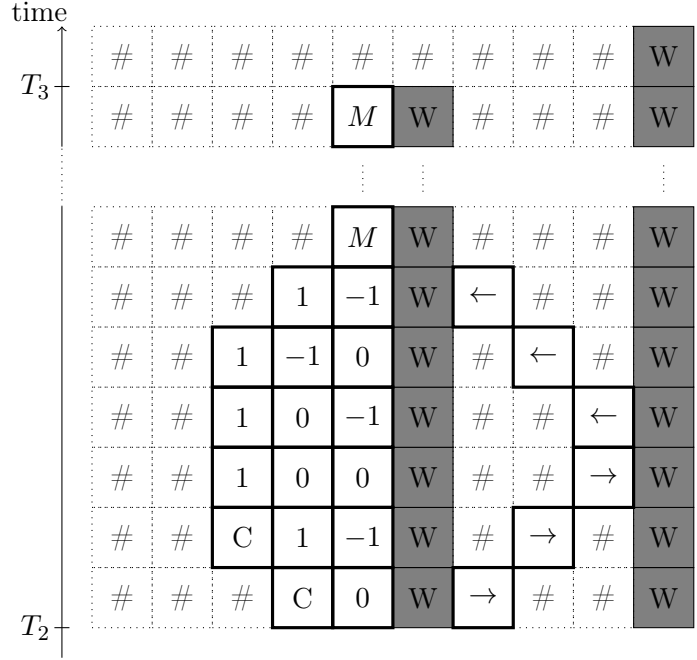

\begin{fact}\label{fact:erasingtime}
 Left walls of segments of length $\ell$ are erased at time $T'_\ell = \min (T_\ell, 2^\ell+\ell)$.
\end{fact}

\begin{proof} Except for the situation described above, the only other way for an initialised wall to be erased is a
time counter attached to an initialised wall, see Facts~\ref{fact:detached} and \ref{fact:attached}. A redundant binary
counter whose initial value is 0 reaches length $\ell$ at time $2^\ell+\ell$, the second term come from the carry propagation.
\end{proof}

Uninitialised merging counters are destroyed in exactly the same way as uninitialised time counters. To prevent uninitialised merging signals 
from disturbing a merging process, any right merging signal $\rightarrow$ erase incoming left merging signals $\leftarrow$. Merging signals 
arriving to a wall outside of a merging process is simply ignored and destroyed.

\subsection{Correctness of the construction}\label{section:correction}

To sum up, we have two time sequences $(T_n)_{n\in\N}$ and $(T'_n)_{n\in\N}$ such that:
\begin{itemize}
 \itemsep0em
 \item At time $T_n$, the computation of $w_n$ is finished and the copy starts;
 \item At time $T'_n$, the segments of length $n$ merge with their left neighbour.
\end{itemize}
Furthermore, those sequences are equal for $n$ large enough.

The computation, copy and merging processes described in the previous section have to be performed between time $T_n$ and time $T_{n+1}$, which
requires that the segments are not too large. In this section, we control the length of segments at
time $T_n$.
\begin{proposition}\label{Equiv-tn}
 $T_n = \Theta(\lfloor\sqrt n\rfloor q^{\lfloor\sqrt n\rfloor})$ where $q$ is defined in
Section~\ref{section:computation}.
\end{proposition}

\begin{proof}
$T_n = \sum_{k=1}^n T_k-T_{k-1}$. Since $T_{k+1} - T_k = q^{\lfloor\sqrt k\rfloor}$, and:
\[(2\lfloor\sqrt{n}\rfloor-1)q^{\lfloor\sqrt{n}\rfloor-1}\leq \sum_{k=1}^{\lfloor\sqrt n\rfloor-1} (2k+1)q^k \leq
\sum_{k=1}^n q^{\lfloor\sqrt k\rfloor}\leq \sum_{k=1}^{\lfloor\sqrt n\rfloor} (2k+1)q^k \leq (2\lfloor\sqrt
n\rfloor+1)q^{\lfloor\sqrt n\rfloor+1},\]
the proposition follows.
\end{proof}

\subsubsection{Acceptable segments}

\begin{definition}
Denote:
\begin{align*}\Gamma^t_{[i,j]} =& \left\{x\in\az\ |\ [i,j]\textrm{ is a segment of }F^t(x)\right\}\\
\Gamma^t_{l,k} =&\left\{x\in\az: [0,l]\text{ is included in a segment of }F^t(x)\text{ of length }k\right\}\\
=& \bigsqcup_{i=-k+\ell+1}^{0}\Gamma^t_{[i,i+k+1]}\qquad\mbox{(disjoint union)}\\
\textrm{ and }\quad \Gamma^t_{l,\geq k}=&\bigsqcup_{i\geq k} \Gamma^t_{l,i}.\end{align*}
\end{definition}

\begin{proposition}[Lower bound]\label{prop:LargerSegment}~

Let $\mu\in\Mergfull(\az)$. For all $l\in\N$, one has $\mu(\Gamma^{T_n}_{l,\geq n})
\underset{n\to\infty}{\longrightarrow} 1.$
\end{proposition}
\begin{proof}
$T_n=T'_n$ for $n$ large enough, so we do the proof for $T'_n$. Since $\mu$ has full support,

\[\mu\left(x\in\A^\Z:x_0=\start \textrm{ and }x_i\ne\start \textrm{ for all }i\in \{-2, -1, 1,2,\dots,n\}\right)\neq0.\]
By $\s$-ergodicity of $\mu$, segments of length larger than $n$ appear at time $0$ in $\mu$-almost all configurations, and those
segments survive up to time $T'_n$ by construction. In particular, $\F^{T'_n}\mu([\start])\neq 0$.

By $\s$-ergodicity of $\F^{T'_n}\mu$, the cell 0 is $\mu$-almost surely included in some segment at time $T'_n$, and this segment has length larger than $n$ by definition of $T'_n$. By $\s$-invariance, the probability that $[0,l]$ crosses a border of the segment tends to 0 as $n$ tends to infinity.
\end{proof}

\begin{definition}\label{def:AcceptableSegment}
Let $x\in\az$, $[i,j]$ a segment at time $t\in[T_n,T_{n+1}]$. It is \define{acceptable} if $j-i-1\leq K_n =
\sqrt{T_{n+1}-T_n}$. For $n$ large enough, $K_n=q^{\frac{\lfloor\sqrt n\rfloor}{2}}$.
\end{definition}

\begin{proposition}[Upper bound]\label{prop:AllAcceptable}
Let $\mu\in\Mmixfull(\az)$. One has $\mu(\Gamma^{T_n}_{l,\geq K_n})\underset{n\to\infty}{\longrightarrow} 0$, that is to
say:
\[\mu(\{x\in\az : [0,l]\textrm{ is in an acceptable segment of }F^t(x)\})\underset{t\to\infty}{\longrightarrow} 1\]
and the rate of convergence is exponential.
\end{proposition}

\begin{proof}
Again, $T_n=T'_n$ for $n$ large enough, so we do the proof for $T'_n$.
Any segment at time $T'_n$ corresponds to a segment larger than $n$ merged with $0$ or more
consecutive segments of length $n$ at time $T'_{n-1}$ (only the left wall of segments of size $n$ are destroyed at time
$T'_n$). See Figure~\ref{figure:AllAcceptable} for an illustration of this decomposition.
Therefore we define: 
\[\Delta^t_{n,\alpha}=\{x\in\az : \textnormal{starting from $0$ there is a strip of $\alpha$
consecutive segments of size $n$ in } F^t(x)\}.\]

First we bound the value of $\mu(\Delta^t_{n,\alpha})$. For any $m>0$, by considering one symbol out
of every $m$:
\begin{align}
\mu\left(\Delta^t_{n,\alpha}\right)
&\leq \mu\left(\bigcap_{i=0}^\alpha\s^{in}\left(\left[{\start}\right]\right) \right)\notag\\
&\leq \mu\left(\bigcap_{i=0}^{\lfloor\frac \alpha m\rfloor}\s^{in\cdot m}\left(\left[{\start}\right]\right)
\right)\notag\\
&\leq (1+\psi_\mu(mn))^{\lfloor\frac \alpha m\rfloor}\mu\left(\left[{\start}\right]\right)^{\lfloor\frac
\alpha m\rfloor+1},\label{acc2}
\end{align}
where $\psi_\mu$ are the weak mixing coefficients of $\mu$ as defined in Section~\ref{section:dynamical}.\bigskip

Now take $x$ such that $[0,l]$ is included in a segment longer than $k$ at time $T'_n$. As we said before, this segment is issued from the merging of one segment with $0$ or more segments of length $n-1$ at time $T'_{n-1}$. Take any $L>2n$ and distinguish the two following 
cases concerning the segments at time $T'_{n-1}$ it is issued from:
\begin{itemize}
 \item There were less than $\left\lfloor \frac Ln\right\rfloor$ segments of length $n$: then the other segment
 is larger than $k-L$. By shifting the configuration by $L-l$ cells at most, we can ensure that $[0,l]$ is included in
this segment at time $T'_{n-1}$.
 \item There were more than $\left\lfloor \frac Ln\right\rfloor$ segments of length $n$. Therefore there is a strip of
$\left\lfloor \frac Ln\right\rfloor$ segments of length $n$ starting at some $j\in[-k,k]$.
\end{itemize}
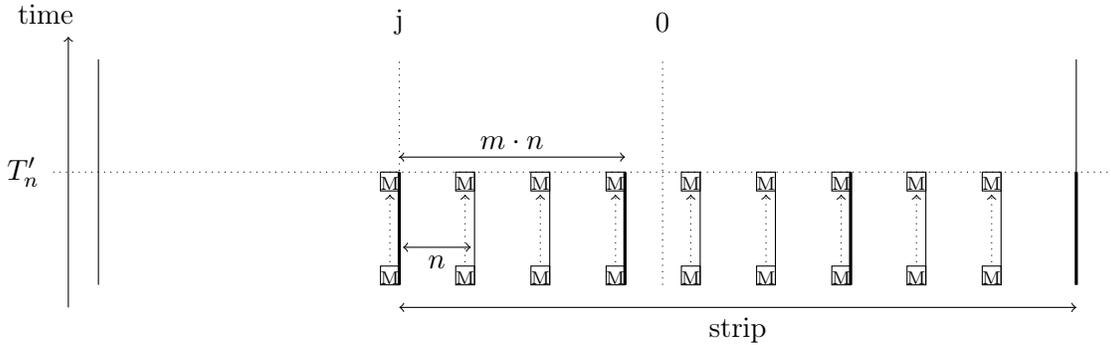
\begin{figure}[!ht]
 \begin{center}
  \begin{tikzpicture}
   \draw (0,0) -- (0,3) (13,0) -- (13,3);
   \foreach \x in {4,...,12}
   {
   \draw (\x,0) -- (\x,1.5);

   \draw (\x-.25,0) rectangle (\x,0.25);
   \draw (\x-.125,0.1) node {\tiny M};
   \draw[->, dotted] (\x-.125,0.3) -- (\x-.125,1.2);
   \draw (\x-.25,1.25) rectangle (\x,1.5);
   \draw (\x-.125,1.35) node {\tiny M};
   }
   \foreach \x in {4,7,10,13}
   {
   \draw[very thick] (\x,0) -- (\x,1.5);
   }
   \draw [<->] (4,-0.3) -- (13,-0.3);
   \draw (8.5,-0.6) node {strip};
   \draw [<->] (4.05,0.5) -- (4.95,0.5);
   \draw [<->] (4,1.7) -- (7,1.7);
   \draw (5.5,1.9) node {$m\cdot n$};
   \draw (4.5,0.3) node {$n$};
   \draw[dotted] (-0.6,1.5) -- (13.5,1.5);
   \draw[->] (-0.4,-0.3) -- (-0.4,3.3);
   \draw (-0.7,3.6) node {time};
   \draw (-1,1.5) node {$T'_n$};
   \draw[dotted] (7.5,0) -- (7.5,3);
   \draw (7.5,3.5) node {0};
   \draw[dotted] (4,0) -- (4,3);
   \draw (4,3.5) node {j};

  \end{tikzpicture}
  \caption{Illustration of the proof of Proposition~\ref{prop:AllAcceptable} with $\alpha=9$ and
$m=3$.}\label{figure:AllAcceptable}
 \end{center}
\end{figure}
In other words,
\[\Gamma^{T'_n}_{l,\geq k} \subset \bigcup_{i=-L+l}^0\s^i\left(\Gamma^{T'_{n-1}}_{l,\geq k-L}\right) \cup
\bigcup_{j=-k+1}^{k-1} \s^j\left(\Delta^{T'_{n-1}}_{n,\left\lfloor \frac Ln\right\rfloor}\right).\]
From which it follows:
\begin{equation}\label{acc1}
\mu\left(\Gamma^{T'_n}_{l,\geq k}\right) \leq L\mu\left(\Gamma^{T'_{n-1}}_{l,\geq k-L}\right)+
2k\mu\left(\Delta^{T'_{n-1}}_{n,\left\lfloor \frac Ln\right\rfloor}\right).
\end{equation}
Now take an arbitrary $n_0>0$ and a constant $M\geq n_0$. For any $n\leq n_0$ and $k\geq L$, Using (\ref{acc2}) with $m = \left\lceil \frac Mn\right\rceil$ inside equation (\ref{acc1}) yields:
\begin{align*}
\mu\left(\Gamma^{T'_n}_{l,\geq k}\right)& \leq L\mu\left(\Gamma^{T'_{n-1}}_{l,\geq k-L}\right)+
2k\left[1+\psi_\mu\left(n\cdot \left\lceil \frac Mn\right\rceil\right)\right]^{\frac
{L}{M}}\mu\left(\left[{\start}\right]\right)^{\frac{L}{M}+1}\\
& \leq L\mu\left(\Gamma^{T'_{n-1}}_{l,\geq k-L}\right)+ 2k\left[(1+\psi_\mu(M))
\mu\left(\left[{\start}\right]\right)\right]^{\frac{L}{M}}
\end{align*}
Applying this equation inductively, and assuming $k\geq n_0L$, we obtain:
\begin{align}\label{acc3}
\mu\left(\Gamma^{T'_{n_0}}_{l,\geq k}\right) &\leq L^{n_0}\mu\left(\Gamma^{0}_{l,\geq
k-n_0L}\right)+2kn_0\left[(1+\psi_\mu(M))\mu\left(\left[{\start}\right] \right)\right]^{\frac{L}{M}}
\end{align}
For the first component of the right-hand term, we have:
\begin{align*}
\mu\left(\Gamma^{0}_{l,\geq k-n_0L}(x)\right)&\leq \mu\left(\az\smallsetminus\bigcap_{j=-k+n_0L}^{-1}
\bigcup_{i=0}^{k-n_0L}\left[{\start}\right]_{j+i}\right) \\
&\leq \mu\left(\bigcup_{j=-k+n_0L}^{-1}\bigcap_{i=0}^{\lfloor\frac{k-n_0L}{n_0}\rfloor}
\left[{\A\backslash\start}\right]_{j+in_0}\right) \\
&\leq(k-n_0L)(1+\psi_\mu(n_0))^{\lfloor\frac{k-n_0L}{n_0}\rfloor}\mu\left(\left[{\A\backslash\start}\right]\right)^{
\lfloor\frac {k-n_0L} {n_0}\rfloor+1}
\end{align*}
the second line being obtained by considering one symbol out of every $n_0$. To conclude, we fix the values $M = n_0$, 
$L = n_0^2\sqrt n_0$, and $k =K_{n_0} = \sqrt{T_{n_0+1}-T_{n_0}}$. Since $\psi_\mu(n) \to 0$ and Equation (\ref{acc3}) holds for any $n_0$, we have $\mu(\Gamma^{T'_n}_{l,\geq K_n})\underset{n\to\infty}{\longrightarrow}0$ and the rate of convergence is exponential.
\end{proof}

\subsubsection{Density of auxiliary states}\label{section:auxiliary}

By auxiliary state, we mean any element of $\A\backslash \B$, that is to say $\start$, $\wall$ 
and any element of $\A$ which is not of the form $(b, \#, \#, \#, \#, \#)$.

\begin{proposition}\label{prop:AcceptableIsFormatted}
 For $t$ large enough, an acceptable segment is formatted and contains only initialised processes.
\end{proposition}

\begin{proof}
In a segment of length $k$, Fact~\ref{fact:formattime} ensures that the segment is formatted if $t\geq k(1+\log k)$. All remaining
uninitialised processes may take up to $k$ more steps to be erased.

When $T_n\leq t<T_{n+1}$, for an acceptable segment of length $k$, we have $k(2+\log k) \leq K_n(2+\log(K_n)) = o(T_n)$
by Proposition~\ref{Equiv-tn}. Taking $n$ large enough, we conclude.
\end{proof}

\begin{proposition}\label{prop:copy}
Let $\mu\in\Mergfull(\az)$ and $u\in\B^{[0,\ell]}$ for some fixed $\ell$. For a given length $k$ such that $n+1\leq
k\leq K_n$, we have:
\begin{itemize}
\item If $t\in[T_n+k,T_{n+1}]$,
\[\left|\mu\left(F^{-t}([u]_0)\ |\ \Gamma^{T_n}_{\ell,k}\right)-\meas{w_n}([u])\right| = O\left(\frac
1{\sqrt{n}}\right);\]
\item If $t\in[T_n,T_n+k]$, \[\left|\mu\left(F^{-t}([u]_0)\ |\
\Gamma^{T_n}_{\ell,k}\right)-\left(\frac{k-(t-T_n)}{k}
\meas{w_{n-1}}([u])+\frac{t-T_n}{k}\meas{w_n}([u])\right)\right| = O\left(\frac 1{\sqrt{n}}\right).\]
\end{itemize}
\end{proposition}

\begin{figure}[!ht]
\begin{center}
 \begin{tikzpicture}
 \fill[black!20] (0.6,0) -- (8.4,0) -- (8.4,0.6) -- (3.6,5.4) -- (0.6,5.4) -- cycle;
 \draw[red, very thick] (8.4,0.6) -- (3.6,5.4);
 \draw (2.8,1.8) node {$w_{n-1}$};
 \draw (5.7,4.5) node {$w_n$};
 \foreach \x in {0,14}
 {
 \foreach \y in {0,...,8}
 {
 \filldraw[fill = black!50!] (0.6*\x,0.6*\y) rectangle (0.6*\x+.6,0.6*\y+.6);
 \draw (0.6*\x+0.3,0.6*\y+0.3) node {W};
 }
 }
 \filldraw[fill = black!50!] (0.6*8,0) rectangle (0.6*8+.6,.6);
 \draw (0.6*8+0.3,0.3) node {W};
 \fill[pattern = north east lines] (0.6*8,0.6) -- (0.6*8+.6,0.6) -- (0.6*8+.6,3.6) -- (0.6*8,4.2) -- cycle;
 \filldraw[fill = black!50!] (0.6*11,0) rectangle (0.6*11+.6,.6);
 \draw (0.6*11+0.3,0.3) node {W};
 \fill[pattern = north east lines] (0.6*11,0.6) -- (0.6*11+.6,0.6) -- (0.6*11+.6,1.8) -- (0.6*11,2.4) -- cycle;
 \end{tikzpicture}
\caption{Illustration of Proposition~\ref{prop:copy}. The output is not correctly written in dashed areas because of the
destruction of a wall.}
\end{center}
\end{figure}
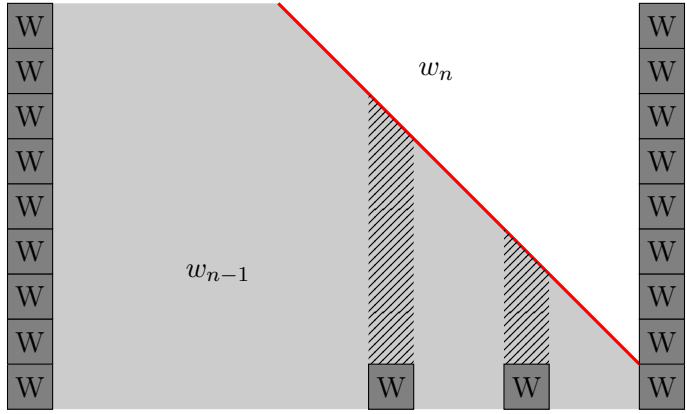

\begin{proof}
%\[\mu\left(\Gamma^{T_n}_{\ell,k}\right) = \bigsqcup_{i=-k+\ell+1}^{0}\Gamma^{T_n}_{[i,i+k+1]}
%= \bigsqcup_{i=0}^{k-\ell-1}\s^i\left(\Gamma^{T_n}_{[-1,k]}\right)\quad\quad\tx{(disjoint union)}.\]
Take $x\in\Gamma^{T_n}_{[-1,k]}$. Since a segment of length $k$ with $n+1\leq k\leq K_n$ is acceptable, it is formatted, and any uninitialised symbol has been destroyed. Since $|w_n| = O(\sqrt n)$ (the length of the output is smaller than the computing space), the copying process uses $O(\sqrt n)$ auxiliary cells.

\paragraph{First point:}\ The tail of the copying process progresses at speed one, so at time $T_n+k$ the copy of $w_n$ is finished, 
and until time $T_{n+1}$ the segment only contains copies of $w_n$ except for the time
counter, computation and merging counter area ($O(\sqrt n)$ cells) and a merging signal (one cell).\bigskip

Therefore for all $x\in\Gamma^{T_n}_{[-1,k]}$, one has $\left|\freq (u, F^t(x)_{[0,k-1]}) - \meas{w_n}([u])\right| =
\frac {O(\sqrt n)}k = O\left(\frac 1{\sqrt{n}}\right)$, taking into account the last copy of $w_n$ in the segment which
can be incomplete ($|w_n|\leq \sqrt n$), and since $k\geq n$. Thus we have: 
\[\left|\frac 1k\sum_{i=0}^{k-1}\mu\left(F^{-t}([u]_i)\ |\ \Gamma^{T_n}_{[-1,k]}\right)- \meas{w_n}([u])\right| = O\left(\frac 1{\sqrt
n}\right).\]
To conclude,
\begin{align*}\mu\left(F^{-t}([u]_0)\ |\ \Gamma^{T_n}_{\ell,k}\right) =& \sum_{i=-k+\ell}^{-1} \mu\left(F^{-t}([u]_0)\
|\
\Gamma^{T_n}_{[-1-i,k-i]}\right)\cdot \mu\left(\Gamma^{T_n}_{[-1-i,k-i]}\ |\ \Gamma^{T_n}_{\ell,k}\right)\\ 
=&\frac{1}{k-\ell} \sum_{i=1}^{k-\ell} \mu\left(F^{-t}([u]_0)\ |\ \Gamma^{T_n}_{[i-1,i+k]}\right) \\
=&\frac{1}{k-\ell} \sum_{i=1}^{k-\ell} \mu\left(F^{-t}([u]_i)\ |\ \Gamma^{T_n}_{[-1,k]}\right)
\end{align*}
where the last two lines are by $\s$-invariance of $\mu$. Since each term is at distance $O\left(\frac 1{\sqrt n}\right)$ of $\meas{w_n}([u])$, the result follows.

\paragraph{Second point:}\ When $t\in[T_n,T_n+k]$, the copy is still taking place, with $t-T_n$ cells containing copies
of $w_n$ and the rest containing copies of $w_{n-1}$, except for except for $O(\sqrt n)$ various auxiliary states, 
and possibly defects when a wall has been destroyed at time $T_n$ (there are at most $\frac kn$ of
them). Therefore
\[\left|\freq\left(u,F^t(x)_{[0,k-1]}\right)-\left(\frac{k-(t-T_n)}{k}\meas{w_{n-1}}([u])+\frac{t-T_n}{k} \meas{w_{n}}([u]
)\right)\right| = \frac 1k O(\sqrt n)\cdot \frac kn= O\left(\frac 1{\sqrt{n}}\right),\]
since $k\geq n$. Using the same reasoning as the
first point, we conclude.
\end{proof}

\subsubsection{Proof of Theorem~\ref{MainTheorem} - first point}\label{section:FirstPoint}

We prove the following: for a given computable sequence of words $(w_n)_{n\in\N}$, the CA $F$ that we described above satisfies that for any measure $\mu \in \Mmixfull(\az)$, $\V(F,\mu) = \V((w_n)_{n\in\N})$.\bigskip

Let $\mu\in\Mmixfull(\az)$ and $u\in\B^{[0,\ell]}$. By
Propositions~\ref{prop:LargerSegment} and \ref{prop:AllAcceptable}, $\mu\left(\bigcup_{k=n+1}^{K_n}
\Gamma^{T'_n}_{\ell,k}\right)\underset{n\to\infty}{\longrightarrow} 1$ exponentially fast, 
and $\Gamma^t_{\ell,k} = \Gamma^{T'_n}_{\ell,k}$ for $t\in[T'_n, T'_{n+1}-1]$. Therefore:
\[\exists C>0, \max_{T_n\leq t<T_{n+1}}\left|F^t_{\ast}\mu([u]) - \sum_{k=n+1}^{K_n}
\mu\left(F^{-t}([u])|\Gamma^t_{\ell,k}\right)\mu\left(\Gamma^t_{\ell,k}\right)\right| = O\left(e^{-Cn}\right).\]
Take $n$ large enough that $T_n=T'_n$. By Proposition~\ref{prop:copy},
\begin{align*}
\max_{T'_n\leq
t<T'_{n+1}}\left|\F^t\mu([u])-\sum_{k=n+1}^{K_n}\mu(\Gamma^{T_n}_{\ell,k})\right.&\left(\max\left(0,\frac{k-(t-T_n)}{k}
\right)\right.\meas{w_{n-1}}([u])\\
&+\left.\left.\min\left(1,\frac{t-T_n}k\right)\meas{w_n}([u])\right)\right| = O\left(\frac 1{\sqrt n}\right).
\end{align*}
Let $f_n$ be the piecewise affine function defined by:
\[\begin{array}{lccl}f_n:&[T_n,T_{n+1}]&\longrightarrow&[0,1]\\
&t&\longmapsto& \displaystyle\sum_{k=n+1}^{K_n}\min\left(1,\frac{t-T_n}k\right)\mu\left(\Gamma^{T_n}_{\ell,k} \right)
+ \frac{t-T_n}{T_{n+1}-T_n}\mu\left(\Gamma^{T_n}_{\ell,>K_n}\right).\end{array}\]
The second term is chosen so that $f_n(T_n) = 0$ and $f_n(T_{n+1}) = 1$, but it converges to 0 exponentially fast and thus
does not affect the equation by more than $O\left(\frac 1{\sqrt{n}}\right)$. Therefore:

\[\max_{T_n\leq t<T_{n+1}}\left|\F^t\mu([u])-\left(f_n(t)\meas{w_n}([u])+(1-f_n(t))\meas{w_{n-1}}([u])\right)\right|
= O\left(\frac 1{\sqrt n}\right).\]
\[\max_{T_n\leq t<T_{n+1}}\dm\left(F^t_{\ast}\mu,\left[\meas{w_{n-1}},\meas{w_n}\right]\right) =
O\left(\frac 1{\sqrt n}\right),\]
so $\V(F,\mu)\subset\V((w_n)_{n\in\N})$. Since $f_n$ is $\frac 1n$-Lipschitz on $[T_n,T_{n+1}]$,
any $\nu\in\left[\meas{w_{n-1}},\meas{w_n}\right]$ is at distance at most $\frac 1n$ of an element of the form
$\left(f_n(t)\meas{w_{n}}+(1-f_n(t))\meas{w_{n-1}}\right)$ for $T_n\leq t<T_{n+1}$.

We conclude that $\V(F,\mu)=\V((w_n)_{n\in\N})$.

\paragraph{Rate of convergence}~
For clarity, assume that $w_n$ is computable in space $O(\sqrt{n})$ by repeating elements if necessary.

By Proposition~\ref{Equiv-tn} we have $T_n = \Theta(\lfloor\sqrt n\rfloor q^{\lfloor\sqrt n\rfloor})$ so, writing
$n(t)$ the current value of $n$ at time $t$, we have $n(t)=\Theta(\log(t)^2)$ and $O\left(\frac 1{\sqrt{n(t)}}\right) =
O\left(\frac 1{\log t}\right)$.

We find that the rate of convergence is:
 \begin{align*}\dm\left(\F^t\mu,\V\left((w_n)_{n\in\N}\right)\right)\leq&\ 
\dm\left(\F^t\mu, \left[\meas{w_{n(t)-1}},\meas{w_n(t)}\right]\right)+
\sup_{\nu\in\left[\meas{w_{n(t)-1}},\meas{w_{n(t)}}\right]}\dm\left(\nu, \V\left((w_n)_{n\in\N}\right)\right)
 \\
=&\ O\left(\frac{1}{\log(t)}\right)+\sup\left\{\dm\left(\nu,
\V\left((w_n)_{n\in\N}\right)\right):\nu\in\bigcup_{n\geq
n(t)}\left[\meas{w_n},\meas{w_{n+1}}\right]\right\},\end{align*}

by the last proof. 

\subsubsection{Proof of Theorem~\ref{MainTheorem} - second point}\label{section:SecondPoint} 

Now we treat the case where $\V((w_i)_{i\in\N})=\{\nu\}$. Let $F$ be the cellular automaton associated with this sequence as
described above, and consider $\mu\in\Mergfull(\az)$. Since $\mu$ is not assumed to be $\psi$-mixing,
Proposition~\ref{prop:AllAcceptable} does not apply, and there is no guarantee most segments are acceptable. 
However large segments are still rare; more precisely, $\mu(\Gamma^t_{0,\geq k})\underset{k\to\infty}{\longrightarrow} 0$ 
for all $t$ since all sets $\Gamma^t_{0,k}$ are disjoint.

\Claim{$\F^t\mu([\A\backslash \B]) \underset{t\to\infty}{\longrightarrow} 0$, i.e., the density of auxiliary
states tends to 0.}
\bclaimprf Suppose we are in an initial segment of length $k$. Detached time counters, Turing machines and
merging counters initially present are destroyed in less than $k$ steps. Similarly, left merging signals and copy auxiliary
states initially present progress at speed -1, so they are destroyed before time $k$. Any uninitialised wall is destroyed
after $k(1+\log k)$ steps at most, and any counter attached to it are destroyed after less than $k$ more steps. For all
those states, the probability of apparition after time $k(2+\log k)$ is less than $\mu(\Gamma_{0,\geq
k}^0)\underset{k\to\infty}\longrightarrow 0$.

At time $T'_n$, all segments are longer than $n$, so the density of initialised walls and initialised auxiliary states
inside each segment is $O\left(\frac {\sqrt{n}}n\right)$.

Only uninitialised formatting counters and right merging signals remain. Inside each segment, call \define{non-formatted area}
the interval between the initialised formatting counter of the left wall and the rightmost cell containing one of those
two states. At each step, this area decreases by one cell to its left but may grow by one cell to its right as a
counter or signal progresses. Notice that merging with other segments cannot increase this area since segments of length
$n$ at time $T_n$ are formatted (see Figure~\ref{figure:formatVsNonFormat}).

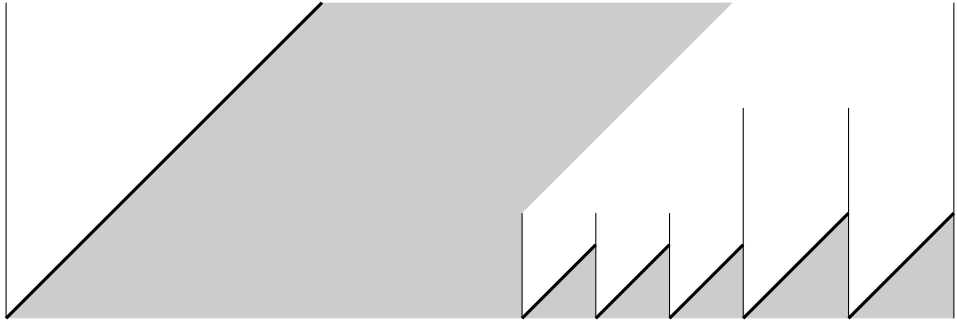
\begin{figure}[!ht]
 \begin{center}
  \begin{tikzpicture}[scale=1.4]
   \fill[black!20!white] (1,0) -- (4,3) -- (7.9,3) -- (5.9,1) -- (5.9,0) -- cycle;
   \fill[black!20!white] (5.9,0) -- (6.6,0.7) -- (6.6,0) -- cycle;
   \fill[black!20!white] (6.6,0) -- (7.3,0.7) -- (7.3,0) -- cycle;
   \fill[black!20!white] (7.3,0) -- (8,0.7) -- (8,0) -- cycle;
   \fill[black!20!white] (8,0) -- (9,1) -- (9,0) -- cycle;
   \fill[black!20!white] (9,0) -- (10,1) -- (10,0) -- cycle;
   \draw (1,0) -- (1,3);
   \draw[very thick] (1,0) -- (4,3);
   \draw (5.9,0) -- (5.9,1);
   \draw[very thick] (5.9,0) -- (6.6,0.7);
   \draw (6.6,0) -- (6.6,1);
   \draw[very thick] (6.6,0) -- (7.3,0.7);
   \draw (7.3,0) -- (7.3,1);
   \draw[very thick] (7.3,0) -- (8,0.7);
   \draw (8,0) -- (8,2);
   \draw[very thick] (8,0) -- (9,1);
   \draw (9,0) -- (9,2);
   \draw[very thick] (9,0) -- (10,1);
   \draw (10,0) -- (10,3);   
  \end{tikzpicture}
 \end{center}
\caption{Illustration of the last part of the proof of Claim 1. Slanted lines are formatting counters 
and grey areas are potentially non-formatted.}\label{figure:formatVsNonFormat}
\end{figure}

Therefore, a segment at time $T_n$ can contain a non-formatted area longer than $\sqrt{n}$ only if it is issued from a
segment longer than $\sqrt{n}$ initially. Other segments have a non-formatted area smaller than $\sqrt n$ for a length
larger than $n$.
By $\s$-invariance, \[\mu(\{x\in\az\ |\ x_0\textrm{ is in a non-formatted area}\}) \leq \frac {\sqrt n}n +
\mu\left(\Gamma^0_{0,\geq \sqrt n}\right) \underset{n\to\infty}\longrightarrow0.\]
Therefore, for $a\in\A\backslash\B$, we have $\F^t\mu([a])\underset{t\to\infty}\to 0$.
\eclaimprf

\Claim{For any $n\in\N$, $\dm\left(\F^t\mu, Conv\left((\meas{w_i})_{i\geq
n}\right)\right)\underset{t\to\infty}{\longrightarrow}0$, where $Conv(X)$ is the convex hull of the set $X$.}

\bclaimprf
Consider a segment of length $k$ at time $T_n$. At time $T_n+k$ the copying process for $w_n$ is finished, but
since the segment is not necessarily acceptable, other copying processes may have started in the meanwhile (see Figure~\ref{figure:ConvexHull}). Therefore, the segment contains:
\begin{itemize}
 \item auxiliary states, with negligible frequency;
 \item strips containing repeated copies of $w_n$, then $w_{n+1}, w_{n+2}$\dots separated by ongoing copy processes and the frequency of auxiliary copy states being negligible. 
\end{itemize}

Consider a segment of size $k$ at time $T_n$ in the positions $[i,i+k]$. At time $T_n+k$ it is filled with copies of $w_n$. When $t\geq T_n+k$, the  positions $[i,i+k]$ contain a succession of stripes containing $w_N,
w_{N+1},\dots$ with $N\geq n$ plus a negligible part of auxiliary states and defects. The strip containing $w_i$ is larger than $i$ since this word is produced in a segment of size larger than $i$. One deduces that $\mu\left(F^{-(T_n+t)}([u]_0)\ |\ \Gamma^{T_n}_{\ell,k}\right)$ is quite near of $Conv((\meas{w_i}([u]))_{i\geq n})$ for $t\geq k$. Since $\mu(\Gamma^{T_n}_{0,\geq k}) \underset{k\to\infty}{\longrightarrow} 0$, we have

\[\dm\left(\F^{T_n+t}\mu, Conv((\meas{w_i})_{i\geq n})\right) \underset{t\to\infty}\longrightarrow 0.\]
\eclaimprf

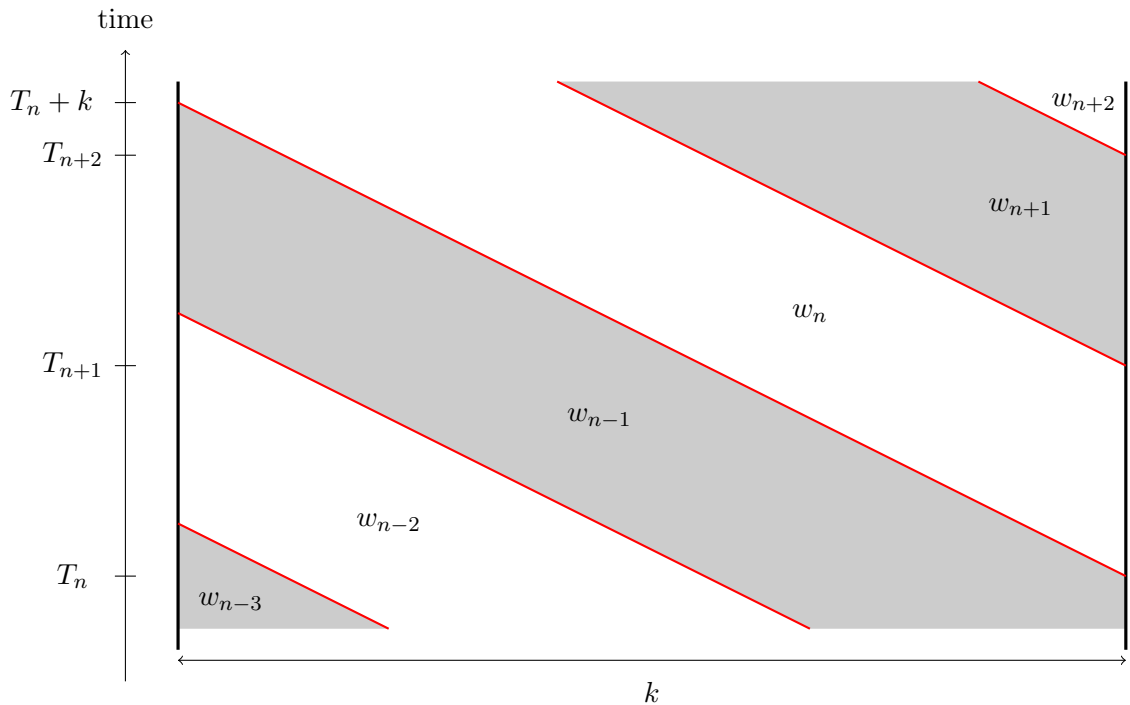
\begin{figure}[!ht]
 \begin{center}
  \begin{tikzpicture}[scale=1.4]
   \fill[black!20] (0,0) -- (0,1) -- (2,0) -- cycle (6,0) -- (9,0) -- (9,.5) -- (0,5) --
(0,3) -- cycle (9,2.5) -- (9,4.5) -- (7.6,5.2) -- (3.6,5.2) -- cycle;
   \draw[->] (-.5,-.5) -- (-.5,5.5);
   \draw (-.5, 5.8) node {time};
   \draw (-.6,.5) -- (-.4,.5) (-.6,2.5) -- (-.4,2.5) (-.6,4.5) -- (-.4,4.5) (-.6,5) -- (-.4,5);
   \draw (-1,.5) node {$T_n$};
   \draw (-1,2.5) node {$T_{n+1}$};
   \draw (-1,4.5) node {$T_{n+2}$};
   \draw (-1.2,5) node {$T_n+k$};
   \draw[<->] (0,-.3) -- (9,-.3);
   \draw (4.5,-.6) node {$k$};
   \draw[very thick] (0,-.2) -- (0,5.2) (9,-.2) -- (9,5.2);
   \draw[red,thick] (9,.5) -- (0,5) (2,0) -- (0,1) (6,0) -- (0,3) (9,2.5) --
(3.6,5.2) (9,4.5) -- (7.6,5.2);
   \draw (.5,.25) node {$w_{n-3}$};
   \draw (2,1) node {$w_{n-2}$};
   \draw (4,2) node {$w_{n-1}$};
   \draw (6,3) node {$w_n$};
   \draw (8,4) node {$w_{n+1}$};
   \draw (8.6,5) node {$w_{n+2}$};
  \end{tikzpicture}
 \end{center}
\caption{Illustration of Claim 2. When $t>T_n+k$, a segment of length $k$ is a succession of stripes containing $w_n,
w_{n+1},\dots$ plus a negligible part of auxiliary states and defects.}\label{figure:ConvexHull}
\end{figure}

The second point of the Theorem~\ref{MainTheorem} follows easily from Claim 2.

\begin{remark}
It does not follow from the last claim that the sequence $(\F^t\mu)$ is close to any of the $\meas{w_i}$ at any point,
which is the reason why the result holds
only for a single measure. Controlling the length of the segments as needed in the proof of the first point
requires $\psi$-mixing.
\end{remark}

\section{Removing the auxiliary states}\label{section:NoAux}

Before stating consequences of Theorem~\ref{MainTheorem}, we consider in this section the case where the
cellular automaton does not use auxiliary states, that is, $\A = \B$. A straightforward extension is
impossible: if $\nu$ is a full support measure, and $F : \az \to \az$ a cellular automaton such that $\F^t\mu\to\nu$ for some
initial measure $\mu$, then $F$ is a surjective automaton which leaves the uniform Bernoulli measure invariant. 
Therefore, starting from the uniform measure, $F$ can only reach the uniform Bernoulli measure. 

\bigskip

However, if the limit measure does not have full support, the previous results can be extended by using a word not
charged by the measure to encode the auxiliary states in some sense.

\begin{theorem}\label{MainTheoremNoAux}
Let $(w_n)_{n\in\N}$ be a uniformly computable sequence of words of $\B^\ast$, where $\B$ is a finite alphabet, and assume there
exists a word $u$ that does not appear as subwords in any of the $w_n$. Then there is a cellular automaton $F : \B^\Z\to
\B^\Z$ such that for any measure $\mu \in \Mmixfull(\B^\Z)$, $\V(F,\mu) = \V((w_n)_{n\in\N})$.
\end{theorem}

However, because of the destructive nature of the formatting counter in the construction, the proof in Section~\ref{section:SecondPoint}
cannot be adapted and we cannot weaken the hypothesis to $\mu \in \Mergfull(\B^\Z)$ when $\K$ is a singleton.

\begin{proof}
Let $\A\supset \B$ be the alphabet and $F$ the CA associated to the sequence $(w_n)_{n\in\N}$ by
Theorem~\ref{MainTheorem}. Our
aim is to provide an encoding of any configuration of $\az$ in $\bz$ and a cellular automaton $F'$ that behaves
similarly to $F$ after encoding.\bigskip

Denote $U_d \subset \B^d$ the set of words of length $d$ with prefix $u$, that do not contain $u$ as subword
(except at the first letter), and that do not end with a prefix of $u$. $\#U_d \underset{d\to\infty}\longrightarrow
\infty$, so for $d$ large enough, we can find an injection $\varphi : \A\backslash\B \to U_d$ (encoding the auxiliary
states), and we extend it by putting $\varphi = Id$ on $\B$. For a finite word, we define $\varphi(u_1\dots u_n) =
\varphi(u_1)\dots\varphi(u_n)$, and this can be naturally extended further to configurations $\Phi:\az\longmapsto \B^\Z$
by considering that $\varphi(a_0)$ starts on the column zero. Notice that this encoding is not $\s$-invariant.\bigskip

Let $\mathbf{T}\subset\az$ be the set of configurations such that the word $u$ does not appear on the main layer
($\mathbf{T}$ is a subshift of finite type). Since $u$ marks unambiguously the beginning of a word
of $\varphi(\A\backslash\B)$, the restriction $\Phi:\mathbf{T}\to\bz$ is injective.

We can define locally a decoding $\Psi:\Phi(\mathbf{T})\to\mathbf{T}$ such that $\Psi\circ\Phi=\textrm{Id}$, by looking
 $d$ cells to the right for occurrences of $u$.
If $u$ appears, we are an \define{output cell}, that is, the image by $\varphi$ of a single letter $b\in\B$ 
(corresponding to $(b,\#,\#,\#,\#,\#)$ for $b\in \B$ in the previous construction); otherwise, we belong in an \define{auxiliary cluster},
the image by $\varphi$ of a letter $A\backslash B$ that occupy $d$ cells while containing one letter of output. See
Figure~\ref{fig:encoding} for an example.\bigskip

\begin{figure}[!ht]
 \begin{tikzpicture}
\foreach \t in {3,4,5,13,14,15}
\filldraw[fill=black!20] (.7*\t,0) rectangle (.7*\t+0.7,.7);
\foreach \t in {3,13}
\draw[red,thick] (.7*\t, -.2) rectangle (.7*\t +5.6, .9);
\foreach \t/\x in {0/0, 1/0, 2/1, 3/1, 4/0, 5/1, 6/1, 7/1, 8/0, 9/0, 10/0, 11/1, 12/1, 13/1, 14/0, 15/1, 16/0, 17/0,
18/1, 19/0, 20/0, 21/1}
{
\draw (.7*\t,0) rectangle (.7*\t+0.7,.7);
\draw (.7*\t +.35, .35) node {\x};
}
\draw[->] (7.5,1) -- (7.5,1.8);
\draw (7.8, 1.4) node {$\Phi$};
\foreach \t\x in {0/0, 1/0, 2/1, 3/$a_1$, 4/1, 5/1, 6/$a_2$, 7/1}
{
\draw (4.9+.7*\t,2) rectangle (4.9+.7*\t+0.7,2.7);
\draw (4.9+.7*\t +.35, 2.35) node {\x};
}
\draw (4.6, 2.35) node {$\dots$};
\draw (10.9, 2.35) node {$\dots$};
 \end{tikzpicture}
\caption{Encoding of the auxiliary states with $u = 101$ and $d=3$. In this case $U_d \subset 101\cdot\A^3\cdot 00$. }
\label{fig:encoding}
\end{figure}
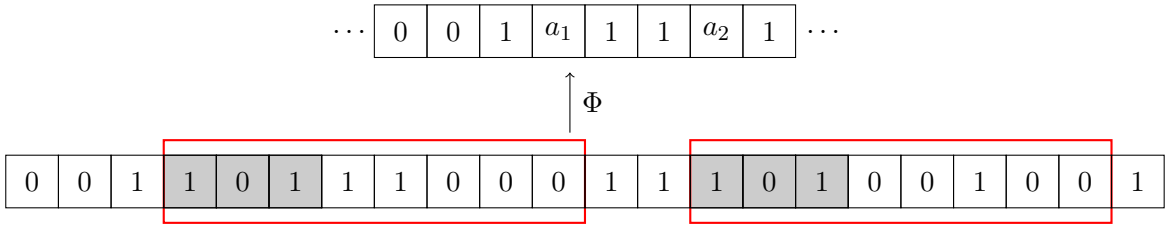

Intuitively, we want to build a cellular automaton that behaves similarly as the automaton defined
in Theorem~\ref{MainTheorem}, where elements $(b,\#,\#,\#,\#,\#)$ are represented by output cells
and all other elements by auxiliary clusters. However, $\Phi$ and $\Psi$ are not $\sigma$-invariant, so $\Phi\circ
F\circ\Psi$ is not a cellular automaton. Instead, we build manually a cellular automaton on $\B^\Z$ that behaves in roughly
the same way as $\Phi\circ F\circ \Psi$. \bigskip

Provided the neighbourhood is larger than $[-4d,4d]$, each cell can ``read" the
cluster in which it belongs, and the three clusters to its right and left.
At time 0, if a word $u$ is not the prefix of a word of $U_d$, it is replaced by a word $b^d$ and can never be
created again.  To avoid creating an auxiliary cluster by mistake, we fix to this purpose a letter $b\in\B$ such that
$b^d\notin U_d$. Similarly, auxiliary clusters that are destroyed for any reason leave behind them output $b$ cells.

 \begin{remark} For clarity, in all diagrams of this section, we suppose that $\B = \{0,1\}$, $d=3$ (it would be much
larger in a real implementation) and we represent auxiliary clusters as blocks with layers, instead of words from
$\B^d$. Also we fix $b=0$ in the definition above. \end{remark}

The different parts of the construction are modified in the following way.

\begin{itemize}
 \item $\start$ and $\wall$ clusters, time counters, and Turing machines have the same behaviour as in the previous
construction. However, since the counters take more space, it is necessary to erase $3d$ cells to the left and right of
each $\start$ cluster at time 0.

\begin{figure}[!ht]
\begin{center}
 \begin{tikzpicture}[scale=1.2]

\foreach \t in {0,...,5}
{
\filldraw[fill=black!20] (-7,.9*\t) rectangle (-9.1,.9*\t+0.7);
\draw (-8,.9*\t +.35) node {W};
}

\foreach \x\y\z\t in
{0/0/1/1,3/0/1/0,1/1/1/0,4/1/1/1,2/2/0/1,5/2/1/1,3/3/1/1,6/3/0/1,4/4/1/1,7/4/1/0
,5/5/1/0}
{
 \draw (-0.7*\x,.9*\y+0.35) rectangle (-0.7*\x-2.1,.9*\y+0.7);
 \draw (-0.7*\x-0.7,.9*\y) rectangle (-0.7*\x-1.4,.9*\y+0.35);
 \draw (-0.7*\x-1.05,.9*\y+0.525) node {\z};
 \draw (-0.7*\x-1.05,.9*\y+0.175) node {\t};
}

\foreach \x\y\z in
{6/0/?,7/0/?,8/0/?,9/0/?,0/1/1,7/1/?,8/1/?,9/1/?,0/2/1,1/2/0,8/2/?,9/2/?,0/3/1,
1/3/0,2/3/1,9/3/?,0/4/1,1/4/0,2/4/1,3/4/1,
0/5/1,1/5/0,2/5/1,3/5/1,4/5/1,8/5/0,9/5/0}
{
\draw (-0.7*\x,.9*\y) rectangle (-0.7*\x-0.7,.9*\y+0.7);
\draw (-0.7*\x-.35,.9*\y+.35) node {\z};
}

\draw (0.7,2.9) rectangle (2.8,3.25);
\draw (1.4,2.55) rectangle (2.1,2.9);
\draw (1.75,2.3) node {\small{\texttt{main}}} ;
\draw (1.75,3.5) node {\small{\texttt{copy}}} ;
\end{tikzpicture}
\end{center}
\caption{End of the copying process described in Figure~\ref{figure:copy}, copying the word 1101.}
\label{fig:NoAuxCopy}
\end{figure}
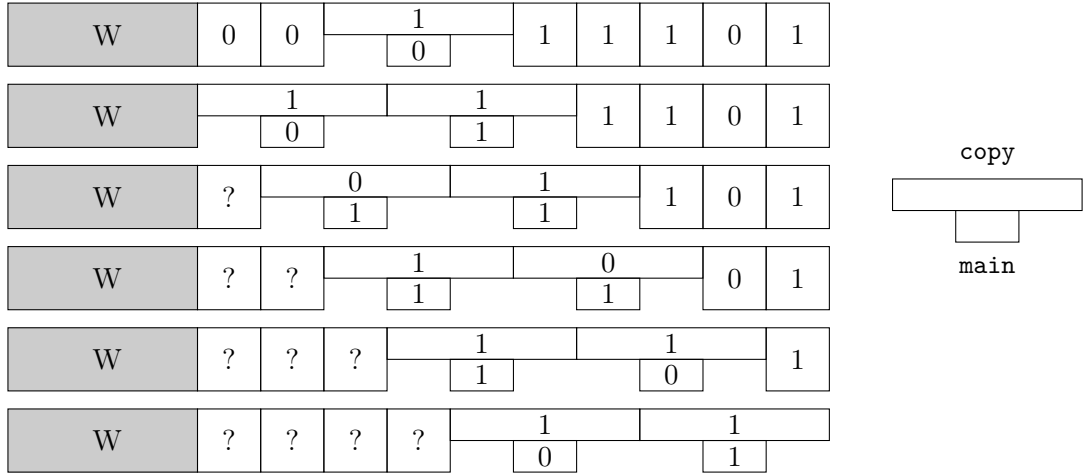

 \item The tail of copying processes progresses to the left at speed one, and behaves normally as long as it does not
meet another auxiliary state (see Figure~\ref{fig:NoAuxCopy}). When the process has finished the copy, it is destroyed and
leaves $b$ cells behind.

 \item Formatting counters progress to the right at speed $d$. This is too fast to keep information on the output
layer, so the counter leaves behind output cells $b$ defined above. Any other signal it meets (e.g. copying process or
length-measuring signal) is similarly erased.
\item When close to a time counter, it may happen that the formatting counter cannot progress by $d$ cells exactly (see
Figure~\ref{fig:FormatOffset}).  In this case, it is offset by less than $d$ cells, and formatting clusters separated by
small offsets in this way are still considered to be the same counter for the rule of the automaton. The subsequent
comparison process is unchanged.

\begin{figure}[!ht]
\begin{center}
 \begin{tikzpicture}[every text node part/.style={align=center},scale=1.2]
\foreach \x\y\z in
{0/0/\texttt{time},3/0/\texttt{time},10/0/\texttt{formatting},13/0/\texttt{formatting},0/1/\texttt{time},3/1/\texttt{time},0/2/\texttt{time},
7/1/\texttt{formatting},10/1/\texttt{formatting},
3/2/\texttt{time}\\\texttt{formatting},7/2/\texttt{formatting},0/3/\texttt{time}\\\texttt{formatting},3/3/\texttt{time}\\\texttt{formatting}}
{
 \draw (-0.8*\x,\y) rectangle (-0.8*\x-2.4,\y+0.8);
 \draw (-0.8*\x-1.2,\y+0.4) node {\z};
}

\foreach \x\y\z in {6/0,7/0,8/0,9/0,6/1,6/2,6/3}
{
\draw (-0.8*\x,\y) rectangle (-0.8*\x-0.8,\y+0.8);
\draw (-0.8*\x-.4,\y+.4) node[text width=.1cm] {?};
}

\foreach \x\y\z in {13/1,14/1,15/1,10/2,11/2,12/2,13/2,14/2,15/2,
7/3,8/3,9/3,10/3,11/3,12/3,13/3,14/3,15/3}
{
\draw (-0.8*\x,\y) rectangle (-0.8*\x-0.8,\y+0.8);
\draw (-0.8*\x-.4,\y+.4) node[text width=.1cm] {$0$};
}
\end{tikzpicture}
\end{center}
\caption{A formatting counter gets offset when entering the time counter area. Notice the auxiliary clusters being
replaced by output cells containing $b=0$.}
\label{fig:FormatOffset}
\end{figure}

\begin{figure}[!ht]
\begin{center}
 \begin{tikzpicture}[every text node part/.style={align=center},scale=1.2]
\foreach \x in {-1,...,6}
{
\filldraw[fill=black!20] (-9.1,0.9*\x) rectangle (-11.3,0.9*\x+0.7);
\draw (-10.1,0.9*\x+.35) node {W};
\filldraw[fill=black!20] (0,0.9*\x) rectangle (2.1,0.9*\x+0.7);
\draw (1.05,0.9*\x+.35) node {W};
}

\foreach \x\y in {0/-1,3/-1,0/0,3/0,0/1,3/2,3/3,0/4,0/5,3/5,0/6,3/6}
{
 \draw (-0.7*\x,0.9*\y) rectangle (-0.7*\x-2.1,0.9*\y+0.7);
 \draw (-0.7*\x-1.05,.9*\y+0.35) node {\texttt{time}};
}

\foreach \x\y in {10/-1, 7/0}
{
 \draw (-0.7*\x,0.9*\y) rectangle (-0.7*\x-2.1,0.9*\y+0.7);
 \draw[thick, ->] (-0.7*\x-1.6,0.9*\y+0.2) -- (-0.7*\x-0.5,\y+0.5);
}

\foreach \x\y in {6/5, 9/6}
{
 \draw (-0.7*\x,0.9*\y) rectangle (-0.7*\x-2.1,.9*\y+0.7);
 \draw[thick, ->] (-0.7*\x-0.5,.9*\y+0.2) -- (-0.7*\x-1.6,.9*\y+0.5);
}

\foreach \x\y in {3/1,0/2}
{
 \draw (-0.7*\x,.9*\y) rectangle (-0.7*\x-2.1,.9*\y+0.7);
 \draw (-0.7*\x-1.05,.9*\y+0.35) node {\texttt{time}};
 \draw[thick, ->] (-0.7*\x-1.6,.9*\y+0.2) -- (-0.7*\x-0.5,.9*\y+0.5);
}

\foreach \x\y in {0/3, 3/4}
{
  \draw (-0.7*\x,.9*\y) rectangle (-0.7*\x-2.1,.9*\y+0.7);
  \draw (-0.7*\x-1.05,.9*\y+0.35) node {\texttt{time}};
  \draw[thick, ->] (-0.7*\x-0.5,.9*\y+0.2) -- (-0.7*\x-1.6,.9*\y+0.6);
}

\foreach \x\y in
{6/-1,7/-1,8/-1,9/-1,6/0,10/0,11/0,12/0,6/1,10/1,11/1,12/1,6/2,10/2,11/2,12/2,
6/3,10/3,11/3,12/3,6/4,10/4,11/4,12/4,10/5,11/5,12/5,7/6,8/6,12/6}
{
\draw (-0.7*\x,.9*\y) rectangle (-0.7*\x-0.7,.9*\y+0.7);
\draw (-0.7*\x-.4,.9*\y+.35) node[text width=.1cm] {?};
}

\foreach \x\y in {7/1,8/1,9/1,7/2,8/2,9/2,7/3,8/3,9/3,7/4,8/4,9/4,9/5,6/6}
{
\draw (-0.7*\x,.9*\y) rectangle (-0.7*\x-0.7,.9*\y+0.7);
\draw (-0.7*\x-.4,.9*\y+.35) node[text width=.1cm] {0};
}
\end{tikzpicture}
\end{center}
\caption{Determination of length. Here $d=3,\,t_0 = 8$ and $o=1$, for a measured length of 13.}
\label{fig:LengthOffset}
\end{figure}

\item Merging signals which determine length of segments also progress at speed $d$. To avoid possible interactions with
copying processes (similarly to the case of formatting counters), the determination of length starts only after the copy
is finished. Thus a merging signal is only offset when entering the time counter area. After bouncing off the right
wall, it returns to the left wall where its offset can be measured. If it takes $t_0$ steps to return with an offset of
$\alpha$, then the segment has length $\frac {t_0}2\cdot d+\alpha$ (see Figure~\ref{fig:LengthOffset}). This value is
compared to $n$ and the rest of the process is not modified.

\end{itemize}
In this way, Propositions~\ref{prop:AllAcceptable} and \ref{prop:AcceptableIsFormatted} still hold. We can check that
at time $t$, with $T_n\leq t <T_{n+1}$, the copy process followed by the process of determination of length for segments
of size $n+1$ still take less than $T_{n+1}-T_n$ steps. Furthermore, the frequency of auxiliary states is multiplied by a fixed constant $d$.
Hence the proof in Section~\ref{section:FirstPoint} can be adapted, and the theorem follows.\end{proof}

\section{Problems solved with this construction}\label{section:RelatedProblem}

In this section, we use Theorem~\ref{MainTheorem} to solve various problems, starting with the characterisation of
reachable limit measures and connected $\mu$-limit measure sets.
We then consider the disconnected case, the case when auxiliary states are not allowed, Césàro mean convergence and
consequences of these results for the decidability of asymptotic properties of cellular automata.

\subsection{Characterisation of reachable \texorpdfstring{$\mu$}{mu}-limit measures set}\label{section:Characterization}

\subsubsection{The connected case}
Reciprocals of the computable obstructions described in Section~\ref{section:ComputableObstruction} follow directly from
Theorem~\ref{MainTheorem}. 

\begin{corollary}\label{cor:ConvUniqueMeasure}Let $\nu\in\Mscomp(\bz)$ be a limit-computable measure. There is an
alphabet $\A\supset \B$ and a cellular automaton
$F:\az\to\az$ such that for any $\mu \in \Mergfull(\az)$, one has $\F^t\mu\underset{t\to\infty}\longrightarrow\nu$.
\end{corollary}

\begin{proof}
 Combine Proposition~\ref{prop:ApproximComputableMeasure} with the second point of Theorem~\ref{MainTheorem}.
\end{proof}

\begin{corollary}\label{cor:ConvSetMeasure}
Let $\K\subset \Ms(\B^\Z)$ be a compact, $\Pi_2$-computable and connected ($\Pi_2$-CCC) subset of $\Ms(\B^\Z)$. There is
an alphabet $\A \supset \B$ and a cellular automaton $F:\az\to\az$ such that for any $\mu \in \Mmixfull(\az)$, one has
$\V(F,\mu) = \K$.
\end{corollary}
This is in particular a full characterisation of limit measures and connected $\mu$-limit measures sets that are
reachable from some computable initial measure $\mu \in\Ms(\az)$.
\begin{proof}
Combine Proposition~\ref{prop:polygonalcover} with the first point of Theorem~\ref{MainTheorem}.
\end{proof}

For both of these statements, a rate of convergence is given in Theorem~\ref{MainTheorem}, and this rate depends partly of the quality of the approximation of the target measure or the target set by a uniformly computable sequence of computable measures supported by periodic orbits.

\begin{open}
 Can the rate of convergence be improved, or can we prove that this is the best possible rate?
\end{open}

The following corollary is the counterpart of Corollary~\ref{cor:ConvSetMeasure} using
Theorem~\ref{MainTheoremNoAux}. Corollary~\ref{cor:ConvUniqueMeasure} does not have a counterpart since its proof uses
the second point of Theorem~\ref{MainTheorem}.

\begin{definition}
A word $u\in\A^\ast$ is said to be \define{not charged} by a set $\K\in\Ms(\az)$ if for all $\nu\in\K$, $\nu([u]) = 0$.
\end{definition}

\begin{corollary}\label{cor:ConvSetMeasureNoAux}
Let $\K\subset \Ms(\B^\Z)$ be a non-empty $\Pi_2$-CCC subset of $\Ms(\B^\Z)$ that does not charge a word $u\in\B^\ast$.
Then there is a cellular automaton $F:\bz\to\bz$ such that for
any measure $\mu \in \Mmixfull(\B^\Z)$, $\V(F,\mu) = \K$.
In particular, any limit-computable measure which does not have full support can be obtained by this way.
\end{corollary}
\begin{proof}
Since $\K$ does not charge $u$, we can assume without loss of generality that no word in the uniformly computable sequence
$(w_n)_{n\in\N}$ associated to $\K$ by Proposition~\ref{prop:polygonalcover} contain $u$ as subword. Indeed, if not we replace by an other word, this transformation does not have influence on $\V((w_n)_{n\in\N})$ since $u$ does not have. Thus Theorem~\ref{MainTheoremNoAux} applies.
\end{proof}

We leave open in particular the case of limit measures with full support which can happen only if $F$ is surjective. For Corollaries~\ref{cor:ConvSetMeasureNoAux}
and \ref{cor:ConvCesaroMeasureNoAux},
solving this case would imply to characterise the possible asymptotic behaviours of surjective automata. In this case a
similar construction seems difficult since the state $\wall$ which appear only in the initial configuration cannot be coded.

\begin{open}
 Which sets of measures can be reached at the limit by surjective cellular automata?
\end{open}

\subsubsection{Towards the non-connected case}
In Corollary~\ref{cor:ConvSetMeasure} the $\mu$-limit measures set is assumed to be connected. Indeed, in the
construction of Theorem~\ref{MainTheorem}, each word $w_n$ is copied
progressively on each segment, so that we reach the closure of an infinite polygonal path which
is connected. However, non-connected $\mu$-limit measures sets also have some obstructions. For example, if
$\V(F,\mu)$ is finite, we have the following proposition.

\begin{proposition}\label{prop:FiniteAdherenceValue}
 Let $F:\az\to\az$ be a cellular automaton and $\mu\in\Ms(\az)$ such that $\V(F,\mu)$ is finite. Then $\F$ induces a
cycle on $\V(F,\mu)$.
\end{proposition}
\begin{proof}
Let $d=\min\{\dm(\nu,\nu'):\nu,\nu'\in\V(F,\mu)\textrm{ with }\nu\ne\nu'\}>0$ and consider $\nu\in\V(F,\mu)$. It is
possible to extract a sequence $(n_i)_{i\in\N}$ such that $\dm(\F^{n_i}\mu,\nu)<\frac{d}{3}$
and $\dm(\F^{n_i+1}\mu,\nu)>\frac{2d}{3}$. Since $\dm(\F^n\mu,\V(F,\mu))\underset{n\to\infty}{\longrightarrow}0$ and $\Ball(\nu,\frac{d}{3})\cap\V(F,\mu)=\{\nu\}$, we have $\dm(\F^{n_i}\mu,\nu)\underset{i\to\infty}{\longrightarrow}0$. By continuity of $\F$, one has $\dm(\F^{n_i+1}\mu,\F\nu)\underset{i\to\infty}{\longrightarrow}0$.

One deduces that for all $\nu\in\V(F,\mu)$ there exists $\nu'\in\V(F,\mu)$ such that $\F\nu=\nu'$. So there is
$k\in\N$ such that $\V(F,\mu)=\{\nu_0,\dots,\nu_{k-1}\}$ and $\F\nu_i=\nu_{i+1}$ where the addition is modulo $k$.
\end{proof}

Furthermore, if $\mu$ is computable, then $\V(F,\mu)$ is $\Pi_2$-computable and every isolated point is $\Delta_2$-computable. In particular if $\V(F,\mu)$ is finite, every point is $\Delta_2$-computable. We exhibit some examples of more sophisticated behaviours based on the construction in Theorem~\ref{MainTheorem}. The
first one is a family of cellular automata where $\V(F,\mu)$ is a finite set of connected components mapped by some periodic CA, which is a
partial reciprocal of Proposition~\ref{prop:FiniteAdherenceValue}.
The second one is a family of cellular automata where $\V(F,\mu)$ has an infinite number of connected components.

\begin{example}[Finite set of connected components]\label{example:finite}
Suppose $\K=\{\nu_0,\dots,\nu_{k-1}\}\subset\Ms(\bz)$ is a finite set of $\s$-invariant limit-computable measures such
that $G\nu_i=\nu_{i+1}$ for some periodic cellular automaton $G:\bz\to\bz$ ($G^p = Id$ for some $p\in\N$). Then there is
an alphabet
$\A\supset\B$ and a cellular automaton $F:\az\to\az$ such that $\V(F,\mu)=\K$ for $\mu\in\Mergfull(\az)$. Indeed,
let $F$ be the cellular automaton satisfying $\F^t\mu \to \nu_0$ obtained by Theorem~\ref{MainTheorem}, and consider
the cellular automaton that applies $G$ on the main layer and applies the local rule of $F$ once every $k$ steps. Since $G$ is periodic, the sample produced by the process of copy stay near form $\{\nu_0,\dots,\nu_{k-1}\}$.

The same idea holds if $\K$ is a finite union of $\Pi_2$-CCC sets which are mapped by a periodic cellular automaton
$G:\bz\to\bz$.
\end{example}

\begin{example}[Infinite set of connected components]
We give a sketch of a modification of the construction of Theorem~\ref{MainTheorem} to obtain examples of cellular
automata where $\V(F,\mu)$ has an infinite number of connected components. This is the first such example to our
knowledge. The construction uses the firing squad cellular automaton introduced by Mazoyer \cite{Mazoyer-1996}
$F_{\texttt{FS}}:\B_{\texttt{FS}}^\Z\to \B_{\texttt{FS}}^\Z$, which has the following properties:

\begin{itemize}
 \item the alphabet contains 4 states $\left\{\fire,\block,\white,\order\right\}\subset\B_{\texttt{FS}}$;
 \item if $x_{[0,n]}=\order\white^{n-1}\block$ then $F^{2n}_{\texttt{FS}}(x)_{[0,n]}=\fire^{n+1}$;
 \item the state $\fire$ does not appear in $(F^t_{\texttt{FS}}(x)_{j})_{(t,j)\in[0,n]\times[0,2n-1]}$.
\end{itemize}

Consider a uniformly computable sequence $(\K_i)_{i\in\N}$ of disjoint $\Pi_2$-CCC subsets of $\Ms(\bz)$.
There is a uniformly computable sequence of words $(w_n)_{n\in\N}$ of $\B^\ast$ such that
$\V((w_n)_{n\in\N}=\bigcup_{i\in\N}\K_i$. Define $w'_n=w_n\times\white^{|w_n|}\in (\B\times \B_{\texttt{FS}})^\ast$
and consider the cellular automaton $F:\az\to\az$ given by Theorem~\ref{MainTheorem} which produces
$\V((w'_n)_{n\in\N})$, with $\A\supset\B\times\B_{\texttt{FS}}$. We modify $F$ to obtain $\widetilde{F}$ in the
following way.

\begin{itemize}
\item at time $T_n$, when the copy of $w_n$ is initiated, we initialise a counter on another layer to count the length
$k$ of the segment; \item at time $t=T_{n+1}-2k$, the state $\order$ appears on the left border of each segment (this is
a computable number and the time counter keeps track of current time);
\item All $\fire$ symbols are immediately transformed into $\white$ symbols.
\end{itemize}

This requires the segments to be shorter than $T_{n+1}-T_n$ cells, but the probability that $[0,l]$ belongs to such a
segment tends to 1 as time tends to infinity (Proposition~\ref{prop:AllAcceptable}). Furthermore, the state $\fire$
appears only at times $(T_n)_{n\in\N}$. Therefore, in those segments, $\widetilde{\F}\mu$ approximates the measure
$\meas{w_n}\times\meas{\lfire}$ at time $T_{n+1}$ and the measure $\meas{w_n}\times\meas{\lwhite}$ at time $T_{n+1}+1$.

For an initial measure $\mu\in\Mmixfull(\az)$, one has $\V(\widetilde{F},\mu)=
\left(\bigcup_i\K_i\right)\times\meas{\lfire}\cup\K'$ for some $\K'\subset\Ms\left(\B\times
\left(\B_{\texttt{FS}}\setminus\{\fire\}\right)^\Z\right)$. In particular, $\V(\widetilde{F},\mu)$ has an infinite
number of connected components.
\end{example}

\begin{open}
Is it possible to characterise all compact subsets of $\Ms(\az)$ that can be reached as $\mu$-limit measures set of some
cellular automaton when $\mu$ is computable?
\end{open}

\subsection{Convergence in Cesàro mean}

In this section, by adapting the enumeration $(w_n)_{n\in\N}$, we obtain similar results on $\V'(F,\mu)$, the set
of limit points for the Cesàro mean sequence. It is easy to prove that $\V'(F,\mu)$ is nonempty, connected and
included in the convex hull of $\V(F,\mu)$.

\begin{corollary}\label{cor:ConvCesàroMeasure}
Let $\B$ be a finite alphabet and $\K'\subset \Ms(\B^\Z)$ a $\Pi_2$-CCC set. There exist an alphabet $\A \supset \B$,
and a cellular automaton $F:\az\to\az$ such that for any $\mu \in \Mmixfull(\az)$, one has $\V'(F,\mu) = \K'$.
\end{corollary}

This a full characterisation of sets that can be reached from some initial measure $\mu \in\Mcomp(\az)$ as
$\mu$-limit measures set in Cesàro mean. This corollary is a consequence of the following stronger result, where we
have control over both $\V(F,\mu)$ and $\V'(F,\mu)$.

\begin{corollary}\label{cor:ConvCesàroMeasureBis}
Let $\B$ be a finite alphabet and $\K'\subset\K\subset \Ms(\B^\Z)$ two $\Pi_2$-CCC sets. There exist an alphabet $\A
\supset \B$ and a cellular automaton $F:\az\to\az$ such that for any $\mu \in \Mmixfull(\az)$, one has:
\begin{itemize}
 \item $\V(F,\mu) = \K$;
 \item $\V'(F,\mu) = \K'$.
\end{itemize}
\end{corollary}
This is a full characterisation of pairs of connected subsets $(\K,\K')$ such that $\K'\subset\K$ that can be
reached from some initial measure $\mu \in\Mcomp(\az)$ in this way.

\begin{proof}We use notations from the proof of Proposition~\ref{prop:polygonalcover}. Notably
$(w_n)_{n\in\N}$ and $(w'_n)_{n\in\N}$ are the uniformly computable sequences of words associated
to $\K$ and $\K'$, respectively, and $\Vb_k$ and $\Vb_k^t$ are defined with regard to $\K$.\bigskip

We define a new sequence of words $(w''_n)_{n\in\N}$ in the following manner, using a similar method as
Proposition~\ref{prop:polygonalcover}. For $n\in\N$, let $i_n\leq n$ be the maximal value such that one can find a path
$w_n=u_0, u_1, \dots, u_l = w'_n, u_{l+1},\dots, u_{l'} = w_{n+1}$ with $u_1, \dots u_{l-1}, u_{l+1},\dots, u_{l'} \in
V_{i_n}^t$ and $\dm(u_k,u_{k+1})\leq 4b(i_n)$ for all $k\in[0,l']$.

Let $P_n:\{0,\dots,p_n\}\to\Vb_{i_n}^t$ be such a path. Since $\Vb_{i_n}^t \subset \A^{\leq i_n+1}$, this
path is of length $p_n\leq 2|A|^{i_n+1} \leq 2|A|^{|w_n|+1}<2|A|^{n+1}$.\bigskip

For $i \in [|\A|^{n^2},|\A|^{(n+1)^2}]$, we define:
\vspace{-0.1cm}
\begin{itemize}
 \item[-] if $i<|\A|^{n^2}+p_n,\ w''_i= P_n(i-|\A|^{n^2})$;
 \item[-] otherwise,\ $w''_i = w'_n$.
\end{itemize}

and let $F$ be the CA associated to $(w''_n)_{n\in\N}$ by Theorem~\ref{MainTheorem}. Since all elements of $(w_n)_{n\in\N}$ appear, we can prove as
in Proposition~\ref{prop:polygonalcover} that $\V(F,\mu) = \V((w''_n)_{n\in\N}) = \K$.\bigskip

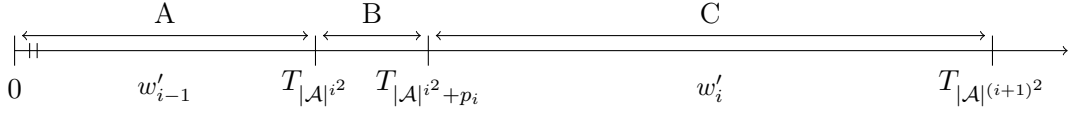
\begin{figure}[!ht]
 \begin{center}
  \begin{tikzpicture}
   \draw[->] (0,0) -- (14,0);
   \foreach \x\y in {0/0,4/T_{|\A|^{i^2}},5.5/T_{|\A|^{i^2}+p_i},13/T_{|\A|^{(i+1)^2}}}
   {
   \draw (\x,-0.2) -- (\x,0.2);
   \draw (\x,-0.5) node {$\y$};
   }
   \draw (0.3,-0.1) -- (0.3,0.1);
   \draw (0.2,-0.1) -- (0.2,0.1);
   \draw[<->] (0.1,0.2) -- (3.9,0.2);
   \draw[<->] (4.1,0.2) -- (5.4,0.2);
   \draw[<->] (5.6,0.2) -- (12.9,0.2);
   \draw (2,0.5) node {A} (2,-0.5) node {\small $w'_{i-1}$} (4.75,0.5) node {B} (9.25,0.5) node {C} (9.25,-0.5) node
{\small $w'_i$};
  \end{tikzpicture}
 \end{center}
\caption{Intuitively, we prove $A+B\ll C$, then $B\ll A$.}
\end{figure}

We have \[\frac {|\A|^{n^2} +p_n}{|\A|^{(n+1)^2} - (|\A|^{n^2}+p_n)} < \frac{|\A|^{n^2+1}}{|\A|^{(n+1)^2} -
|\A|^{n^2+1}} \underset{n\to\infty}\longrightarrow 0.\]
In other words, the subset $[0,|\A|^{n^2} +p_n]$ is (asymptotically) of negligible density in $[0,|\A|^{(n+1)^2}]$.
Since $T_{i+1}-T_i=q^{\lfloor\sqrt{i}\rfloor}$ (where $q$ is defined in Section~\ref{section:computation}) is an increasing sequence, 
the subset $[0,T_{|\A|^{n^2} +p_n}]$ is of
negligible density in $[0,T_{|\A|^{(n+1)^2}}]$. This means that,
putting $t_{n+1} = T_{|\A|^{(n+1)^2}}$, $d(\varphi_{t_{n+1}}^F(\mu),
\meas{w'_{n+1}}) \underset{n\to\infty}\longrightarrow 0$.\bigskip

Furthermore, notice that for $x,y\in \R_+$, when $y\leq \sqrt{x}$, we have $\lfloor \sqrt{x+y}\rfloor \leq \lfloor
\sqrt{x}\rfloor +1$ and $\lfloor \sqrt{x-y}\rfloor \geq \lfloor \sqrt{x}\rfloor -1$. Thus :
\[T_{|\A|^{n^2}+p_n}-T_{|\A|^{n^2}} < q^{|\A|^{\frac {n^2}2}+1}\cdot2|\A|^{n+1}.\]
\[T_{|\A|^{n^2}+p_n} > T_{|\A|^{n^2}} - T_{|\A|^{n^2}-|\A|^{\frac{n^2}2}} > q^{|\A|^{\frac
{n^2}2}-1}\cdot|\A|^{\frac{n^2}2},\]
and therefore
\[\frac{T_{|\A|^{n^2} +p_n} - T_{|\A|^{n^2}}} {T_{|\A|^{n^2} +p_n}}
\underset{n\to\infty}\longrightarrow 0.\]
This means that, putting $t'_n = T_{|\A|^{n^2} +p_n}$, $d(\varphi_{t'_n}^F(\mu),
\meas{w'_n}) \underset{n\to\infty}\longrightarrow 0$.\bigskip

To sum up, we have two sequences of times $t_0<t'_0<\dots<t_n<t'_n<\dots$ such that, for all $n\in\N$, the Cesàro mean sequence
$(\varphi_t^F(\mu))_{t\in\N}$ is (asymptotically) close to $\meas{w'_n}$ between times $t_n$ and $t'_n$, and is close to
$\meas{w'_{n+1}}$ at time $t_{n+1}$. Furthermore, between times $t'_n$ and $t_{n+1}$, $\varphi_t^F(\mu)$ is by
definition a convex combination of $\varphi_{t'_n}^F(\mu)$ and $\meas{w'_{n+1}}$, and thus it is close to the segment
$[\meas{w'_n},\meas{w'_{n+1}}]$. We conclude that asymptotically, the sequence is close to $\V((w'_n)_{n\in\N})$, and
thus its set of limit points is $\K'$.
\end{proof}

\begin{open}
 Is it possible to extend Corollary~\ref{cor:ConvCesàroMeasureBis} when $\K'$ is not included in $\K$?
\end{open}

Using Example~\ref{example:finite} we can only provide some examples where $\V(F,\mu)\cap\V'(F,\mu)=\emptyset$. 

This result has a counterpart with no auxiliary states, using Theorem~\ref{MainTheoremNoAux}.

\begin{corollary}\label{cor:ConvCesaroMeasureNoAux}
Let $\K'\subset\K\subset \Ms(\B^\Z)$ two nonempty $\Pi_2$-CCC sets that both do not charge the same word $u\in\B^\ast$.
Then there exists a cellular automaton $F:\B\to\B$ such
that for any $\mu \in \Mmixfull(\az)$,
\begin{itemize}
 \item $\V(F,\mu) = \K$;
 \item $\V'(F,\mu) = \K'$.
\end{itemize}
\end{corollary}

\subsection{Undecidability consequences}

We give an undecidability result extending a result of Delacourt on $\mu$-limit sets \cite{Delacourt-2011}.

\begin{corollary}[Rice theorem on $\mu$-limit measures sets]\label{cor:Rice}
Let $P$ be a nontrivial property on non-empty $\Pi_2$-CCC sets of $\Ms(\B^\Z)$ (i.e. not always or never true). There
is no algorithm that can decide, given an alphabet $\A$ and a CA $F:\az\to\az$, whether $\V(F,\mu)$ satisfies $P$ for
$\mu\in\Mmixfull(\bz)$.\end{corollary}

\begin{proof} We proceed by reduction to the halting problem.
Since $P$ is nontrivial, let $\K_1$ and $\K_2$ be two $\Pi_2$-CCC sets that satisfies and does not satisfy $P$,
respectively.
By Proposition~\ref{prop:polygonalcover}, there exists two uniformly computable sequences of words $(w_n)_{n\in\N},(w'_n)_{n\in\N}
\in (\A^{\ast})^\N$ such that $\K_1 =
\V((w_n)_{n\in\N}), \K_2= \V((w'_n)_{n\in\N})$.

Now let $\mathcal{TM}$ be a Turing machine. Define the sequence $(w''_n)_{n\in\N}$ in the following way.
\begin{itemize}
 \item If $\mathcal{TM}$ halts on the empty input in less than $n$ steps, $w''_n = w_n$.
 \item Otherwise, $w''_n = w'_n$.
\end{itemize}
This sequence is computable by simulating $n$ steps of the Turing machine and outputting the corresponding word.
Therefore, we can use the previous construction to
build a CA $F$ such that $\V(F,\mu)= \V((w''_n)_{n\in\N})$. If $\mathcal{TM}$ halts on the empty input, then $w''_n
= w_n$ for $n$ large enough; otherwise,
$w''_n = w'_n$ for $n$ large enough. Thus, $\V(F,\mu)$ satisfies $P$ if and only if $\mathcal{TM}$ halts.
\end{proof}

In Corollary~\ref{cor:Rice} the alphabet is considered as an input of the problem. A similar result with a fixed
alphabet requires to use the construction of Section~\ref{section:NoAux}.

\begin{corollary}\label{cor:RiceNoAux}
Let $\B$ be an alphabet, $\mu \in \Mmixfull(\B^\Z)$, $u\in\B^\ast$, and $P$ be a nontrivial property
on non-empty $\Pi_2$-CCC sets that do not charge $u$. There is no algorithm that can decide, given a CA $F:\bz\to\bz$,
whether $\V(F,\mu)$ satisfies $P$.
\end{corollary}

A direct extension would not be possible. If $\lambda$ is the uniform Bernoulli measure, the problem of whether
$\V(F,\lambda)$ contains only the uniform Bernoulli measure is equivalent to the surjectivity of $F$, which is
decidable \cite{Amoroso-1972}. In the non-probabilistic setting, the only decidable property about the asymptotic behaviour of CA is surjectivity \cite{Guillon-2010}. However, the question of which nontrivial properties on limit measures and $\mu$-limit
measures sets with full support are decidable remains open.

Corollaries~\ref{cor:Rice} and \ref{cor:RiceNoAux} extend naturally for a single limit and the Ces\`aro mean sequence.

\subsection{Computation on the set of measures}\label{section:Oracle}

In this section, we modify the construction to
perform computation on the space of probability measures, that is, we want the $\mu$-limit measures set
to be a function of the initial measure; this requires to keep some information in the construction. When the initial measure is not computable, 
we can use this information as a ``source'' of noncomputability to reach $\mu$-limit
measures sets that would be unreachable otherwise.

\subsubsection{Computation with oracle}

It is possible to construct an arithmetical hierarchy for computability of measures. More precisely, a measure is \define{$\Delta_1$-computable} if it is computable and a measure $\mu$ is \define{$\Delta_n$-computable} for $n\geq2$ if there exists a computable function $f:\A^{\ast}\times\N^{n-1}\to\Q$ such that 
\[\mu([u])=\lim_{i_{n-1}\to\infty}\lim_{i_{n-2}\to\infty}\dots\lim_{i_1\to\infty}f(u,i_1,\dots,i_{n-1})\qquad\forall u\in\A^{\ast}.\]

Similarly as Proposition~\ref{prop:OneStep}, if $\mu$ is $\Delta_n$-computable and $\F^t\mu\underset{t\to\infty}{\longrightarrow}\nu$ then $\nu$ is $\Delta_{n+1}$-computable. In the same way, one can introduce naturally the notion of $\Pi_n$-computable closed set. If $\mu$ is $\Delta_n$-computable then $\V(F,\mu)$ is $\Pi_{n+1}$-computable.

In fact, the obstructions shown in Section~\ref{section:ComputableObstruction} can be generalised to obstructions on $\mu\longmapsto \V(F,\mu)$, including cases where the initial measure is not necessarily computable, by considering
computability with access to an oracle $\mu\in\Ms(\az)$. In all the following, we fix a subset $\mathcal{M}\subset\Ms(\az)$.

A \define{Turing machine with oracle} in $\mathcal M$ has the same behaviour as a classical
Turing machine, except that an oracle $\mu\in\mathcal M$ is fixed prior to computation. The machine can query the oracle
at any time during the
computation by writing $u\in\A^\ast$ and $n\in\N$ on an special additional \define{oracle tape} and entering a special
\define{oracle state}.
At this step, the content of the oracle tape is considered as the oracle input and, after one step, the contents of
the oracle tape are replaced by an approximation of $\mu([u])$ up to an error $2^{-n}$ and the computation resumes.

Let $X,Y$ two countable sets. A function $f:\mathcal M\times X\to Y$ is
\define{computable with oracles in $\mathcal M$} if there exists a Turing machine with oracle in $\mathcal M$ which
takes as input $x\in X$ and returns $y=f(\mu,x)\in Y$, up to reasonable encoding.

\begin{definition}
A function $\varphi:\mathcal M\longrightarrow\Ms(\bz)$ is \define{computable with
oracles in $\mathcal M$} if there exists a computable function with oracles in $\mathcal M$ $f:\mathcal M\times
\N\longrightarrow\B^{\ast}$ such that $|\varphi(\mu)-\meas{f(\mu,n)}|\leq 2^{-n}$. This is an extension of the previous
definition where the image is not countable, hence the abuse of notation.

A sequence of functions $(f_n:\mathcal M\times\Ms(\az)\longrightarrow\R)_{n\in\N}$ is \define{a uniformly computable sequence of computable functions
with oracles in $\mathcal M$} if:
\begin{itemize}\itemsep0em
\item there exists $a:\mathcal M\times\N\times\N\times\A^{\ast}\longrightarrow\Q$ computable with oracles in $\mathcal M$ such that
\[\left|f_n(\mu,\meas{w})-a(\mu,n,m,w)\right|\leq\frac{1}{m}\textrm{ for all }\mu\in\mathcal M, w\in\A^{\ast}\textrm{ and }n,m\in\N;\]
\item there exists $b:\mathcal M\times\N\longrightarrow\Q$ computable with oracles in $\mathcal M$ such that $\dm(\nu,\nu')<b(\mu,m)$
implies $\left|f_n(\mu,\nu)-f_n(\mu,\nu')\right|\leq\frac{1}{m}$ for all $\mu\in\mathcal M$ and $n,m\in\N$.
\end{itemize}

Let $\Kset$ be the set of compact subsets of $\Ms(\bz)$. Defining the computability of a function
$\Psi:\mathcal M\longrightarrow\Kset$ can be done in various ways, similarly as in Proposition~\ref{prop:EquivComput}.
For example, $\Psi$ is \define{$\Pi_2$-computable} if the distance function $\mu, \nu\longmapsto d_{\Psi(\mu)}(\nu)$ is
$\Sigma_2$-computable with oracles in $\mathcal M$.
\end{definition}

The proofs of Section~\ref{section:ComputableObstruction} can be adapted in this framework. For any cellular
automaton $F$ on $\az$:

\begin{itemize}
\item $\mu\longmapsto\F\mu$ is computable with oracles in $\Ms(\az)$ (equivalent to
Proposition~\ref{prop:OneStep});
\item $\mu\longmapsto\V(F,\mu)$ and $\mu\longmapsto\V'(F,\mu)$ are $\Pi_2$-computable with oracles in $\mathcal
Ms(\az)$ (equivalent to Proposition~\ref{prop:add=s2ccc});
\item if $\Psi:\mathcal M\longrightarrow\Kset$ is a $\Pi_2$-computable function with oracles in $\mathcal M$ and if
every element of $\Psi(\mathcal M)$ is connected, then there exists a computable function $f:\mathcal
M\times\N\longrightarrow\A^{\ast}$ with oracles in $\mathcal M$ such that $\Psi(\mu)=\V((f(\mu,n))_{n\in\N})$, where
$\V((f(\mu,n))_{n\in\N})$ is the closure of the limit points of the polygonal path (equivalent to
Proposition~\ref{prop:polygonalcover}).
\end{itemize}

\subsubsection{Towards a reciprocal}
In this section, we give a partial reciprocal to these obstructions. To use the initial measure $\mu\in\Ms(\az)$ as an
oracle, we need to keep some information from the initial configuration. We adapt the original construction in the
following way:\bigskip

Each segment keeps a sample of the initial configuration, using the frequency of patterns inside this sample as an
oracle in the computation. We need to ensure that the frequency of a pattern $u\in\A^k$ in this sample is close to
$\mu([u])$ with a high probability. For this, we use Theorem~III.1.7 of~\cite{Shields-1996} applied on a measure
$\mu\in\Mmixfull(\az)$ that ensures we have an exponential rate of convergence for every length. Formally, for
any $k,m,n\in\N$, $c>0$:
\[\mu\left(\left\{x\in\az : \max_{u\in\A^k}\{|\mu([u])-\freq(u,x_{[0,n]})|\}\geq
\varepsilon\right\}\right)\leq(k+m)\psi(m)^{\frac{n}{k}}\left(\frac{n}{k}+1\right)^{\card(A)^k}
2^{-\frac{nc\varepsilon^2}{4k}}.\]
However, in our construction, we are unable to keep all information from the initial configuration since the formatting
process destroys information in the segment. In all the following, we will only keep information about the density of
$\start$ symbols, and the reached $\mu$-limit set of measures depends on this parameter only. The same method could be
adapted to keep information about longer words, only considering the positions of $\start$ symbols.

\begin{theorem}\label{theorem:Computable}
Let $\Psi:\Mmixfull(\{0,1\}^\Z)\to\Kset$ be a $\Pi_2$-computable function where $\Kset$ is the set of compact
connected
subsets of $\Ms(\bz)$. Assume that if $\mu,\mu'\in\Mmixfull(\{0,1\}^\Z)$ are such that $\mu([1])=\mu'([1])$, we have
$\Psi(\mu)=\Psi(\mu')$.

Then there exists an alphabet $\A\supset \B$ and a cellular automaton $F:\az\to\az$ such that for all
$\mu\in\Mmixfull(\az)$, we have $\V(F,\mu)=\Psi(\pi\mu)$ where $\pi$ is the 1-block map defined by $\pi(x)_i = 1$
when $x_i = \start$, and $\pi(x)_i=0$ otherwise.\end{theorem}

Notice that since only one density is considered, it would be equivalent in this case to consider a $\Pi_2$-computable
function with oracles in $\R$ and define a function $\R\to\Kset$. We kept the statement more technical to be consistent with the general
case.
\begin{proof}
Let $f:\Mmixfull(\{0,1\}^\Z)\times\N\longrightarrow\A^{\ast}$ be a computable function with oracles in
$\Mmixfull(\{0,1\}^\Z)$ such that $\Psi(\mu)=\V((f(\mu,n))_{n\in\N})$ and consider the associated Turing machine with
oracle.

Let $F$ be the cellular automaton defined in Theorem~\ref{MainTheorem} that simulated the Turing machine corresponding to 
$((f(\mu,n))_{n\in\N})$. Of course we need to specify the behaviour of the automata when the machine performs an oracle query.

We add a new layer $\alph{oracle}$ in which each
segment at time $t$ stores the frequency of the state $\start$ in this segment at time 0. To do that, we modify
the construction in the following way:
\begin{itemize}
 \item We subdivide the layer $\alph{oracle}$ in two parts, on which each wall $\wall$ keeps on its left:
\begin{itemize}
 \item the first counter for the number of $\start$ symbols that have been destroyed in its left segment;
 \item the second counter for the length of this segment, 0 if the segment is not formatted.
 \end{itemize}
\item Another counter accompanies each formatting counter, measuring the length of the segment as it progresses.
\item The second counter is initialised as $0$. When the time counter attached to this wall makes a comparison with an
initialised formatting counter (the comparison returns the result ``=''), the second counter stores the length of the
segment. It may take the value $0$ again if it merges with a non-formatted segment (see Figure~\ref{figure:OracleCounter}).
\item When a wall is destroyed by a merging process, it sends to its right an \define{oracle signal} at speed~$1$
containing the information stored in its oracle counters. Such a signal should not cross a formatting counter, so it is
slowed down if necessary.
\item When a wall's counters are $(c_1,c_2)$ and a signal $(c'_1,c'_2)$ comes from its left, there are three cases:
\begin{itemize}
 \item If $c_2=0$, the left segment cannot be formatted; the signal cannot come from an initialised wall and can be safely
ignored. The counters does not change.
 \item If $c_2\ne 0$, the left segment has been formatted and all false signals erased. Thus the information comes from an
initialised wall. The new number of $\start$ symbols is $c''_1=c_1+c'_1+1$ to take the merging into account.
 \begin{itemize}
 \item If $c'_2=0$, the segment just merged with a non-formatted segment and $c''_2=0$;
 \item otherwise $c''_2=c_2+c'_2$.
 \end{itemize}
 The counters take the values $(c''_1,c''_2)$.
\end{itemize}
See Figure~\ref{figure:OracleCounter}. We remark that if the length of the segment is $k$, the information can be coded
in space $\log(k)$, and it is possible to actualise the values before another signal can come from the left.

\begin{figure}[!ht]
 \begin{center}
  \begin{tikzpicture}
  \draw[->] (-.7,-.5) -- (-.7,5.5);
  \draw (-.7,5.8) node {time};
  \draw (-.8,3) -- (-.6,3);
  \draw (-1,3) node {$T_n$};
   \foreach \t in {0,6,13.5}
   {
   \draw[very thick] (\t,-.2) -- (\t,5.5);
   \draw (\t-.35,0) node {\footnotesize{$0,0$}};
   }
   \foreach \t in {3,4.5,12}
   {
   \draw[very thick] (\t,-0.2) -- (\t,3);
   \draw (\t-.35,0) node {\footnotesize{$0,0$}};
   }
   \foreach \a\c in {0/3,3/4.5,4.5/6,6/13.5,12/13.5}
   \draw[thick] (\a,0) -- (\c,0.7*\c-0.7*\a);
   \draw (2.65,2.3) node {\footnotesize{$0,k$}};
   \draw (4.15,1.25) node {\footnotesize{$0,n$}};
   \draw (5.65,1.25) node {\footnotesize{$0,n$}};
   \draw (13.15,1.25) node {\footnotesize{$0,n$}};
   \foreach \a\b\c in {3/6,4.5/6,12/13.5}
   \draw[thick,dotted] (\a,3) -- (\c,0.7*\c-0.7*\a+3);
   \draw (5.55,4.2) node {\footnotesize{$1,2n$}};
   \draw (5.3,5.3) node {\footnotesize{$2,2n+k$}};
   \draw (13.15,4.2) node {\footnotesize{$1,0$}};
   \draw[<->] (0,-.4) -- (3,-.4);
   \draw (1.5,-.6) node {$k$};
   \draw[<->] (4.5,-.4) -- (6,-.4);
   \draw (5.25,-.6) node {$n$};
  \end{tikzpicture}
 \end{center}
 \caption{Each wall has its counter displayed when its value changes. Slanted thick lines are formatting counters, dotted
lines are signals transmitting information.}
\label{figure:OracleCounter}
 \end{figure}
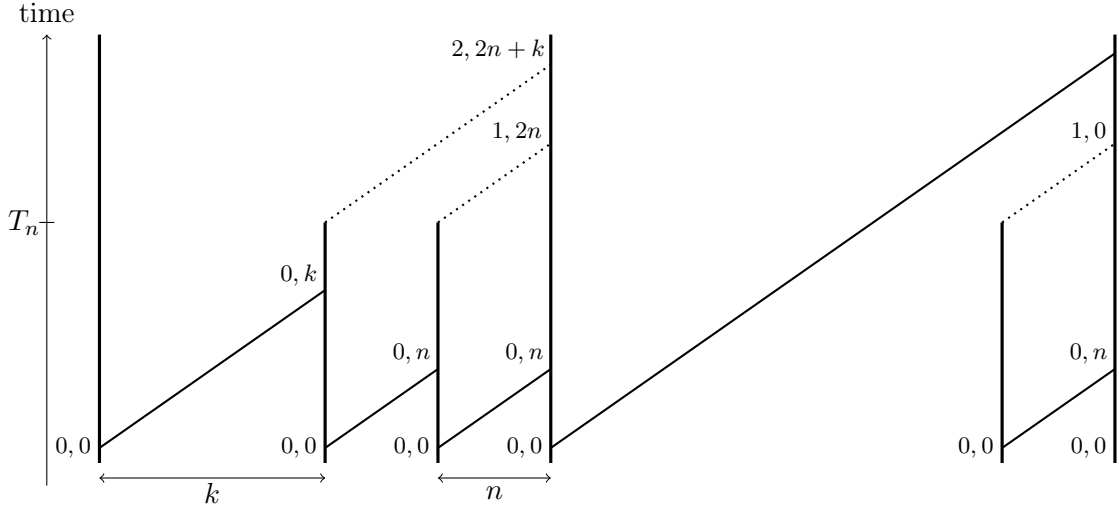

\item If two symbols $\start$ are too close in the initial configuration, they are destroyed by the bootstrapping
process (see Section~\ref{section:bootstrapping}). If a $\start$ is in a group of $\start$ separated by two cells or
less, the rightmost $\start$ sends a formatting counter and the leftmost one starts a time counter. Thus a group of
$\start$ separated by two cells or less behave as a single symbol for initialisation purposes. Each $\start$ symbol except
the leftmost one is transformed immediately into an oracle signal $(1,d)$, where $d$ is the distance to the nearest $\start$ to its left.
The other cells present initially are erased.
\item The Turing machine simulation described in Section~\ref{section:computation} can be adapted to simulate a Turing
machine with oracle. When there is an oracle query for the value of $\mu([\start])$ with precision $2^{-i}$ at time
$t\in [T_n,T_{n+1}]$, there are two possibilities:
\begin{itemize}
\item if $n^{-\frac{1}{6}}\leq 2^{-i}$, the Turing machine uses the information stored in the oracle layer to return the
frequency of $\start$ on the segment at time 0, and this corresponds to an approximation of $\mu([\start])$ with
sufficient precision;
\item if $n^{-\frac{1}{6}}  > 2^{-i}$, the computation stops, and the last word
successfully computed is output. The same thing happens until a time when enough information is available.
\end{itemize}
\end{itemize}

Let us check that $\V(F,\mu)=\Psi(\pi_\ast\mu)$ for $\mu\in\Mmixfull(\az)$. It is clear that the density of auxiliary
states tends to 0, so if the sample approximates correctly $\mu([\start])$, the sequence of words $(w_n)_{n\in\N}$
produced by the cellular automaton correspond to $(f(\mu,n))_{n\in\N}$ up to some repetition. Thus we only need to prove
that the probability that a cell belongs to a segment whose sample corresponds to a ``bad'' approximation tends to $0$
when $t$ tends to $\infty$. Recall that $\Gamma^{T_n}_{[i,j]} = \{x\in\az\ |\ [i,j]$ is a segment at time $T_n\}$.

\begin{align*}
 B_n&\phantom{=}=&&\mu\left(\left\{x\in\az : x_0\textrm{ belongs in a segment
with a ``bad'' sample at time }T_n \right\}\right)\\
&\phantom{=}=&&\sum_{i<0,j>0}\mu\left(\left\{x\in \Gamma^{T_n}_{[i,j]} :
|\mu([\start]) - Freq(\start,x_{[i,j]})|>n^{-\frac 16}\right\}\right)\\
&\phantom{=}=&&\sum_{k>0}k\cdot\mu\left(\left\{x\in\Gamma^{T_n}_{[0,k]} :
|\mu([\start]) - Freq(\start,x_{[0,k]})|>n^{-\frac 16}\right\}\right),
\end{align*}
by $\s$-invariance. By restricting ourselves to $n\leq k\leq K_n$, and for any $m\in\N$ large enough that
$\psi_\mu(m)<1$:
\begin{align*}
 B_n&\phantom{=}\leq&&\mu\left(\Gamma^{T_n}_{0,\geq
K_n}\right)+\sum_{k=n}^{K_n}k\cdot\mu\left(\left\{x\in\az : |\mu([\start]) -
Freq(\start,x_{[0,k]})|>n^{-\frac
16}\right\}\right)\\
 &\phantom{=}\leq&& \mu\left(\Gamma^{T_n}_{0,\geq K_n}\right)+
K_n^2(1+m)\psi_\mu(m)^n\left(n+1\right)^{\card(A)}2^{-\frac{c}{4}n^{\frac{2}{3}}}\\
 &\underset{n\to\infty}{\longrightarrow}&&0.
 \end{align*}

The result follows.
\end{proof}

This result may seem surprising since the same cellular automaton has very different asymptotic behaviours depending
on the initial measure.
\begin{open}
Can Theorem~\ref{theorem:Computable} be extended to characterise functions
$\Psi:\Mmixfull(\{0,1\}^\Z)\to\Kset$ that are realisable as the action of a cellular automaton $F$ in the sense that
for all $\mu$, $\V(F,\mu) = \Psi(\mu)$?
\end{open}

\subsection*{Acknowledgements}
The authors want to thank the anonymous referee for his rigorous and detailed 
review which helped us to clarify the paper. 
Moreover, this work was partially supported by the ANR project QuasiCool (ANR-12-JS02-011-01) and the ANR project Valet (ANR-13-JS01-0010).

%\noindent{\footnotesize 2010 \em{Mathematics Subjects Classification.} 37B10, 37B15,03D10,03D80\\
%\em{Keywords.} Symbolic Dynamics, Cellular automata, SRB measure,Turing machines}

\end{document}